\documentclass{article}
\usepackage[paper=a4paper,left=30mm,right=30mm,top=25mm,bottom=25mm]{geometry}

\usepackage{lmodern}
\usepackage{dsfont}
\usepackage{pifont}
\usepackage{amsmath} 
\usepackage{amsthm} 
\usepackage{amssymb} 
\usepackage{booktabs}
\usepackage{mathrsfs} 
\usepackage{mathtools} 
\usepackage{scalerel}  
\usepackage{stackengine}  
\usepackage{mathtools}
\usepackage{enumerate}
\usepackage{color} 
\usepackage{hyperref}

\newtheorem{theorem}{Theorem}[section]

\newtheorem{lemma}[theorem]{Lemma}

\newtheorem{proposition}[theorem]{Proposition}

\newtheorem{definition}[theorem]{Definition}
\newtheorem{remark}[theorem]{Remark}

\DeclareMathOperator{\asin}{asin}  

\DeclareMathSymbol{,}{\mathord}{letters}{"3B}  
\def\1{\mathds{1}} 
\def\abs#1{\left\lvert#1\right\rvert} 
\def\of#1{\left(#1\right)} 
\def\inbrackets#1{\of{#1}} 
\def\setof#1{\left\{#1\right\}} 

\makeatletter
\renewcommand\d[1]{\ensuremath{\;\mathrm{d}#1\@ifnextchar\d{\!}{}}} 
\newcommand\nopagebreakhere{\par\nobreak\@afterheading}  
\makeatother

\newcommand{\bsx}{\boldsymbol{x}}

\newcommand{\de}{\mathrm{\,d}}
\newcommand{\cmark}{\color{green}{\ding{51}}\color{black}}%
\newcommand{\xmark}{\color{red}{\ding{55}}\color{black}}%
\newcommand{\gnearrow}{\color{green}{\nearrow}\color{black}}%
\newcommand{\gsearrow}{\color{green}{\searrow}\color{black}}%
\newcommand{\gdownarrow}{\color{green}{\downarrow}\color{black}}%
\newcommand{\guparrow}{\color{green}{\uparrow}\color{black}}%
\newcommand{\qmark}{\textbf{?}}%

\newcommand{\eqd}{\stackrel{\mathrm{d}}=}

\newcommand{\E}{\mathbb{E}}

\newcommand{\R}{\mathbb{R}}

\newcommand{\cC}{\mathcal{C}}

\providecommand{\keywords}[1]{\textbf{Keywords } #1}

\newcommand{\TP}{\mathrm{TP_2}}

\newcommand{\conv}{{\scalebox{1.5}{\raisebox{-0.2ex}{$\ast$}}}}

\newcommand{\Ran}{\mathsf{Ran}}

\newcommand{\rank}{\mathsf{rank}}

\newcommand{\Var}{\mathrm{Var}}

\author{Jonathan Ansari$^1$ and Marcus Rockel$^{2}$}
\title{Dependence properties of bivariate copula families}

\begin{document}
\maketitle
\vspace{-5ex}

\begin{center}
    \small\textit{
		$\phantom{2}^1$Department of Artificial Intelligence and Human Interfaces,\\
		University of Salzburg,\\
		Hellbrunner Straße 34, 5020 Salzburg, Austria,\\
        jonathan.ansari@plus.ac.at \\[2mm]
	}
	\small\textit{
		$\phantom{2}^2$Department of Quantitative Finance,\\
		Institute for Economics, University of Freiburg,\\
		Rempartstr. 16,	79098 Freiburg, Germany,\\
        marcus.rockel@finance.uni-freiburg.de \\[2mm]
	}
\end{center}
\begin{abstract}{
    Motivated by recently investigated results on dependence measures and robust risk models, this paper provides an overview of dependence properties of many well-known bivariate copula families, where the focus is on the Schur order for conditional distributions, which has the fundamental property that minimal elements characterize independence and maximal elements characterize perfect directed dependence.
    We give conditions on copulas that imply the Schur ordering of the associated conditional distribution functions.
    For extreme-value copulas, we prove the equivalence of the lower orthant order, the Schur order for conditional distributions, and the pointwise order of the associated Pickands dependence functions.
    Further, we provide several tables and figures that list and illustrate various positive dependence and monotonicity properties of copula families, in particular, from classes of Archimedean, extreme-value, and elliptical copulas.
    Finally, for Chatterjee's rank correlation, which is consistent with the Schur order for conditional distributions, we give some new closed-form formulas in terms of the parameter of the underlying copula family.
 }
 \end{abstract}
 \vspace{0ex}
  \maketitle
\tableofcontents
\section{Introduction}
In recent years, there is an increasing number of scientific papers on dependence measures (a.k.a. measures of predictability or measures of regression dependence), i.e., on functionals \(\kappa\) of random vectors \((X,Y)\) satisfying the properties that \(\kappa(Y|X)\) only attains values in the interval \([0,1]\,,\) where the values \(0\) and \(1\) characterize independence and perfect directed dependence, respectively, meaning that \(\kappa(Y|X)=0\) if and only if \(X\) and \(Y\) are independent and \(\kappa(Y|X)=1\) if and only if there exists some Borel measurable (not necessarily increasing or decreasing) function \(f\) such that \(Y=f(X)\), see, e.g., \cite{dette2013copula,fuchs2021quantifying,Gamboa-2022,fgwt2021,Junker-2021,Nies-2021,Shih-2021,Sungur-2005,Trutschnig-2011,Wiesel-2022}.
The certainly most famous such measure is Chatterjee's rank correlation
\begin{align}\label{eqchatt}
    \xi(Y|X) := \frac{\int_\R \Var(P(Y\geq y\mid X)) \de P^Y(y) }{\int_\R \Var(\1_{\{Y\geq y\}}) \de P^Y(y)}\,,
\end{align}
which takes a simple form, has a fast estimator and allows interesting applications, such as a model-free, dependence based forward feature selection, see \cite{Ansari-Fuchs-2022,deb2020b,chatterjee2020,chatterjee2021}.
A large class of measures of predictability \(\kappa\) is based on an ordering \(\prec\) that satisfies the axioms
\begin{enumerate}[({A}1)]
    \item \label{ax1} Characterization of independence: \((Y|X)\prec (Y'|X')\) for all \((X',Y')\) with \(Y'\stackrel{\mathrm{d}}= Y\) if and only if \(X\) and \(Y\) are independent,
    \item \label{ax2} Characterization of perfect directed dependence: \((Y'|X')\prec (Y|X)\) for all \((X',Y')\) with \(Y'\stackrel{\mathrm{d}}= Y\) if and only if \(Y\) is perfectly dependent on \(X\,,\)
    \item \label{ax3} Consistency with \(\kappa\,:\) \((Y|X)\prec (Y'|X')\) implies \(\kappa(Y|X)\leq \kappa(Y'|X')\,,\)
\end{enumerate}
where \(\eqd\) denotes equality in distribution.
An interesting such ordering, which reflects, in particular, the fundamental properties of Chatterjee's rank correlation, is the Schur order for conditional distributions in \eqref{defschurconddist}, see \cite{Ansari-Fuchs-2022,Ansari-Fuchs-2023}.
It is a rearrangement-invariant order that compares the variability of conditional distribution functions in the conditioning variable, where small variability means low predictability and large variability corresponds to a high determination of \(Y\) given \(X\,.\)
Some related global dependence stochastic orders based on the variability of conditional expectations and conditional variances are studied in \cite{Shaked-2013}. For similar dependence orders which are, however, not rearrangement-invariant in the conditioning variable, we refer to \cite{Averous-2000,Colangelo-2008,dette2013copula,Hollander-1990,Joe-2018,Yanagimoto-1969}.\\
The Schur order for conditional distributions and so Chatterjee's rank correlation, which both can be extended to multivariate vectors of input variables, 
are merely rank-based concepts and depend in the case of continuous marginal distributions only on the underlying copula.
More precisely, they are
fully described by stochastically increasing\footnote{The concept of a 'stochastically increasing' bivariate copula \(C\) often used in the literature is more accurately denoted as 'conditionally increasing in sequence' since for \((U,V)\sim C\,,\) the conditional distribution \(V|U=u\) is stochastically increasing in \(u\,,\) which, however, is not a symmetric concept, see Definition \ref{defposdep}\eqref{defposdep2} and \cite{Mueller-Stoyan-2002}.} 
bivariate copulas, for which a pointwise comparison is equivalent to the comparison in the sense of the Schur order, see \cite[Proposition 3.4]{Ansari-Fuchs-2022}. \\
The Schur order for conditional distributions also applies to recently studied comparison results for \(\conv\)-products of several bivariate copulas modeling the dependence structure of conditionally independent factor models.
Comparison results for large classes of such models with respect to the strong notion of supermodular order are given in \cite{Ansari-Rueschendorf-2022} allowing applications in risk analysis when some structural assumptions on the underlying distribution are imposed.
Risk bounds for these models are specified by a set of marginal distributions and a set of stochastically increasing or \(\TP\)-copulas, which both are concepts of positive dependence. \\
Motivated by the above-mentioned applications, in this paper we investigate positive dependence and ordering properties for members of various well-known bivariate copula families with the aim to provide a concise overview of their dependence properties. More specifically, we determine for copulas, in particular from classes of Archimedean, extreme-value and elliptical distributions, whether they are conditionally increasing/decreasing, $\TP$ or neither.
For this, we either cite references, verify well-known characterizations from the literature (e.g., \cite{Durante-2016,Mueller-2005,Nelsen-2006}) or use direct calculations.
Further, we determine whether the copulas are increasing or decreasing in their parameter with respect to the lower orthant order and Schur order for conditional distributions.
For classes of extreme-value distributions, we prove that the latter orderings are equivalent and can also be characterized by the pointwise ordering of the associated Pickands dependence functions.
While measures of concordance such as Kendall's tau and Spearman's rho are consistent with respect to the lower orthant ordering of copulas, various measures of predictability such as Chatterjee's rank correlation are consistent with respect to the Schur order for conditional distributions, see \cite{Ansari-Fuchs-2023}.
For the aforementioned three measures of association, we illustrate their behavior in dependence on the copula family parameters in several plots and also provide some new closed-form formulas. \\
The remainder of this paper is organized as follows: Section \ref{sec2} provides the necessary tools for analyzing bivariate dependencies.
In Section \ref{secdeppropbivcopfam}, we focus on ordering results with respect to the Schur order for conditional distributions and provide several tables and figures which give a concise overview of the dependence properties of more than \(35\) bivariate copula families.
The often tedious calculations are all deferred to the appendix.

\section{Basic concepts of dependence modeling}\label{sec2}

In this section, we provide the main tools for modeling bivariate dependence structures.
First, we give the definition of a copula and consider the well-known classes of Archimedean, extreme-value and elliptical copulas.
Then, we introduce the stochastic orderings and dependence concepts which we make use of.
Finally, we consider some specific measures of association.
For multivariate extensions of all these concepts, we refer to the literature on dependence modeling, see, e.g., \cite{Durante-2016,Nelsen-2006}.

\subsection{Copulas}

A bivariate copula is a function \(C\colon [0,1]^2\to [0,1]\) that is grounded, \(2\)-increasing and that has uniform marginals, i.e., \(C(u)=0\) for \(u=(u_1,u_2)\) whenever \(u_1=0\) or \(u_2=0\,,\) \(C(u_1,u_2)+C(v_1,v_2)-C(u_1,v_2)-C(v_1,u_2)\geq 0\) for all $u_1,u_2,v_1,v_2\in [0,1]$ with \(u_1\leq v_1\) and \(u_2\leq v_2\,,\) and \(C(u_1,u_2)=u_i\) for all $(u_1,u_2)\in [0,1]^2$ and \(i\in \{1,2\}\) whenever \(u_j=1\) for \(j\neq i\,.\)
The motivation to consider copulas comes from Sklar's theorem, see, e.g., \cite[Theorem 2.3.3]{Nelsen-2006}, which states that every bivariate distribution function \(F\colon \R^2\to [0,1]\) can be decomposed into its marginal distribution functions \(F_1\) and \(F_2\) and a copula \(C\) such that
\begin{align}\label{eqsklar}
    F(x)=C(F_1(x_1),F_2(x_2))\quad \text{for all } x=(x_1,x_2)\in \R^2\,.
\end{align}
The copula \(C\) is uniquely determined on \(\Ran(F_1)\times \Ran(F_2)\,,\) where \(\Ran(F_i)\) denotes the range of \(F_i\,.\)
Further, for any bivariate copula \(C\) and for all univariate distribution functions \(F_1\) and \(F_2\,,\) the right-hand side of \eqref{eqsklar} defines a bivariate distribution function.
Denote by \(\cC_2\) the class of bivariate copulas.
In the following, we consider some well-known subclasses of \(\cC_2\,.\)

\subsubsection{Archimedean copulas}
\label{subsubsec:Arch}

Let \(\varphi\colon [0,1]\to [0,\infty]\) be a continuous, strictly decreasing function such that \(\varphi(1)=0\,.\)
Define the pseudo-inverse \(\psi\colon [0,\infty]\to [0,1]\) by \(\psi(t) := \varphi^{-1}(t)\) if \(0\leq t \leq \varphi(0)\) and by \(\psi(t) := 0\) if \(\varphi(0)< t \leq \infty\,.\)
Then, the function \(C_\varphi\colon [0,1]^2\to [0,1]\) defined by
    \begin{align*}
        C_\varphi(u,v) = \psi(\varphi(u)+\varphi(v))
    \end{align*}
is a bivariate copula if and only if \(\varphi\) is convex, see, e.g., \cite[Theorem 4.1.4]{Nelsen-2006}.
For such convex \(\varphi\,,\) the copula \(C_\varphi\) is denoted as \emph{Archimedean copula} with \emph{generator} \(\varphi\,.\)
In Section \ref{secdeppropArch}, we provide various dependence properties of the Archimedean copula families given in \cite[Chapter 4]{Nelsen-2006}.

\subsubsection{Extreme-value copulas}
\label{subsubsec:EVC}

Let \(A\colon [0,1]\to [1/2,1]\) be a convex function that satisfies the constraints \(\max\{t,1-t\}\leq A(t)\leq 1\) for all \(t\in [0,1]\,.\)
Then, a bivariate copula \(C = C_A\) is an \emph{extreme-value copula} generated by \emph{Pickands dependence function} \(A\,,\) if
\begin{align}\label{defEVC}
    C_A(u,v)=\exp\left( \ln(uv) A\Big(\frac{\ln v}{\ln u + \ln v}\Big) \right) \quad\text{for all } (u,v)\in (0,1)^2\,,
\end{align}
see, e.g., \cite[Theorem 6.6.7]{Durante-2016}.
We study dependence properties of several extreme-value copula families in Section \ref{secdeppropEVC}.

\subsubsection{Elliptical copulas}
\label{subsubsec:elliptical}

A bivariate random vector \(X=(X_1,X_2)\) follows an elliptical distribution centered at \(\mu\in \R^2\) if the characteristic function of \(X-\mu\) is a function of a quadratic form, i.e., if \(\varphi_{X-\mu}(t) = \phi(t^T\Sigma t)\) for all \(t\in \R^2\) for some positive semi-definite matrix \(\Sigma\in \R^{2\times 2}\) and some \emph{characteristic generator} \(\phi\colon [0,\infty)\to [0,\infty)\,.\)
For given \(\phi\) and for \(\rho\in [-1,1]\,,\) setting w.l.o.g.\@ \(\mu=0\) and \(\Sigma = \left(\begin{smallmatrix}
    1 & \rho \\
    \rho & 1
\end{smallmatrix}\right)\,,\)
any copula implicitly obtained from Sklar's theorem by the distribution function of \(X\) via \eqref{eqsklar} is denoted as \emph{elliptical copula} with parameter \(\rho\) and generator \(\phi\,,\) see, e.g., \cite[Section 6.7]{Durante-2016}.

The random vector \(X\) admits a stochastic representation given by
\begin{align*}
    X = \mu + R A U
,\end{align*}
where \(R\) is a non-negative random variable, \(A\in \R^{2\times k}\) is a matrix such that \(\Sigma = AA^T\) for \(k=\rank(\Sigma)\,,\) and where \(U =(U_1,U_2)\) is a bivariate random vector that is independent of \(R\) and uniformly distributed on the \(2\)-sphere, i.e., on the unit circle \(\{(x,y)\mid x^2+y^2=1\}\,.\)
If \(X\) has a Lebesgue-density, \(\rho\in (-1,1)\,,\) and \(P(X=0)=0\,,\) then the \emph{radial part} \(R\) admits a Lebesgue-density \(g\,,\) see, e.g., \cite[Section 2.6]{Fang-1990}.
In Section \ref{secpropellcop}, we study properties of elliptical copulas with \emph{density generator} \(g\) and parameter \(\rho\,.\)

\subsection{Stochastic orderings}\label{secstoo}

For comparing dependencies, orderings on the set of copulas are useful.
In the first part of this section, we consider the lower orthant (i.e., the pointwise) ordering of copulas.
In the second part, we introduce to the recently studied rearrangement-based orderings of copulas.

\subsubsection{Orthant orderings}

The certainly most popular ordering on the set of bivariate copulas is the lower orthant order which is defined by the pointwise comparison of bivariate copulas as follows.

\begin{definition}[Lower orthant order]\label{defi:lower_orthant_order}~\\
    Let \(D\) and \(E\) be bivariate copulas.
    Then \(D\) is said to be smaller than \(E\) with respect to the \emph{lower orthant order}, written \(D\leq_{lo} E\,,\) if \(D(u,v)\leq E(u,v)\) for all \((u,v)\in [0,1]^2\,.\)
\end{definition}

The uniquely determined maximal and minimal elements in the class of bivariate copulas are the upper and lower Fr\'{e}chet copula \(M\) and \(W\,,\) respectively, defined by \(M(u,v):=\min\{u,v\}\) and \(W(u,v):=\max\{u+v-1,0\}\) for \((u,v)\in [0,1]^2\,.\)
The upper (lower) Fr\'{e}chet copula models comonotonicity (countermonotoncity), i.e., for random variables \(U\) and \(V,\) it holds \(C_{U,V}=M\) (\(C_{U,V}=W\)) if and only if \(U=V\) (\(U=1-V\)) almost surely, see, e.g., \cite[Examples 1.3.3 and 1.3.5]{Durante-2016}.
Given a bivariate copula $C$, its \emph{survival function} \(\overline{C}\colon [0,1]^2\to [0,1]\) is defined by
\begin{align}\label{eqsurvfun}
    \overline{C}(u,v):= 1 -u-v+C(u,v)\,,
\end{align}
see, e.g., \cite{Durante-2016}.
Furthermore, the upper orthant order on \(\cC_2\) is defined by the pointwise comparison of survival functions of bivariate copulas, i.e.,
\begin{align*}
    D \leq_{uo} E \quad \colon \Longleftrightarrow \quad \overline{D}(u,v) \leq \overline{E}(u,v) \quad \text{for all } (u,v)\in [0,1]^2\,.
\end{align*}
As an immediate consequence of \eqref{eqsurvfun}, for bivariate copulas the lower and upper orthant order are equivalent.
Hence, the in the literature frequently considered \emph{concordance order}, which is defined through the lower and upper orthant ordering of copulas, coincides for bivariate copulas with the lower orthant order. Note that in the three- and higher-dimensional setting, the lower and upper orthant order diverge, see \cite{Mueller-Stoyan-2002}.

\subsubsection{Orderings of predictability}

Due to axiom (A\ref{ax2}), an ordering of predictability should be invariant with respect to bijective transformations of the input variable \(X\,.\)
To this end, define for integrable functions \(f,\,g\colon (0,1)\to \R\) the \emph{Schur order} \(f\prec_S g\) by
\begin{align}\label{defschurfun}
\int_0^x f^*(t)\de t &\leq \int_0^x g^*(t)\de t ~~~\text{for all } x\in (0,1) \text{ and} ~~~
\int_0^1 f^*(t)\de t = \int_0^1 g^*(t)\de t\,,
\end{align}
where \(h^*\) denotes the decreasing rearrangement\footnote{Roughly speaking, the decreasing rearrangement of a (piecewise constant) function can be obtained by sorting the graph in descending order, compare Figure \ref{fig:schur_order}.} of an integrable function \(h\colon (0,1)\to \R\,,\) i.e., the essentially uniquely determined decreasing function \(h^*\) such that \(\lambda(h^*\geq w)=\lambda(h\geq w)\) for all \(w\in \R\,,\) where \(\lambda\) denotes the Lebesgue measure on \((0,1)\,,\) see, e.g., \cite{Ru-2013};
for an overview of rearrangements, see \cite{Chong-1974, Chong-1971, Day-1972, Hardy-1929, Luxemburg-1967, ruschendorf1981ordering}.
Minimal elements in the Schur order are constant functions while maximal elements do generally not exist. \\
Denote by \(q_X\) the quantile function of a real-valued random variable \(X\,,\) i.e., \(q_X(t):=\inf\{x\in \R\mid F_X(x)\geq t\}\,,\) \(t\in (0,1)\,.\)
Due to the following definition, conditional distribution functions are compared with respect to the conditioning variable in the Schur order, see \cite[Section 3.1]{Ansari-Fuchs-2022}.

\begin{definition}[Schur order for conditional distributions]~\\
    Let \((X_1,X_2)\) and \((Y_1,Y_2)\) be a bivariate random vector with \(X_2\eqd Y_2\,.\)
    Then \(X_2\) given \(X_1\) is said to be smaller than \(Y_2\) given \(Y_1\) in the \emph{Schur order for conditional distributions}, written \((X_2|X_1)\leq_S (Y_2|Y_1)\), if
    \begin{align}\label{defschurconddist}
        \E[\1_{\{X_2\leq y\}}\mid X_1 = q_{X_1}(\cdot)] \prec_S \E[\1_{\{Y_2\leq y\}}\mid Y_1 = q_{Y_1}(\cdot)] \quad \text{for all }y\in \R\,.
    \end{align}
\end{definition} 
Due to \eqref{defschurconddist}, the Schur order in the above definition compares the variability of conditional distribution functions in the conditioning variable with respect to the Schur order for functions, see Figure \ref{fig:schur_order}.
Since minimal elements of the Schur order for functions in \eqref{defschurfun} are constant functions, it follows that minimal elements with respect to the Schur order for conditional distributions are independent random vectors.
Similarly, for bounded functions, maximal elements in the Schur order for functions attain essentially two values given by the bounds.
It follows that maximal elements with respect to the Schur order for conditional distributions are perfectly directed dependent random vectors, see \cite[Theorem 3.5]{Ansari-Fuchs-2022}.
As shown in \cite{Ansari-Fuchs-2023}, the Schur order for conditional distributions satisfies the axioms (A\ref{ax1})--(A\ref{ax3}) for a large class of functionals \(\kappa\,.\)
In particular, it is invariant under bijective transformations of the conditioning variable.
An important property is that Chatterjee's rank correlation \(\xi\) defined in \eqref{eqchatt} is consistent with the Schur order for conditional distributions. \\
In the case of continuous marginal distributions, the copula \(C\) of a bivariate random vector \((X,Y)\) is uniquely determined and the conditional distribution function of \(Y|X=x\) can be represented as
\begin{align}\label{eqsklarcondcdf}
    F_{Y|X=x}(y) = \partial_1 C(F_X(x),F_Y(y))
\end{align}
for all \(y\in \R\) and for all \(x\in \R\) outside an \(F_X\)-null set which may depend on \(y\,,\) see, e.g., \cite[Theorem 2.2]{Ansari-Rueschendorf-2021}, where \(\partial_i\) denotes the partial derivative of a function of several arguments with respect to the \(i\)th component.
Hence, for \(U_i,V_i\) uniformly distributed in \((0,1)\,,\) \(i\in \{1,2\}\,,\) the Schur order for conditional distributions \((U_2|U_1)\leq_S (V_2|V_1)\) is equivalent to
\begin{align*}
    \partial_1 C_{U_1,U_2}(\cdot,v) \prec_S \partial_1 C_{V_1,V_2}(\cdot,v) \quad \text{for all } v\in [0,1]\,.
\end{align*}
This motivates to define a version of the Schur order for conditional distributions considering derivatives of bivariate copulas as follows, see \cite{Ansari-2019,Ansari-Rueschendorf-2021,strothmann2022rearranged}.

\begin{definition}[Schur order for copula derivatives]
    \label{defschurcopder}~\\
    Let \(D\) and \(E\) be bivariate copulas.
    Then \(D\) is said to be smaller than \(E\) in the \emph{Schur order for bivariate copula derivatives} with respect to the first (resp. second) component, written  \(D\leq_{\partial_1 S} E\) (resp. \(D\leq_{\partial_2 S}E\)) if
    \begin{align*}
        \partial_1 D(\cdot,v)\prec_S \partial_1 E(\cdot,v) \quad \text{(resp. } \partial_2 D(v,\cdot) \prec_S \partial_2 E(v,\cdot)\text{)} \quad \text{for all } v\in (0,1) \,.
    \end{align*}
    We write \(D\leq_{\partial S} E\) if \(D\leq_{\partial_1 S} E\) and \(D\leq_{\partial_2 S} E\,.\)
    Further, for \(i\in \{1,2\}\,,\) we write \(D=_{\partial_i S} E\) if \(D\leq_{\partial_i S} E\) and \(D\geq_{\partial_i S} E\,.\)
\end{definition}

\begin{figure}[htb]
    \centering
    \includegraphics[width=0.45\textwidth]{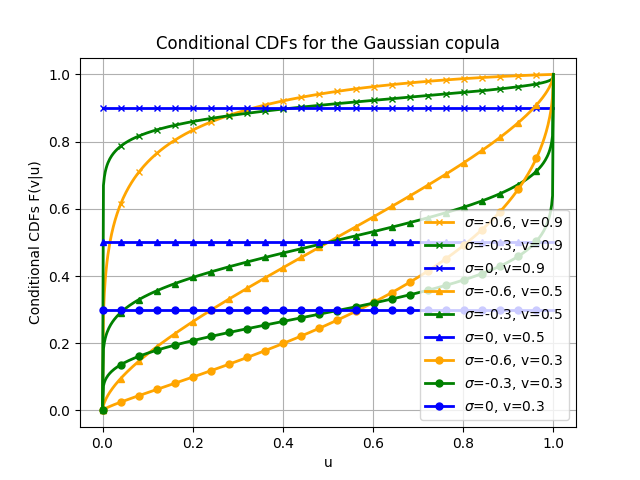}
    \includegraphics[width=0.45\textwidth]{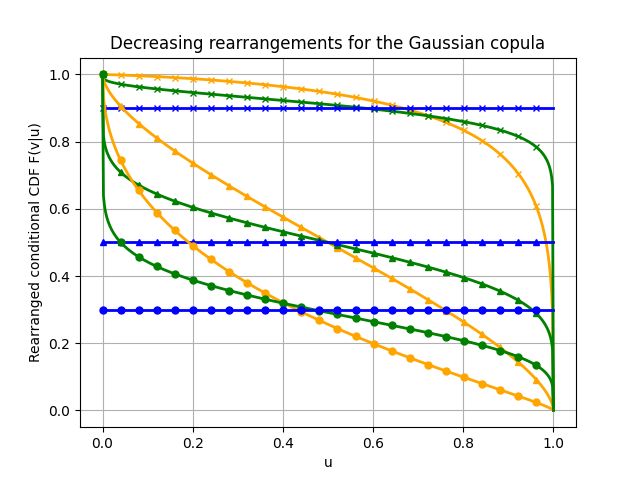} \\
    \includegraphics[width=0.45\textwidth]{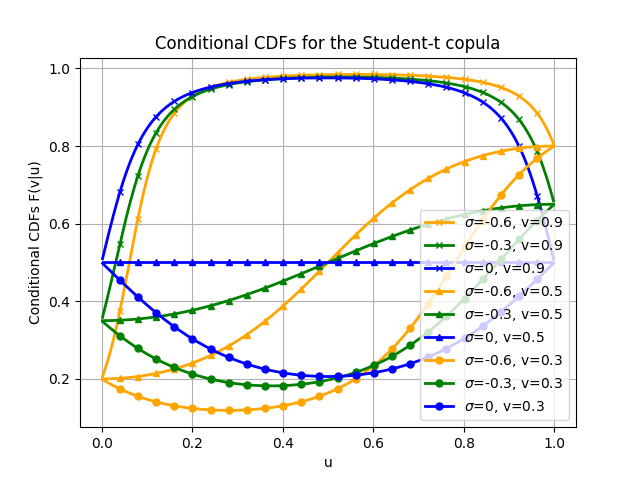}
    \includegraphics[width=0.45\textwidth]{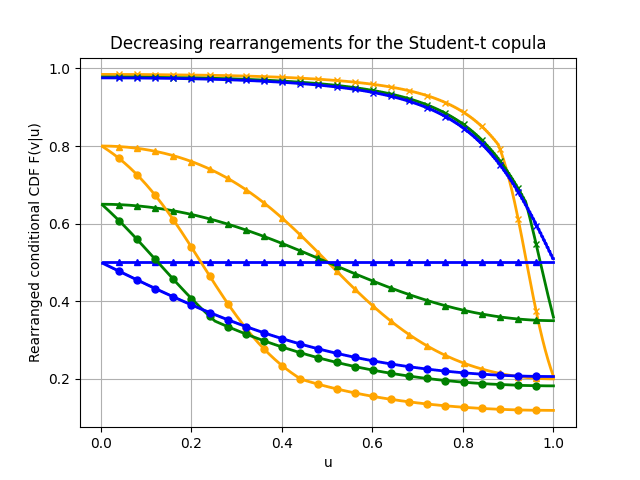}
    \caption{
        Variability of conditional distribution functions described by the decreasing rearrangements (right) of the copula derivatives \(\partial_1 C(\cdot,v)\) (left) in the case of the Gaussian copula (top) and the Student-t copula with one degree of freedom (bottom) for three choices of correlation parameters and for three choices of \(v\): \(\rho = -0.6\) (orange plots), \(\rho = -0.3\) (green plots) and \(\rho=0\) (blue plots), \(v=0.3\) (crossed plots), \(v= 0.6\) (triangulared plots), and $v=0.9$ (circled plots). 
        As the upper right plot indicates, the Gaussian copula family is increasing in the correlation parameter \(|\rho|\) with respect to the Schur order for conditional distributions. Further, it shows that the independence copula is the minimal element with respect to the Schur order for conditional distributions as a consequence of the definition of the Schur order for functions in \eqref{defschurfun}.
        The lower left plot also illustrates that Student-t copulas are not CI, see Definition \ref{defposdep} as well as Table \ref{tab:non_arch_results}. \\
    }
    \label{fig:schur_order}
\end{figure}

Since constant functions are minimal with respect to the Schur order for functions, the independence copula \((u,v)\mapsto \Pi(u,v):= uv\,,\) \((u,v)\in [0,1]^2\,,\) is the uniquely determined minimal element with respect to \(\leq_{\partial_1 S}\,,\) \(\leq_{\partial_2 S}\,,\) and \(\leq_{\partial_S}\) in the class \(\cC_2\) of bivariate copulas.
A visualization of the Schur order for copula derivatives is given in Figure \ref{fig:schur_order}. \\
As discussed above, the Schur order for conditional distributions and the Schur order for copula derivatives coincide in the following sense.

\begin{lemma}[Characterization of the Schur orderings]\label{lemcharschur}~\\
    Let \(D\) and \(E\) be bivariate copulas and let \((U_1,U_2)\) and \((V_1,V_2)\) be bivariate random vectors with distribution functions \(F_{U_1,U_2}=D\) and \(F_{V_1,V_2}=E\).
    For \(i\in \{1,2\}\) and \(j\in \{1,2\}\setminus \{i\}\,,\) the following statements are equivalent:
    \begin{enumerate}[(i)]
        \item \(D\leq_{\partial_i S} E\,.\)
        \item \((U_j|U_i) \leq_{S} (V_j|V_i)\,.\)
    \end{enumerate}
\end{lemma}

Under some positive dependence assumptions on the underlying distributions, the lower orthant order and the Schur order for conditional distributions are equivalent, as we discuss in the following subsection.


\subsection{Positive and negative dependence concepts}
\label{subsec:pos_and_neg_dep_concepts}

Many members of well-known bivariate copula families exhibit positive or negative dependencies.
We make use of the following positive dependence concepts.

\begin{definition}[Positive dependence concepts]\label{defposdep}~\\
    A bivariate random vector \((X_1,X_2)\) is said to be
    \begin{enumerate}[(i)]
        \item \emph{positive lower orthant dependent} (PLOD) if \(P(X_1\leq x, X_2\leq y) \geq P(X_1\leq x) P(X_2\leq y)\) for all \(x,y\in \R\,.\)
        \item \label{defposdep2} \emph{conditionally increasing in sequence} (CIS) if \(P(X_2\geq y\mid X_1 =x)\) is increasing in \(x\) outside a Lebesgue-null set for all \(y\in \R\,.\)
        \item \label{defposdep3} \emph{conditionally increasing} (CI) if \((X_1,X_2)\) and \((X_2,X_1)\) are CIS.
        \item \emph{totally positive of order \(2\)} (\(\TP\)) if it has a Lebesgue density \(f\) that is log-supermodular, i.e., \(\log f(x\vee y) + \log f(x\wedge y) \geq \log f(x)+ \log f(y)\) for all \(x,y\in \R^2\,,\) where \(x\wedge y\) and \(x\vee y\) denote the componentwise minimum and maximum, respectively.
    \end{enumerate}
\end{definition}

For continuous marginal distribution functions, the terms in the above definition depend only on the underlying copula, so we also refer the definition to copulas.
The concepts are related by
\begin{align}\label{implposdepcon}
	\TP ~~~\Longrightarrow ~~~ \text{CI} ~~~\Longrightarrow ~~~ \text{CIS} ~~~\Longrightarrow ~~~ \text{PLOD}\,,
\end{align}
where all implications are strict, see \cite[page 146]{Mueller-Stoyan-2002} for an overview of these concepts.
The following lemma relates the lower orthant order and the Schur order for copula derivatives under some positive dependence conditions, see \cite[Lemma 3.16]{Ansari-Rueschendorf-2021}.

\begin{lemma}\label{lem:3.16}
    Let $D$ and $E$ be bivariate copulas.
    Then, the following statements hold true: 
    \begin{enumerate}[(i)]
        \item If $E$ is CIS, then $D\leq_{\partial_1 S} E$ implies $D\leq_{lo} E$.
        \item If $D$ and $E$ are CIS, then $D\leq_{\partial_1 S} E$ and $D\leq_{lo}E$ are equivalent.
    \end{enumerate}
\end{lemma}
The Schur order for conditional distributions generates large subclasses of distributions for which the extremal elements with respect to the lower orthant order are CIS.
To be more precise, consider for a bivariate copula \(E\) the subclass \(\cC^E:=\{C\in \cC_2 \mid C\leq_{\partial_1 S} E\}\) of bivariate copulas that are smaller than \(E\) or equal to \(E\) in the Schur order for copula derivatives with respect to the first component.
The following result is due to \cite[Proposition 3.17]{Ansari-Rueschendorf-2021}.

\begin{lemma}[Extremal elements in \(\cC^E\)]\label{lem:extremal_elements}~\\
    For any bivariate copula \(E\,,\) the following statements hold true:
    \begin{enumerate}[(i)]
        \item There exists a uniquely determined minimal copula \(E_\downarrow\) in \(\cC^E\) and a uniquely determined maximal copula \(E_\uparrow\) in \(\cC^E\) such that \(E_\downarrow \leq_{lo} D \leq_{lo} E_\uparrow\) for all \(D\in \cC^E\,,\)
        \item \(E_\uparrow\) is CIS.
        \item \(E_\downarrow(u,v) = u - E_\uparrow(1-u,v)\) for all \((u,v)\in [0,1]^2\,.\)
        \item \(E_\uparrow=_{\partial_1 S} E =_{\partial_1 S} E_\downarrow\)
    \end{enumerate}
\end{lemma}

The copula \(E_\uparrow\) in the above lemma is denoted as \emph{increasing rearranged copula}, see \cite[Proof of Proposition 3.17]{Ansari-Rueschendorf-2021} and \cite{strothmann2022rearranged} for a construction of \(E_\uparrow\,.\)
In analogy to the CIS property, we say that a random vector $(X_1, X_2)$ (or its distribution function) is \emph{conditionally decreasing in sequence} (CDS) if \(P(X_2\geq y\mid X_1 =x)\) is decreasing in \(x\) outside a Lebesgue-null set for all \(y\in \R\), and \emph{conditionally decreasing} (CD) if also $(X_2, X_1)$ is conditionally decreasing in sequence.
Similar to Lemma \ref{lem:3.16} we obtain the following result.
\begin{lemma}\label{lem:3.16_cds}
    Let $D$ and $E$ be bivariate copulas.
    Then, the following statements hold true: 
    \begin{enumerate}[(i)]
        \item If $E$ is CDS, then $D\leq_{\partial_1 S} E$ implies $E\leq_{lo} D$.
        \item If $D$ and $E$ are CDS, then $D\leq_{\partial_1 S} E$ and $E\leq_{lo}D$ are equivalent.
    \end{enumerate}
\end{lemma}

As we list in Tables \ref{tab:arch_results} and \ref{tab:non_arch_results}, many well-known copulas are CI or CD and thus coincide with their increasing rearranged copula \(C_\uparrow\) or their decreasing rearranged copula \(C_{\downarrow}\,.\)

\subsection{Measures of association}\label{secmoca}

In this section, we consider some well-known measures of association.
While Spearman's rho, Kendall's tau and the tail-dependence coefficients are consistent with the lower orthant order, Chatterjee's rank correlation is consistent with the Schur order for conditional distributions.

\subsubsection{Measures of concordance}

Let \((X,Y)\) be a random vector with continuous marginal distribution functions and let \(C\) be its uniquely determined copula.
Then \emph{Spearman's rho}, denoted by \(\rho_S(X,Y)\) or \(\rho_S(C)\,,\) is defined by
\begin{align}\label{frm_rho_integral}
    \rho_S(X,Y)=\rho_S(C) = 12 \int_{[0,1]^2} C(u,v) \de \lambda^2(u,v) - 3\,,
\end{align}
where \(\lambda^2\) denotes the Lebesgue measure on \([0,1]^2\,.\)
Further, \emph{Kendall's tau}, denoted by \(\tau(X,Y)\) or \(\tau(C)\,,\) is defined by
\begin{align}\label{frm_tau_integral}
    \tau(X,Y) = \tau(C) = 4 \int_{[0,1]^2} C(u,v) \de C(u,v) - 1\,.
\end{align}

Both measures fulfil the axioms of a measure of concordance and are, in particular, consistent with the pointwise ordering of copulas as follows, see, e.g., \cite[Theorem 2.4.9]{Durante-2016}.

\begin{lemma}[Consistency with \(\leq_{lo}\)]\label{lemmoc}~\\
    Let \((X,Y)\) and \((X',Y')\) be bivariate random vectors with continuous distribution functions.
    Then \((X,Y)\leq_{lo} (X',Y')\) implies \(\rho(X,Y)\leq \rho (X',Y')\) and \(\tau(X,Y)\leq \tau(X',Y')\,.\)
\end{lemma}

\subsubsection{Measures of predictability}

For a bivariate random vector \((X,Y)\,,\) a measure of predictability \(\kappa(Y|X)\) (also known as dependence measure) takes values in the interval \([0,\,1]\) where \(0\) is attained if and only if \(X\) and \(Y\) are independent and where \(1\) is attained if and only if \(Y\) is perfectly dependent on $X$. \\
A recently studied measure of predictability that has attracted a lot of attention is \emph{Chatterjee's rank correlation} \(\xi(Y|X)\), also known as \emph{Dette-Siburg-Stoimenov's} dependence measure, which is defined for a bivariate random vector \((X,Y)\) by \eqref{eqchatt},
see \cite{chatterjee2020,dette2013copula}.
Chatterjee's rank correlation is consistent with respect to the Schur order for conditional distributions as follows, see \cite[Theorem 3.5]{Ansari-Fuchs-2022}.

\begin{lemma}[Consistency with \(\leq_S\)]\label{lemconS}~\\
Let \((X,Y)\) and \((X',Y')\) be bivariate random vectors.
Then \((Y|X)\leq_S (Y'|X')\) implies \(\xi(Y|X)\leq \xi(Y'|X')\,.\)
\end{lemma}

If \((X,Y)\) has continuous marginal distribution functions, then \(\xi(Y|X)\) depends only on the copula \(C\) of \((X,Y)\) and
\begin{align}\label{frm_xi_integral}
    \xi(C):=\xi(Y|X) = 6 \int_0^1\int_0^1 (\partial_1 C(u,v))^2 \de u \de v - 2\,,
\end{align}
see, e.g., \cite{fuchs2021quantifying}.
Due to Lemmas \ref{lemcharschur} and \ref{lemconS}, for bivariate copulas \(D\) and \(E\,,\) \(D\leq_{\partial_1 S} E\) implies \(\xi(D)\leq \xi(E)\,;\) we refer to \cite{Ansari-Fuchs-2023} for a class of measures of predictability that are consistent with \(\leq_S\,.\)

\subsubsection{Tail dependence}

Further classical measures of association for bivariate copulas are the tail dependence coefficients, which are defined as follows, see, e.g., \cite{Joe-1997,Nelsen-2006}.

\begin{definition}[Tail dependence coefficients]~\\
    Let $C$ be a bivariate copula.
    Then, the \emph{lower tail dependence coefficient} of $C$ is defined by
    \[
        \lambda_L = \lambda_L^C=\lim _{t \rightarrow 0^{+}} \frac{C(t, t)}{t}
    ,\]
    and the \emph{upper tail dependence coefficient} of $C$ by
    \[
        \lambda_U = \lambda_U^C =2-\lim _{t \rightarrow 1^{-}} \frac{1-C(t, t)}{1-t}
    .\]
\end{definition}
The following lemma is immediate from the definition of the tail dependence coefficient.

\begin{lemma}[Consistency with \(\leq_{lo}\)]\label{lem:tail_consistent_with_lo}~\\
    Let $C$ and $D$ be bivariate copulas with $C\leq_{lo}D$.
    Then $\lambda_L^C\leq \lambda_L^D$ and $\lambda_U^C\leq \lambda_U^D$.
\end{lemma}

\section{Dependence properties of bivariate copula families}\label{secdeppropbivcopfam}
\begin{table}[tbp]
    \begin{center}
        \scalebox{0.8}{
        \begin{tabular}{llll}
            \toprule 
            & Family & Notation & Cumulative Distribution Function $C(u, v)$ \\
            \midrule
            Arch. & Clayton & $C^{\text{Cl}}_{\theta}$ & $\inbrackets{\left(u^{-\theta}+v^{-\theta}-1\right)^{-1 / \theta}}_+$ \\
            & Nelsen2 & $C^{\text{N2}}_{\theta}$ & $\left(1-[(1-u)^\theta+(1-v)^\theta]^{1 / \theta}\right)_+$ \\
            & Ali-Mikh.-Haq & $C^{\text{AMH}}_{\theta}$ & $\frac{u v}{1-\theta(1-u)(1-v)}$ \\
            & Gumbel-Hougaard & $C^{\text{GH}}_{\theta}$ & $\exp (-[(-\ln u)^\theta+(-\ln v)^\theta]^{1 / \theta})$ \\
            & Frank & $C^{\text{Fra}}_{\theta}$ & $-\frac{1}{\theta} \ln \left(1+\frac{\left(e^{-\theta u}-1\right)\left(e^{-\theta v}-1\right)}{e^{-\theta}-1}\right)$\\
            & Joe & $C^{\text{Joe}}_{\theta}$ & $1-\left[(1-u)^\theta+(1-v)^\theta-(1-u)^\theta(1-v)^\theta\right]^{1 / \theta}$\\
            & Nelsen7 & $C^{\text{N7}}_{\theta}$ & $(\theta u v+(1-\theta)(u+v-1))_+$ \\
            & Nelsen8 & $C^{\text{N8}}_{\theta}$ & $\left(\frac{\theta^2 u v-(1-u)(1-v)}{\theta^2-(\theta-1)^2(1-u)(1-v)}\right)_+$\\
            & Gumb.-Barn. & $C^{\text{GB}}_{\theta}$ & $u v \exp (-\theta \ln u \ln v)$ \\
            & Nelsen10 & $C^{\text{N10}}_{\theta}$ & $u v /\left[1+\left(1-u^\theta\right)\left(1-v^\theta\right)\right]^{1 / \theta}$ \\
            & Nelsen11 & $C^{\text{N11}}_{\theta}$ & $\left(u^\theta v^\theta-2\left(1-u^\theta\right)\left(1-v^\theta\right)\right)_+^{1 / \theta}$ \\
            & Nelsen12 & $C^{\text{N12}}_{\theta}$ & $\left(1+\left[(u^{-1}-1)^\theta+(v^{-1}-1)^\theta\right]^{1 / \theta}\right)^{-1}$\\
            & Nelsen13 & $C^{\text{N13}}_{\theta}$ & $\exp \left(1-\left[(1-\ln u)^\theta+(1-\ln v)^\theta-1\right]^{1 / \theta}\right)$\\
            & Nelsen14 & $C^{\text{N14}}_{\theta}$ & $\left(1+\left[(u^{-1 / \theta}-1)^\theta+(v^{-1 / \theta}-1)^\theta\right]^{1 / \theta}\right)^{-\theta}$\\
            & Genest-Ghoudi & $C^{\text{GG}}_{\theta}$ & $\left(1-\left[(1-u^{1 / \theta})^\theta+(1-v^{1 / \theta})^\theta\right]^{1 / \theta}\right)_+^\theta$\\
            & Nelsen16 & $C^{\text{N16}}_{\theta}$ & $\frac{1}{2}\left(S+\sqrt{S^2+4 \theta}\right), \quad S=u+v-1-\theta\left(\frac{1}{u}+\frac{1}{v}-1\right)$\\
            & Nelsen17 & $C^{\text{N17}}_{\theta}$ & $\left(1+\frac{1}{2^{-\theta}-1}\left[(1+u)^{-\theta}-1\right]\left[(1+v)^{-\theta}-1\right]\right)^{-1 / \theta}-1$\\
            & Nelsen18 & $C^{\text{N18}}_{\theta}$ & $\left(1+\theta / \ln \left[e^{\theta /(u-1)}+e^{\theta /(v-1)}\right]\right)_+$\\
            & Nelsen19 & $C^{\text{N19}}_{\theta}$ & $\theta / \ln (e^{\theta / u}+e^{\theta / v}-e^\theta)$\\
            & Nelsen20 & $C^{\text{N20}}_{\theta}$ & $\left[\ln \left(\exp \left(u^{-\theta}\right)+\exp \left(v^{-\theta}\right)-e\right)\right]^{-1 / \theta}$\\
            & Nelsen21 & $C^{\text{N21}}_{\theta}$ & $1-\left(1-\left(\left[1-(1-u)^\theta\right]^{1 / \theta}+\left[1-(1-v)^\theta\right]^{1 / \theta}-1\right)_+^\theta\right)^{1 / \theta}$\\
            & Nelsen22 & $C^{\text{N22}}_{\theta}$ & $\left(\sin\left(\asin(u^{\theta} - 1 ) + \asin(v^{\theta} - 1 ) \right) + 1\right)^{\frac{1}{\theta}}\1_{\setof{\asin(u^{\theta} - 1 ) + \asin(v^{\theta} - 1 ) \geq - \frac{\pi}{2}}}$\\
            \midrule
            EV & BB5 & $C^{\text{BB5}}_{\theta,\delta}$ & $\exp\of{-\left((-\log v)^{\theta} + (-\log u)^{\theta} - \left((-\log u)^{- \delta \theta} + (-\log v)^{- \delta \theta}\right)^{- \frac{1}{\delta}}\right)^{\frac{1}{\theta}}}$ \\
            & Cuadras-Augé & $C^{\text{CA}}_{\delta}$ & $(u\wedge v)^{\delta}(uv)^{1-\delta}$ \\
            & Galambos & $C^{\text{Gal}}_{\delta}$ & $uv\exp\of{\left(\log\of{\frac1u}^{- \delta} + \log\of{\frac1v}^{- \delta}\right)^{- \frac{1}{\delta}}} $ \\
            & Gumbel-Hougaard & $C^{\text{GH}}_{\theta}$ & $\exp (-[(-\ln u)^\theta+(-\ln v)^\theta]^{1 / \theta})$ \\
            & Hüsler-Reiss & $C^{\text{HR}}_{\delta}$ & $\exp\left( \ln(uv) A\Big(\frac{\ln v}{\ln u + \ln v}\Big) \right) $ with $A(t):= (1-t)\Phi(z_{1-t}) + t\Phi(z_t)$ \\
            & & & where $z_t := \frac1\delta + \frac{\delta}2\log\of{\frac{t}{1-t}}$ \\
            & Joe-EV & $C^{\text{JoeEV}}_{\alpha_1,\alpha_2,\delta}$ & $uv\exp\of{\left(\left(\alpha_{2} \log\of{\frac1v}\right)^{- \delta} + \left(\alpha_{1}\log\of{\frac1u}\right)^{- \delta}\right)^{- \frac{1}{\delta}}} $ \\
            & Marshall-Olkin & $C^{\text{MO}}_{\alpha_1, \alpha_2}$ & $\min \left(u^{1-\alpha_1} v, u v^{1-\alpha_2}\right)$\\
            & t-EV & $C^{\text{tEV}}_{\rho}$ & $\exp\left( \ln(uv) A\Big(\frac{\ln v}{\ln u + \ln v}\Big) \right) $ with $A(t) = (1-t)T_{\nu + 1}(z_{1-t}) + t T_{\nu + 1}(z_t)$ \\
            & & & where $z_t := \sqrt{\frac{1+\nu}{1-\rho^2}}\inbrackets{\of{\frac{t}{1-t}}^{1/\nu}-\rho}$ \\
            & Tawn & $C^{\text{Tawn}}_{\alpha_1,\alpha_2, \theta}$ & $u^{1-\alpha_1}v^{1-\alpha_2}e^{-\left(\left(\alpha_{2} \log\of{\frac1v}\right)^{\theta} + \left(\alpha_{1}\log\of{\frac1u}\right)^{\theta}\right)^{\frac{1}{\theta}}} $ \\
            \midrule
            Ell. & Gaussian & $C^{\text{Gauss}}_{\rho}$ & $\int_{-\infty}^{\Phi^{-1}\left(v\right)} \int_{-\infty}^{\Phi^{-1}\left(u\right)} \frac{1}{2 \pi\left(1-\rho^2\right)^{1 / 2}} \exp \left\{\frac{-\left(x^2-2 \rho x y+y^2\right)}{2\left(1-\rho^2\right)}\right\} \mathrm{d} x \mathrm{~d} y$  \\
            & Student-t & $C^{\text{t}}_{\rho, \nu}$ & $\int_{-\infty}^{T_\nu^{-1}\left(v\right)} \int_{-\infty}^{T_\nu^{-1}\left(u\right)}\frac{\Gamma[(\nu+\rho) / 2]}{\Gamma(\nu / 2) \nu^{\rho / 2} \pi^{\rho / 2}|\Sigma|^{1 / 2}}\left[1+\frac{1}{\nu}\bsx^{\mathrm{T}} \Sigma^{-1}\bsx\right]^{-(\nu+\rho) / 2}  \de \bsx$\\
            & Laplace & $C^{\text{Lap}}_{\rho}$ & $\int_{-\infty}^{F^{-1}\left(v\right)} \int_{-\infty}^{F^{-1}\left(u\right)} \frac{1}{\pi|\Sigma|^{1/2}}\left(\frac{\bsx^{\prime} \Sigma^{-1} \bsx}{2}\right)^{v / 2} K_v\left(\sqrt{2 \bsx^{\prime} \Sigma^{-1} \bsx}\right) \de \bsx$\\
            \midrule
            Uncl. & Fréchet & $C^{\text{Fré}}_{\alpha, \beta}$ & $\alpha M(u, v)+(1-\alpha-\beta) \Pi(u, v)+\beta W(u, v)$\\
            & Mardia & $C^{\text{Ma}}_{\theta}$ & $\frac{\theta^2(1+\theta)}{2}M(u, v)+(1-\theta^2) \Pi(u, v)+\frac{\theta^2(1-\theta)}{2} W(u, v)$ \\
            & Farl.-Gumb.-Morg. & $C^{\text{FGM}}_{\theta}$ & $uv + \theta u v (1-u)(1-v)$  \\
            & Plackett & $C^{\text{Pl}}_{\theta}$ & $\frac{1+(\theta-1)(u+v) - \sqrt{(1+(\theta-1)(u+v))^2-4 u v \theta(\theta-1)}}{2(\theta-1)}$ for $\theta\neq 1$ else $\Pi$ \\
            & Raftery & $C^{\text{Ra}}_{\delta}$ & $u \wedge v + (1-\delta)(uv)^{\frac1{1-\delta}}\inbrackets{1-(u\vee v)^{-\frac{1+\delta}{1-\delta}}}$\\
            \bottomrule
        \end{tabular}
        }
    \end{center}
    \caption{
        Overview of Archimedean (Arch.), extreme-value (EV), elliptical (Ell.), and unclassified (Uncl.) copula families for which dependence properties are studied in this paper.
        The Gumbel-Hougaard family also belongs to the class of EV copula families.
        \(\Phi\), \(T_\nu\) and \(F\) denote the standard normal, Student-t (with \(\nu\) degrees of freedom) and standard Laplace distribution function, respectively, and
        $K_\nu(x):=\left(\frac{2}{x}\right)^\nu \frac{\Gamma(\nu+1 / 2)}{\sqrt{\pi}} \int_0^{\infty}\left(1+u^2\right)^{-(\nu+1 / 2)} \cos (r u) \de u$ for $x>0,\nu > -1/2$ is the modified Bessel function.
    }
    \label{tab:copulas}
\end{table}

As motivated in the introduction, in this section, we study various dependence properties of bivariate copula families that are often used in applications.
More precisely, we consider the families listed in Table \ref{tab:copulas}.
For each copula family, we verify or reference for which parameters the copulas are CI, CD or \(\TP\,,\) and whether the families are increasing or decreasing in the lower orthant order and in the recently established Schur order for copula derivatives.
We present the results for Archimedean copula families in Tables \ref{tab:arch_overview} and \ref{tab:arch_results}; for the extreme-value copula families, the elliptical copulas families and the unclassified copula families, we refer to Tables \ref{tab:non_arch_overview} and \ref{tab:non_arch_results}. 
Note that the independence copula is member of many copula families. 
Hence, whenever such families are lower orthant ordered, they consist of positive lower orthant dependent (PLOD) and/or negative lower orthant dependent copulas. \\
In order to verify the dependence properties, we provide in the following subsection sufficient positive/negative dependence and ordering conditions for the families from the respective classes of copulas.
As we study in Section \ref{secmoa}, various ordering properties imply monotonicity of Chatterjee's xi, Spearman's rho and Kendall's tau with respect to the parameter of the underlying copula family.
We also provide some closed-form expressions for these measures in dependence on the copula parameter.

\subsection{Ordering and positive dependence properties}

While characterizations of the lower orthant order in terms of the generator or correlation parameter are well-known for Archimedean and elliptical copula families (see the Propositions \ref{prop_ord_arch_cop} \eqref{ord_arch_cop_1} and \ref{prop:ordering_elliptical_copulas} \eqref{ord_ell_cop_1} below), we establish in Theorem \ref{thm:ordering_ev_copulas} the equivalence between lower orthant ordering of extreme-value copulas and pointwise ordering of the associated Pickands dependence functions.
For deriving ordering results with respect to the Schur order for copula derivatives, we make use of positive dependence properties for the respective classes of copulas.
Before proceeding with the specific classes of copulas, we give two general results concerning the relation between the Schur order for copula derivatives and the lower orthant order as well as dependence properties for survival copulas.\\
The following result characterizes the Schur order for copula derivatives in terms of the pointwise ordering of the rearranged copulas.
If a copula is CI, it coincides with its increasing rearranged copula.
Hence, for families of CI copulas, the lower orthant order is equivalent to the Schur order for copula derivatives with respect to the first (and similarly to the second) component.

\begin{proposition}[Rearranged copulas]\label{proprearcop}~\\
    For \(D,E\in \cC_2\,,\) 
    the following are equivalent:
    \begin{enumerate}[(i)]
        \item \label{proprearcop1} $ D\leq_{\partial_1 S} E$ 
        \item \label{proprearcop2}$ D_\uparrow \leq_{lo} E_\uparrow $
        \item \label{proprearcop3}$ D_\downarrow \geq_{lo} E_\downarrow$
    \end{enumerate}
\end{proposition}

\begin{proof}
    Statement \eqref{proprearcop1} is equivalent to $\cC^D\subseteq\cC^E$.
    Applying Lemma \ref{lem:extremal_elements}, it follows that $D_\uparrow\in\cC^E$ and $ D_\uparrow \leq_{lo} E_\uparrow $.
    Hence, \eqref{proprearcop1} implies \eqref{proprearcop2}.
    To show the reverse direction, we obtain from Lemma \ref{lem:extremal_elements} (ii) that $D_\uparrow$ and $E_\uparrow$ are CIS.
    Hence, by Lemma \ref{lem:3.16} (ii), we have $ D_\uparrow\leq_{\partial_1 S} E_\uparrow$.
    Since \(E_\uparrow =_{\partial_1 S} E\) due to Lemma \ref{lem:extremal_elements} (iv), it follows that $D_\uparrow\in\cC^{E_\uparrow}=\cC^E$, from which we get $ D\leq_{\partial_1 S} E$.
    The equivalence of (ii) and (iii) is given by Lemma \ref{lem:extremal_elements}\,(iii).
\end{proof}

For a bivariate copula \(C\,,\) the survival copula \(\hat{C}\) associated with \(C\) is defined by
\begin{align}\label{defsurvcop}
    \hat{C}(u,v) = u+v+C(1-u,1-v)-1\,, \quad(u,v)\in [0,1]^2\,,
\end{align}
see, e.g., \cite[Definition 1.7.18]{Durante-2016}.
Due to the following result, all dependence and monotonicity properties in Tables \ref{tab:arch_results} and \ref{tab:non_arch_results} transfer to the associated survival copula families.

\begin{proposition}[Survival copula]~\\
    Let \(D\) and \(E\) be a bivariate copula.
    Then, the following statements hold true.
    \begin{enumerate}[(i)]
        \item \label{frm_lo_survival} \(D\leq_{lo} E\) if and only if \(\hat{D}\leq_{lo} \hat{E}\,.\)
        \item \label{frm_cis_survival}\(D\) is CIS if and only if \(\hat{D}\) is CIS
        \item \label{frm_partial_2_survival}\(D\leq_{\partial_1 S} E\) if and only if \(\hat{D}\leq_{\partial_1 S} \hat{E}\,.\)
    \end{enumerate}
\end{proposition}

\begin{proof}
    Since $\hat{\hat{D}}=D$, it suffices to show only one implication for the statements.
    Statement \eqref{frm_lo_survival} follows from the definition of a survival copula in \eqref{defsurvcop}.
    For \eqref{frm_cis_survival}, note that $D$ is CIS if and only if it is concave in $u$ for all $v$, see, e.g., \cite[Corollary 5.2.11]{Nelsen-2006}.
    Since $D$ is concave in $u$ for all $v$, also $u \mapsto u + v - 1 + C(1-u, 1-v)$ is concave for all $v$.
    Hence, $\hat{D}$ is CIS. \\
    Lastly, to derive \eqref{frm_partial_2_survival}, let $v\in[0,1]$ be arbitrary and notice that $\partial_1\hat{D}(\cdot,v) = 1- (\partial_1 D)(1-\cdot, 1-v)$.
    From the hypothesis follows $\partial_1 D(\cdot, 1-v)\prec_S \partial_1 E(\cdot, 1-v)$.
    Since the Schur order is invariant under rearrangements, one gets $ 1-\partial_1 D(\cdot, 1-v)\prec_S 1- \partial_1 E(\cdot, 1-v)$ and thus
    \[ 
        \partial_1\hat{D}(\cdot, v)
        =_S \partial_1\hat{D}(1-\cdot, v)
        = 1- \partial_1 D(\cdot, 1-v)
        \prec_S
        1-\partial_1 E(\cdot, 1-v)
        =\partial_1\hat{E}(1-\cdot, v)
        =_S\partial_1\hat{E}(\cdot, v)
    .\]
    Since $v$ is arbitrary, we conclude that $\hat{D}\leq_{\partial_1 S} \hat{E}$.
\end{proof}

\subsubsection{Archimedean copulas}\label{secdeppropArch}

\begin{table}[t!]
    \begin{center}
        \scalebox{0.8}{
        \begin{tabular}{llllll}
            \toprule
            Family & $\theta$-Interval & $\varphi(0)$ & Generator $\varphi(t)$ & Inverse Generator $\psi(y)$ & Special/Limiting Cases \\
            \midrule
            Clayton& $[-1,\infty)$ & $-\frac{1}{\theta}\text{ if }\theta <0$&  $\frac{-1 + t^{- \theta}}{\theta}$ & $\left(\theta y + 1\right)^{- \frac{1}{\theta}}\1_{\setof{\theta > 0\vee y\leq -1/\theta }}$ & $C^{\text{Cl}}_{-1} = W$, $C^{\text{Cl}}_0=\Pi$, \\
            & &else $\infty$ & & & $C^{\text{Cl}}_1 = \frac{\Pi}{\Sigma - \Pi}$, $C^{\text{Cl}}_{\infty} = M$ \\
            Nelsen2& $[1,\infty)$ & 1 & $\left(1 - t\right)^{\theta}$ & $\left(1 - y^{\frac{1}{\theta}}\right)\1_{\setof{y\leq 1}}$ & $C^{\text{N2}}_{1} = W$, $C^{\text{N2}}_{\infty} = M$ \\
            AMH & $[-1,1]$ & $\infty$ &  $\log{\left(\frac{- \theta \left(1 - t\right) + 1}{t} \right)}$&         $\frac{\theta - 1}{\theta - e^{y}}$ & $C^{\text{AMH}}_{0} = \Pi$,\\
            & & & & & $C^{\text{AMH}}_1 = \frac{\Pi}{\Sigma - \Pi}$ \\
            Gum.-Ho. & $[1,\infty)$ & $\infty$&$\left(- \log{\left(t \right)}\right)^{\theta}$ & $e^{- y^{\frac{1}{\theta}}}$ & $C^{\text{GH}}_{1} = \Pi$, $C^{\text{GH}}_{\infty} = M$ \\
            Frank & $\R$ & $\infty$ &  $- \log{\left(\frac{-1 + e^{- t \theta}}{-1 + e^{- \theta}} \right)}$& $\frac{\theta + y - \log{\left(- e^{\theta} + e^{\theta + y} + 1 \right)}}{\theta}$& $C^{\text{Fra}}_{-\infty} = W,C^{\text{Fra}}_{0} = \Pi$,  \\
            & & & & & $C^{\text{Fra}}_{\infty} = M$ \\
            Joe & $[1,\infty)$ & $\infty$& $- \log{\left(1 - \left(1 - t\right)^{\theta} \right)}$ & $1 - \left(1- e^{- y}\right)^{\frac{1}{\theta}}$ & $C^{\text{Joe}}_{1} = \Pi$, $C^{\text{Joe}}_{\infty} = M$  \\
            Nelsen7 & $[0,1]$ & $\log\of{\frac1{1-\theta}}$ &  $- \log{\left(t \theta - \theta + 1 \right)}$ &$\left(\left(1 - \frac{1}{\theta} + \frac{e^{- y}}{\theta}\right)\right.$ & $C^{\text{N7}}_{0} = W$, $C^{\text{N7}}_{1} = \Pi$  \\
            & & & &$\left.\quad\cdot\1_{\setof{y\leq -\log(1-\theta)}}\right)$ & \\
            Nelsen8 & $[1,\infty)$ & 1 & $\frac{1 - t}{t \left(\theta - 1\right) + 1}$ & $\frac{1 - y}{\theta y - y + 1}\1_{\setof{y\leq 1}}$ & $C^{\text{N8}}_{1} = W$, $C^{\text{N8}}_{\infty} = \frac{\Pi}{\Sigma - \Pi}$ \\
            Gum.-Ba. & $[0,1]$ & $\infty$& $\log{\left(- \theta \log{\left(t \right)} + 1 \right)}$ & $e^{\frac{1 - e^{y}}{\theta}}$ & $C^{\text{GB}}_0 = \Pi$ \\
            Nelsen10 & $[0,1]$ & $\infty$&  $\log{\left(-1 + 2 t^{- \theta} \right)}$ &  $\left(\frac{2}{e^{y} + 1}\right)^{\frac{1}{\theta}}$ & $C^{\text{N10}}_0 = \Pi$ \\
            Nelsen11 &$[0,1/2]$ & $\log(2)$ & $\log{\left(2 - t^{\theta} \right)}$ & $\left(2 - e^{y}\right)^{\frac{1}{\theta}}\1_{\setof{y\leq\log(2)}}$ & $C^{\text{N11}}_0 = \Pi$ \\
            Nelsen12 & $[1,\infty)$ & $\infty$& $\left(-1 + \frac{1}{t}\right)^{\theta}$ & $\frac{1}{y^{\frac{1}{\theta}} + 1}$ & $C^{\text{N12}}_{1} = \frac{\Pi}{\Sigma - \Pi}$, $C^{\text{N12}}_{\infty} = M$ \\
            Nelsen13 & $[0,\infty)$ & $\infty$& $\left(1 - \log{\left(t \right)}\right)^{\theta} - 1$ &$e^{1 - \left(y + 1\right)^{\frac{1}{\theta}}}$ & $C^{\text{N13}}_{1} = \Pi$, $C^{\text{N13}}_{\infty} = M$  \\
            Nelsen14 & $[1,\infty)$ & $\infty$&    $\left(-1 + t^{- \frac{1}{\theta}}\right)^{\theta}$ &  $\left(y^{\frac{1}{\theta}} + 1\right)^{- \theta}$ & $C^{\text{N14}}_{1} = \frac{\Pi}{\Sigma - \Pi}$, $C^{\text{N14}}_{\infty} = M$  \\
            Gen.-Gh. &$[1,\infty)$ & 1 &  $\left(1 - t^{\frac{1}{\theta}}\right)^{\theta}$ &  $\left(1 - y^{\frac{1}{\theta}}\right)^{\theta}\1_{\setof{y\leq1}}$ &  $C^{\text{GG}}_{1} = W$, $C^{\text{GG}}_{\infty} = M$ \\
            Nelsen16 &$[0,\infty)$ & $1\text{ if }\theta = 0$ & $\left(1 - t\right) \left(1 + \frac{\theta}{t}\right)$ &  $\left(\frac{1 - y- \theta}{2} 
            \right.$ &  $C^{\text{N16}}_{0} = W$,  \\
            && else $\infty$ & & $\left.~ + \frac{\sqrt{\theta^{2} + 2 \theta y + 2 \theta + y^{2} - 2 y + 1}}{2}\right)$ & $C^{\text{N16}}_{\infty} = \frac{\Pi}{\Sigma - \Pi}$ \\
            Nelsen17 &$\R\setminus\setof{0}$ & $\infty$&  $- \log{\left(\frac{-1 + \left(t + 1\right)^{- \theta}}{-1 + 2^{- \theta}} \right)}$ &     $\left(\frac{2^{\theta} e^{y}}{2^{\theta} e^{y} - 2^{\theta} + 1}\right)^{\frac{1}{\theta}} - 1$ & $C^{\text{N17}}_{-1} = \Pi$, $C^{\text{N17}}_{\infty} = M$ \\
            Nelsen18 & $[2,\infty)$ & $e^{-\theta}$ &$e^{\frac{\theta}{t - 1}}$ &  $\inbrackets{\frac{\theta}{\log{\left(y \right)}} + 1}\1_{\setof{y\geq e^{-\theta}}}$ & $C^{\text{N18}}_{\infty} = M$ \\
            Nelsen19 & $[0,\infty)$ & $\infty$&    $- e^{\theta} + e^{\frac{\theta}{t}}$ & $\frac{\theta}{\log{\left(y + e^{\theta} \right)}}$ & $C^{\text{N19}}_{0} = \frac{\Pi}{\Sigma - \Pi}$, $C^{\text{N19}}_{\infty} = M$ \\
            Nelsen20 & $[0,\infty)$ & $\infty$&    $e^{t^{- \theta}} - e$ & $\log{\left(y + e \right)}^{- \frac{1}{\theta}}$ & $C^{\text{N20}}_{0} = \Pi$, $C^{\text{N20}}_{\infty} = M$ \\
            Nelsen21 & $[1,\infty)$ & 1 & $1 - \left(1 - \left(1 - t\right)^{\theta}\right)^{\frac{1}{\theta}}$ & $\left(\left(1 - (1 - \left(1 - y\right)^{\theta})^{1/\theta}\right)\right.$ & $C^{\text{N21}}_{1} = W$, $C^{\text{N21}}_{\infty} = M$ \\
            & & & & $\left.\quad\cdot\1_{\setof{y\leq \frac{\pi}{2}}}\right)$ & \\
            Nelsen22 & $[0,1]$ & $\pi/2$  &  $- \operatorname{asin}{\left(t^{\theta} - 1 \right)}$ &   $\left(1 - \sin{\left(y \right)}\right)^{\frac{1}{\theta}}\1_{\setof{y\leq \pi/2}}$ & $C^{\text{N22}}_{0} = \Pi$ \\
            \bottomrule
        \end{tabular}
        }
    \end{center}
    \caption{ 
        Overview of Archimedean copula families, for which dependence properties are given in Table \ref{tab:arch_results}.
        The generators are taken from \cite[Table 3.2]{Nelsen-2006}.
        Generator and inverse generator for $W$ are $\varphi(t)=1-t$ and $\psi(y)=(1-y)_+$, for $\Pi$ they are $\varphi(t)=-\ln(t)$ and $\psi(y)=e^{-y}$, and for $\frac{\Pi}{\Sigma - \Pi}$ they are $\varphi(t)=t^{-1} -1$ and $\psi(y)=(y+ 1)^{-1}$.
        $M$ is not Archimedean.
    }
    \label{tab:arch_overview}
\end{table}

Positive dependence concepts for Archimedean copulas can be characterized in terms of their generators.
For a bivariate Archimedean copula \(C\) with sufficiently smooth inverse generator \(\psi\) it holds that
\begin{align}\label{lempdparch}
    C \text{ is CI\phantom{\(\TP\)}} &~ \Longleftrightarrow \quad -\psi'\phantom{''} \text{ is log-convex on the positive real line,}\\
    \label{lempdparch2} C \text{ is \(\TP\)\phantom{CI}} &~ \Longleftrightarrow \quad \phantom{-}\psi''\phantom{'} \text{ is log-convex on the positive real line,}
\end{align}
see \cite[Theorems 2.8 and 2.11]{Mueller-2005}.
Since Archimedean copulas are symmetric, the concepts CIS and CI coincide.
Concerning negative dependence, it follows similarly to the proof of \cite[Theorems 2.8]{Mueller-2005} that
\begin{align}\label{lempdparch_cd}
    C \text{ is CD\phantom{\(\TP\)}} &~ \Longleftrightarrow \quad -\psi'\phantom{''} \text{ is log-concave on the positive real line.}\
\end{align}
We make use of the positive and negative dependence properties \eqref{lempdparch} and \eqref{lempdparch_cd} to give sufficient conditions for the Schur ordering of bivariate Archimedean copula derivatives.
To this end, a function \(f\colon [0,\infty)\to [0,\infty)\) is said to be \textit{subadditive} if \(f(x+y)\leq f(x)+f(y)\) for all \(x,y\in [0,\infty)\,.\)

\begin{table}[t!]
    \begin{center}
        \scalebox{0.85}{
        \begin{tabular}{llllllll}
            \toprule 
            Copula & CI/CD & $\TP$ & $\leq_{lo}$ & $\leq_{\partial S}$ & $\lambda_L$ & $\lambda_U$\\
            \midrule
            Clayton & $\guparrow$~iff $\theta\geq 0$, & \cmark~iff $\theta\geq 0$ & $\gnearrow$ & $\gnearrow$ if $\theta\geq 0$, & $2^{-1/\theta}\1_{\setof{\theta\geq0}}$& 0  \\
            &  $\gdownarrow$~iff $\theta\leq 0$ & &  & $\gsearrow$ if $\theta\leq 0$ & &  \\
            Nelsen2 & \xmark~iff $\theta>1$& \xmark~& $\gnearrow$ & $\text{\xmark}^*$ & 0& $2 - 2^{1/\theta}$  \\
            Ali-Mikhail-Haq & $\guparrow$~iff $\theta \geq 0$, & \cmark~iff $\theta \geq 0$ & $\gnearrow$ & $\gnearrow$ if $\theta\geq 0$ & 0 & 0  \\
            &  $\gdownarrow$~iff $\theta\leq 0$ & &  & $\gsearrow$ if $\theta\leq 0$ & &  \\
            Gumbel-Hougaard & $\guparrow$ & \cmark~& $\gnearrow$ &$\gnearrow$&0 & $2 - 2^{1/\theta}$  \\
            Frank & $\guparrow$~iff $\theta \geq 0$, & \cmark~iff $\theta \geq 0$ & $\gnearrow$ & $\gnearrow$ if $\theta\geq 0$ & 0 &0 \\
            &  $\gdownarrow$~iff $\theta\leq 0$ & &  & $\gsearrow$ if $\theta\leq 0$ & &  \\
            Joe & $\guparrow$ & \cmark~& $\gnearrow$ & $\gnearrow$ &0 &$2 - 2^{1/\theta}$  \\
            Nelsen7  & $\gdownarrow$ & \xmark~iff $\theta < 1$ & $\gnearrow$ & $\gsearrow$ &0  &0 \\
            Nelsen8  & \xmark~iff $\theta>1$ & \xmark~& $\gnearrow$ & $\text{\xmark}^*$ & 0&0 \\
            Gumbel-Barnett  & $\gdownarrow$& \xmark~iff $\theta > 0$ & $\gsearrow$ & $\gnearrow$ &0 &0\\
            Nelsen10 & $\gdownarrow$ & \xmark~iff $\theta > 0$ & \xmark & $\text{\xmark}$~& 0&0  \\
            Nelsen11 & $\gdownarrow$ &\xmark~iff $\theta > 0$ & $\gsearrow$ & $\gnearrow$ &0 & 0\\
            Nelsen12 & $\guparrow$ & \cmark~ & $\gnearrow$ & $\gnearrow$ & $2^{-1/\theta}$& $2 - 2^{1/\theta}$ \\
            Nelsen13 & $\guparrow$~iff $\theta \geq 1$ & \cmark~iff $\theta \geq 1$ & $\gnearrow$ & $\gnearrow$ if $\theta\geq 1$ &0 & 0 \\
            Nelsen14 & $\guparrow$& \cmark~ & $\gnearrow$ & $\gnearrow$ & $1/2$ & $2 - 2^{1/\theta}$ \\
            Genest-Ghoudi & \xmark~iff $\theta>1$& \xmark~& $\gnearrow$ &  $\text{\xmark}^*$ &0 & $2 - 2^{1/\theta}$ \\
            Nelsen16 & $\guparrow$~iff $\theta\geq 3$ & \cmark~iff $\theta\geq 3+2\sqrt{2}$ & $\gnearrow$ & $\gnearrow$~if $\theta \geq 3$& $1/2$ & 0 \\
            Nelsen17 & $\guparrow$~iff $\theta \geq -1$, & \cmark~iff $\theta \geq -1$ & $\gnearrow$ & $\gnearrow$ if $\theta\geq -1$, & 0& 0\\
            & $\gdownarrow$~iff $\theta\leq -1$ & & & $\gsearrow$ if $\theta\leq -1$ & &  \\
            Nelsen18 & \xmark~& \xmark~ & $\gnearrow$ & $\text{\xmark}^*$ & 0&1 \\
            Nelsen19 & $\guparrow$ & \cmark~& $\gnearrow$ & $\gnearrow$ & 1 & 0 \\
            Nelsen20 & $\guparrow$ & \cmark~ & $\gnearrow$ & $\gnearrow$ & 1 & 0 \\
            Nelsen21 & \xmark~iff $\theta>1$& \xmark~& $\gnearrow$ &  $\text{\xmark}^*$ &0 & $2 - 2^{1/\theta}$ \\
            Nelsen22 & $\gdownarrow$ & \xmark~iff $\theta > 0$ & $\gsearrow$ & $\gnearrow$ & 0&0 \\
            \bottomrule
        \end{tabular}
        }
    \end{center}
    \caption{
        Dependence properties of Archimedean copula families from Table \ref{tab:arch_results}.
        For example, '$\guparrow$~iff $\theta\geq 0$' in the CI/CD column for the Clayton copula means that the copula is conditionally increasing (CI) if and only if $\theta\geq 0$.
        Similarly, '\xmark' in the \(\TP\) column for the Nelsen2 copula family means that the copulas are not \(\TP\) for any parameter.
        The columns \(\leq_{lo}\) and \(\leq_{\partial S}\) indicate whether (and for which parameters) the family is increasing (\(\gnearrow\)) or decreasing (\(\gsearrow\)) with respect to the lower orthant order and Schur order for copula derivatives, respectively, see Section \ref{secstoo} for the definitions.
        Results marked with * are obtained from numerical checks and neither referenced nor proved.
    }
    \label{tab:arch_results}
\end{table}

\begin{proposition}[Ordering Archimedean copulas]\label{prop_ord_arch_cop}~\\
    Let \(D_1\) and \(D_2\) be Archimedean copulas with inverse generator \(\psi_1\) and \(\psi_2\,,\) respectively.
    Then, the following statements hold true.
    \begin{enumerate}[(i)]
        \item \label{ord_arch_cop_1} \(D_1\leq_{lo} D_2\) if and only if \(\psi^{-1}_1\circ \psi_2\) is subadditive.
        \item If \(-\psi_1'\) and \(-\psi_2'\) are log-convex, then subadditivity of \(\psi_1^{-1}\circ \psi_2\) is equivalent to \(D_1\leq_{\partial S} D_2\,.\)
        \item If \(-\psi_1'\) and \(-\psi_2'\) are log-concave, then subadditivity of \(\psi_1^{-1}\circ \psi_2\) is equivalent to \(D_1\geq_{\partial S} D_2\,.\)
    \end{enumerate}
\end{proposition}

\begin{proof}
    The first statement is a well-known result given, e.g., in \cite[Theorem 4.4.2]{Nelsen-2006}.
    For the second statement, notice that the log-convexity of \(-\psi_1'\) and \(-\psi_2'\) yields that $D_1$ and $D_2$ are CI, see \eqref{lempdparch}, and from (i) one obtains $D_1\leq_{lo} D_2$.
    Hence, the claim follows from the equivalence of the lower orthant order and the Schur order for copula derivatives whenever the underlying copulas are CI, see Lemma \ref{lem:3.16}.
    The third statement follows similarly from \eqref{lempdparch_cd} and Lemma \ref{lem:3.16_cds}.
\end{proof}

~\\
~\\
~\\
\subsubsection{Extreme-values copulas}\label{secdeppropEVC}

\begin{table}[b!]
    \begin{center}
        \scalebox{0.8}{	
        \begin{tabular}{lllllcc}
            \toprule
            Type & Copula & $\phantom{-}$Parameters & Pickands dependence function \(A(t)\) & Special / Limiting cases \\
            \midrule
            EV & BB5 &  $\phantom{-}1\leq\theta,~0<\delta$ & $\left(t^\theta+(1-t)^\theta  \right.$ & $C^{\text{BB5}}_{\theta, 0} = C^{\text{GH}}_{\theta}$,  $C^{\text{BB5}}_{\theta, \infty} = M$,\\
            & & & $\left.-\left[(1-t)^{-\theta \delta}+t^{-\theta \delta}\right]^{-1 / \delta}\right)^{1 / \theta}$ & $C^{\text{BB5}}_{1, \delta} = C^{\text{Gal}}_{\delta}$ \\
            & Cuad.-Au. & $\phantom{-}0 \leq \delta \leq 1$ & $1 - \delta\min{\setof{1-t, t}}$ & $C^{\text{CA}}_{0} = \Pi$, $C^{\text{CA}}_{1} = M$ \\
            & Galambos & $\phantom{-}0<\delta$ & $1 - \inbrackets{t^{-\delta} + (1-t)^{-\delta}}^{-1/\delta}$ & $C^{\text{Gal}}_0=\Pi$, $C^{\text{Gal}}_{\infty}=M$ \\
            & Gum.-Hou. & $\phantom{-}1\leq\theta$  & $\inbrackets{t^\theta + (1-t)^\theta}^{1/\theta}$ & $C^{\text{GH}}_1 = \Pi$, $C^{\text{GH}}_\infty=M$ \\
            & Hüsl.-Rei. &$\phantom{-}0\leq\delta$ & $(1-t)\Phi(z_{1-t}) + t\Phi(z_t)$, & $C^{\text{HR}}_0=\Pi$, $C^{\text{HR}}_\infty=M$ \\
            & & & $z_t := \frac1\delta + \frac{\delta}2\log\of{\frac{t}{1-t}}$ & \\
            & Joe-EV & $\phantom{-}0\leq\alpha_1,\alpha_2\leq1$, & $1-\left\{\left[\alpha_1(1-t)\right]^{-\delta}+\left(\alpha_2 t\right)^{-\delta}\right\}^{-1 / \delta}$ & $C^{\text{JoeEV}}_{1, 1, \delta} = C^{\text{Gal}}_{\delta}$, \\
            & & $\phantom{-}0<\delta$ & & $C^{\text{JoeEV}}_{\alpha_1, 0, \delta} = C^{\text{JoeEV}}_{0, \alpha_2, \delta} = \Pi$ \\
            & & & & $C^{\text{JoeEV}}_{\alpha_1, \alpha_2, 0} = \Pi$ \\
            & & & & $C^{\text{JoeEV}}_{\alpha_1, \alpha_2, \infty} = C^{\text{MO}}_{\alpha_1, \alpha_2}$ \\
            & Marsh.-Ol. & $\phantom{-}0 \leq \alpha_1, \alpha_2 \leq 1$ & $\max{\setof{1-\alpha_1(1-t), 1- \alpha_2t}}$ & $C^{\text{MO}}_{0,0} = \Pi$, $C^{\text{MO}}_{1,1}=M$ \\
            & Tawn & $\phantom{-}0\leq\alpha_1,\alpha_2\leq1$, & $(1-\alpha_1)(1-t) + (1-\alpha_2)t $ & $C^{\text{Tawn}}_{\alpha_1, \alpha_2, 1} = C^{\text{Tawn}}_{0, 0, \theta} =\Pi$, \\
            & & $\phantom{-}1\leq\theta$ & $ + \inbrackets{\alpha_1(1-t)^{\theta} + (\alpha_2 t)^\theta}^{1/\theta}$ & $C^{\text{Tawn}}_{\alpha_1, \alpha_2, \infty} = C^{\text{MO}}_{\alpha_1, \alpha_2}$, \\
            & & & & $C^{\text{Tawn}}_{1, 1, \theta} = C^{\text{GH}}_{\theta}$ \\
            & t-EV & $-1<\rho<1,~0<\nu$ & $(1-t)T_{\nu + 1}(z_{1-t}) + t T_{\nu + 1}(z_t)$, & $C^{\text{tEV}}_{0, \rho} = C^{\text{MO}}_{\frac{\rho}{\sqrt{1-\rho^2}}, \frac{\rho}{\sqrt{1-\rho^2}}}$ \\
            & & & $z_t := \sqrt{\frac{1+\nu}{1-\rho^2}}\inbrackets{\of{\frac{t}{1-t}}^{1/\nu}-\rho}$ & $C^{\text{tEV}}_{\infty, \rho} = C^{\text{HR}}_{\rho}$ \\
            \midrule
            Ellip. & Gaussian & $-1 \leq \rho \leq 1$ & -- & $C^{\text{Gauss}}_{-1} = W, C^{\text{Gauss}}_{0} = \Pi$, \\
            & & & & $C^{\text{Gauss}}_{1}=M$ \\
            & Student-t & $-1 \leq \rho \leq 1,~0<\nu$ & -- & $C^{\text{t}}_{\nu, -1} = W$, $C^{\text{t}}_{\nu,1}=M$, \\
            & & & &  $C^{\text{t}}_{\infty,\rho} = C^{\text{Gauss}}_\rho$ \\
            & Laplace & $-1 \leq \rho \leq 1$ & -- & $C^{\text{Lap}}_{-1} = W$, $C^{\text{Lap}}_{1}=M$ \\
            \midrule
            Uncl. & Fréchet & $\phantom{-}0\leq\alpha,\beta,~\alpha+\beta \leq 1$ & -- & $C^{\text{Fré}}_{0, 1}=W$, $C^{\text{Fré}}_{0,0} = \Pi$, \\
            & & & & $C^{\text{Fré}}_{1, 0}=M$ \\
            & Mardia & $-1 \leq \theta \leq 1$ & -- & $C^{\text{Ma}}_{-1}=W$, $C^{\text{Ma}}_{0}=\Pi$, \\
            & & & & $C^{\text{Ma}}_{1}=M$ \\
            & FGM & $-1 \leq \theta \leq 1$ & -- & $C^{\text{FGM}}_0=\Pi$ \\
            & Plackett & $\phantom{-}0<\theta$ & -- & $C^{\text{Pl}}_0=W$, $C^{\text{Pl}}_1=\Pi$, \\
            & & &  & $C^{\text{Pl}}_{\infty}=M$ \\
            & Raftery & $\phantom{-}0 \leq \delta \leq 1$ & -- & $C^{\text{Ra}}_{0} = \Pi$, $C^{\text{Ra}}_{1} = M$ \\
            \bottomrule
        \end{tabular}
        }
    \end{center}
    \caption{
        Overview of elliptical, extreme-value and unclassified copula families, for which dependence properties are given in Table \ref{tab:non_arch_results}.
    }
    \label{tab:non_arch_overview}
\end{table}
 
\begin{figure}[tbp]
    \begin{center}
        \includegraphics[width=0.49\textwidth]{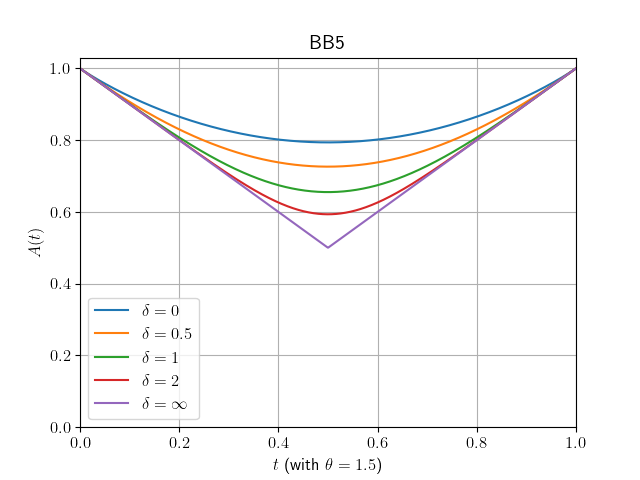}
        \includegraphics[width=0.49\textwidth]{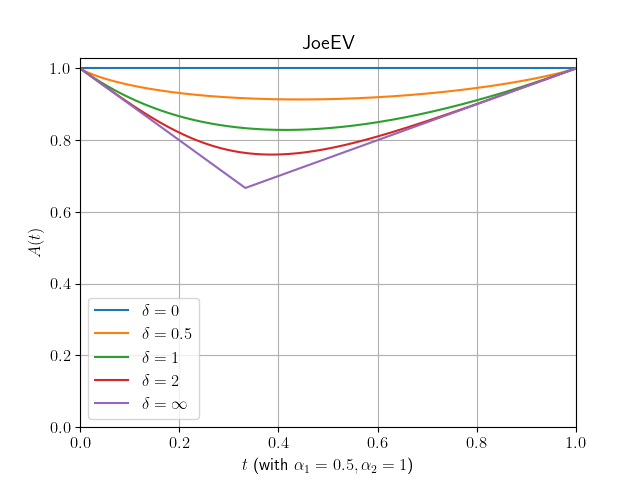}
        \caption{
            Pickands dependence functions for the two-parametric BB5 and for the three-parametric Joe's extreme-value copula family for some parameter choices.
            The Pickands functions are pointwise decreasing in their parameter.
            Hence, by Theorem \ref{thm:ordering_ev_copulas}, the BB5 and the Joe's extreme-value copula family are increasing with respect to the lower orthant order and the Schur order for copula derivatives. The skewed Pickands dependence function for Joe's extreme-value copula family reflects the fact that this copula family is not symmetric.
            For $\delta=0$ and $\delta=\infty$, the plots show the limiting cases from Table \ref{tab:non_arch_overview}.
            The BB5 copula converges uniformly to the Gumbel-Hougaard copula with parameter \(\theta\) as \(\delta \rightarrow 0\), and to the upper Fréchet copula $M$ as $\delta\rightarrow\infty$.
            Joe's extreme-value copula family converges uniformly to the independence copula $\Pi$ as $\delta\rightarrow 0$ and to the Marshall-Olkin copula with parameters $\alpha_1$ and $\alpha_2$ as $\delta\rightarrow\infty$.
        }
        \label{fig:joe_ev_pickand}
    \end{center}
\end{figure}

\begin{table}[tbp]
    \begin{center}
        \scalebox{0.8}{
        \begin{tabular}{rllllllll}
            \toprule 
            Type & Copula & CI/CD & $\TP$ & $\leq_{lo}$ & $\leq_{\partial S}$ & $\lambda_L$ & $\lambda_U$\\
            \midrule
            EV & BB5 & $\guparrow$ &\qmark & $\gnearrow$~in $\delta$ & $\gnearrow$~in $\delta$ &  $0$ & $2 - (2 - 2^{-1/\delta})^{1/\theta}$ \\
            & Cua.-Au. & $\guparrow$ & \xmark~iff $\delta > 0$ & $\gnearrow$ & $\gnearrow$ & $\1_{\setof{\delta = 1}}$ & $\delta$ \\
            & Galambos & $\guparrow$ &\qmark & $\gnearrow$ & $\gnearrow$ & $0$ & $2^{-1/\delta}$ \\
            & Gum.-Ho. & $\guparrow$ &\cmark~ & $\gnearrow$ & $\gnearrow$ & 0 & $2 - 2^{1/\theta}$ \\
            & Hüsl.-Re. & $\guparrow$ & $\text{\cmark}^*$ & $\gnearrow$ & $\gnearrow$ & $0$ & $2(1-\Phi(1/\delta))$ \\
            & Joe-EV & $\guparrow$ & $\text{\xmark}^*$ iff $C^{\text{JoeEV}}\neq\Pi$ & $\gnearrow$~in $\delta$ &$\gnearrow$~in $\delta$ & 0 & $(\alpha_1^{-\delta} + \alpha_2^{-\delta})^{-1/\delta}$ \\
            & Mar.-Ol. & $\guparrow$ &\xmark ~iff $\alpha_1\wedge \alpha_2 > 0$ & $\gnearrow$~in & $\gnearrow$~in & $\1_{\setof{\alpha_1=\alpha_2=1}}$ & $\min\setof{\alpha_1, \alpha_2}$ \\
            & & &$\text{\xmark}^*$~iff $\alpha_1\vee \alpha_2 > 0$ & $\phantom{\cmark}\alpha_1 = \alpha_2$ & $\phantom{\cmark}\alpha_1 = \alpha_2$ & & \\
            & Tawn & $\guparrow$ &$\text{\xmark}^*$ iff $C^{\text{Tawn}}\neq\Pi$ & $\gnearrow$~in $\theta$ & $\gnearrow$~in $\theta$ & 0 & $\left(\alpha_1 + \alpha_2\right.$ \\
            & & & & & & & $\left.- (\alpha_1^\theta + \alpha_2^{\theta})^{\frac1{\theta}}\right)$ \\
            & t-EV & $\guparrow$ & $\text{\xmark}^*$& $\gnearrow$~in $\rho$ & $\gnearrow$~in $\rho$ & $\1_{\setof{\rho = 1}}$ & $2(1-T_{\nu+1}(z_{1/2}))$ & \\
            \midrule 
            Ell. & Gauss & $\guparrow$~iff $\rho \geq 0$, & \cmark~iff $\rho \geq 0$ & $\gnearrow$ & $\gnearrow$~if $\rho \geq 0$, & $\1_{\setof{\abs{\rho} = 1}}$ & $\1_{\setof{\rho = 1 \vee \rho = -1}}$ \\
            & & $\gdownarrow$ iff $\rho \leq 0$ & & & $\gsearrow$~if $\rho \leq 0$ & & \\
            & Student-t & \xmark & \xmark & $\gnearrow$~in $\rho$ & \qmark & $2-2 t_{\nu+1}\of{c_{\rho,\nu}}$ & $2-2 t_{\nu+1}\of{c_{\rho,\nu}}$ \\
            & Laplace & \xmark~if $\rho \leq 0$ & \xmark & $\gnearrow$ & \qmark & \qmark & \qmark \\
            \bottomrule
            Uncl. & Fréchet & $\guparrow$~iff $\beta = 0$& \xmark~iff & $\gnearrow$ ~in & $\gnearrow$~for  & $\alpha$ & $\alpha$ \\
            & & $\gdownarrow$~iff $\alpha = 0$ & $(\alpha,\beta)\neq (1,0)$ & \phantom{\cmark}$\alpha$ and $\beta$ & \phantom{\cmark}$\alpha\wedge\beta=0$  & & \\
            & Mardia & $\guparrow$~iff $\theta = 1$ & \xmark~iff~$\theta <1$ & \xmark & $\gnearrow^*$~iff $\theta\geq 0$, & $\frac{\theta^{2} \left(\theta + 1\right)}{2}$ & $\frac{\theta^{2} \left(\theta + 1\right)}{2}$ \\
            & & $\gdownarrow$~iff $\theta = -1$ & & &$\gsearrow^*$~iff $\theta\leq 0$ &  & \\
            & FGM & $\guparrow$~iff $\theta \geq 0$, & \cmark~iff $\theta \geq 0$ & $\gnearrow$ & $\gnearrow$~if $\theta \geq 0$, & 0 & 0 \\
            & & $\gdownarrow$ iff $\theta \leq 0$ & & & $\gsearrow$~if $\theta \leq 0$ & & \\
            & Plackett &  $\guparrow$~iff $\theta \geq 1$, & \xmark~if $\theta > 2$, & $\gnearrow$ & $\gnearrow$~if $\theta \geq 1$, & 0 & 0 \\
            & & $\gdownarrow$ iff $\theta \leq 1$ &$\text{\cmark}^*$~if $\theta\in[1,2]$ & & $\gsearrow$~if $\theta \leq 0$ & & \\
            & Raftery & $\guparrow$ & \xmark~iff $\delta > 0$ & $\gnearrow$ & $\gnearrow$ & $2\frac{\delta}{\delta + 1}$ & 0 \\
        \end{tabular}
        }
    \end{center}
    \caption{
        Copula family properties for elliptical, extreme-value and unclassified copulas, where $c_{\rho,\nu}:= \sqrt{\frac{(\nu+1)(1-\rho)}{1+\rho}}$ for the tail-dependence coefficients of the Student-t copula family.
        Results marked with * are obtained from numerical checks and neither referenced nor proved.
    }
    \label{tab:non_arch_results}
\end{table}

The following theorem shows on the one hand the equivalence of the lower orthant order for bivariate extreme-value copulas and the reverse pointwise order of the associated Pickands dependence functions.
On the other hand, since bivariate extreme-value copulas are always CI, see \cite[Th\'{e}or\`{e}me 1]{Guillem-2000}, we also obtain the equivalence of the reverse pointwise order for the Pickands dependence functions with the Schur order for conditional distributions and the Schur order for copula derivatives.

\begin{theorem}[Ordering extreme-value copulas]\label{thm:ordering_ev_copulas}~\\
    Let \(D_1,D_2\in \cC_2\) be extreme-value copulas with Pickands dependence function \(A_1\) and \(A_2\,,\) respectively.
    Let \((U,V)\) and \((U',V')\) be bivariate random vectors with \(F_{U,V}=D_1\) and \(F_{U',V'}=D_2\,.\)
    Then, the following statements are equivalent:
    \begin{enumerate}[(i)]
        \item \(A_1(t)\geq A_2(t)\) for all \(t\in (0,1)\,,\)
        \item \(D_1\leq_{lo} D_2\,,\)
        \item \(D_1\leq_{\partial S} D_2\,,\)
        \item \((V|U)\leq_S (V'|U')\,,\)
        \item \((U|V)\leq_S (U'|V')\,.\)
    \end{enumerate}
\end{theorem}

\begin{proof}
    "$(i)\Rightarrow(ii)$": Assume that $A_1(t) \geq A_2(t)$ for all \(t\in (0,1)\,.\)
    Then, for $u,v \in (0,1)$,
    it is $\ln(v)/\ln(uv)\in(0,1)$ and thus $A_1(\ln(v)/\ln(uv)) \geq A_2(\ln(v)/\ln(uv))$.
    Since $0\leq uv \leq 1$, it follows from \eqref{defEVC} that 
    \[
        D_1(u,v)
        = (uv)^{A_1\of{\ln(v)/\ln(uv)}}
        \leq (uv)^{A_2\of{\ln(v)/\ln(uv)}}
        = D_2(u,v)
    ,\]
    which shows $D_1 \leq_{lo} D_2$. \\
    "$(ii)\Rightarrow(i)$": Assume that $D_1 \leq_{lo} D_2$ and let $v:=t$ and $u:=t^{1/t-1}$ for $t\in(0, 1)$.
    This choice satisfies $u,v\in(0, 1)$ and $\ln(v)/\ln(uv)=t\,.$ It follows from \eqref{defEVC} that
    \[
        (uv)^{A_1\of{t}}
        = C_1(u,v)
        \leq C_2(u,v)
        = (uv)^{A_2\of{t}}
    ,\]
    and hence $A_1(t) \geq A_2(t)$. \\
    "$(ii)\Rightarrow(iii)$": Since bivariate extreme-value copulas are always CI, see \cite{Guillem-2000}, the statement follows from Lemma \ref{lem:3.16} (ii). \\
    "$(iii)\Rightarrow(iv), (v)$" is a consequence of the definition of $\leq_{\partial S}$ and Lemma \ref{lemcharschur}. \\
    "$(iv)\Rightarrow(ii)$" follows from first applying Lemma \ref{lemcharschur} and then Lemma \ref{lem:3.16} (i).
    \\
    "$(v)\Rightarrow(ii)$": Denote by $D^T_i(u, v) := D_i(v, u)\,,$ $u,v \in[0,1]^2\,,$ $i\in \{1,2\}\,,$ the transposed copula of \(D\,.\)
    Then, again from Lemma \ref{lemcharschur} and Lemma \ref{lem:3.16} (i), we get $D^T_1\leq_{lo} D^T_2$, which is equivalent to $D_1\leq_{lo} D_2$.
\end{proof}

\begin{remark}
For various well-known families of extreme-value copulas, it can easily be verified that the associated Pickands dependence functions are pointwise ordered.
Hence, Theorem \ref{thm:ordering_ev_copulas} provides a simple characterization for ordering extreme-value copulas with respect to the lower orthant order and the Schur orders, see Table \ref{tab:non_arch_results} and Figure \ref{fig:joe_ev_pickand}.
In particular, if the Pickands dependence functions are ordered, then Kendall's tau, Spearman's rho, the tail-dependence coefficients and Chatterjee's xi are reverse ordered, see Section \ref{secmoca}.
We refer to \cite{Cap-1997,Cap-2000} for a dependence ordering that is based on a probability transform and that is, for extreme-value copulas, strictly weaker than \(\leq_{lo}\) and equivalent to the ordering of Kendall's tau.
\end{remark}

\subsubsection{Elliptical copulas}\label{secpropellcop}

For the multivariate normal distribution, positive dependence concepts are characterized in terms of the correlation matrix, see \cite{Rueschendorf-1981b}.
More generally, for elliptical distributions, positive dependence properties also depend on the elliptical generator.
In the case of bivariate elliptical distributions the \(\TP\)-property of an elliptical copula \(C_\rho\,,\) \(\rho\in (-1,1)\,,\) with density generator $g$ is fulfilled if and only if

\begin{align}
    \begin{aligned}\label{eqcharell}
    -\frac{\rho}{1+\rho} \leq \inf _{t \in T} \frac{t \phi^{\prime \prime}(t)}{\phi^{\prime}(t)} \leq \sup _{t \in T} \frac{t \phi^{\prime \prime}(t)}{\phi^{\prime}(t)} \leq \frac{\rho}{1-\rho} \quad &\text{for } \phi'(t) \neq 0\,,\\
    \phi''(t)=0 \quad &\text{for } \phi'(t)=0\,,
    \end{aligned}
\end{align}
see \cite[Proposition 1.2]{Abdous-2005},
where $t\mapsto \phi(t):=\log(g(t))$ is assumed to be twice differentiable and where $T=\left\{t \in \R_+: \phi^{\prime}(t)<0\right\}$.
In particular, if an elliptically contoured distribution is \(\TP\) for $\rho=0$, then it is Gaussian.
Due to \eqref{implposdepcon}, the criterion in \eqref{eqcharell} is sufficient for an elliptical distribution being CI.
However, we are not aware of a necessary condition for CI based on the elliptical generator.
We obtain the following ordering result for elliptical families.

\begin{proposition}[Ordering elliptical copulas]\label{prop:ordering_elliptical_copulas}~\\
    Let \((C_\rho)_{\rho\in [-1,1]}\) be a family of elliptical copulas with density generator $g$.
    Then, the following statements hold true.
    \begin{enumerate}[(i)]
        \item \label{ord_ell_cop_1} \(\rho\leq \rho'\) if and only if \(C_\rho\leq_{lo} C_{\rho'}\,.\)
        \item If \(g\) satisfies \eqref{eqcharell} for \(|\rho|\in [0,1)\,,\) then \(|\rho| \leq |\rho'|\) implies \(C_{\rho}\leq_{\partial S} C_{\rho'}\,.\)
    \end{enumerate}
\end{proposition}
\begin{proof}
    The first statement is an immediate consequence of \cite[Theorem 2.21]{Joe-1997}.
    For the second statement, consider first the case where $ |\rho'|<1$ and note that if $C_{|\rho|}$ satisfies \eqref{eqcharell}, then $C_{|\rho'|}$ also satisfies \eqref{eqcharell}, as the bounds are less restrictive for $|\rho'|$.
    Since \eqref{eqcharell} characterizes $\TP$, it follows that $C_{|\rho|}$ and $C_{|\rho'|}$ are CI.
    Noticing that an elliptical copula is symmetric in the sense that $C_{\rho}$ is CIS if and only if $C_{-\rho}$ is CDS, it follows that $C_{\rho}$ and $C_{\rho'}$ are either both CI or both CD. 
    Hence, together with (i), it follows \(C_{\rho}\leq_{\partial S} C_{\rho'}\) from Lemma \ref{lem:3.16} (ii) or Lemma \ref{lem:3.16_cds} (ii), respectively.
    In the case where $\abs{\rho'}=1$, we have $C_{\rho'}\in\setof{M, W}$, which are extremal elements in the Schur order for copula derivatives, so $C_{\rho} \leq_{\partial S} C_{\rho'}$ also in this case.
\end{proof}

\subsection{Monotonicity properties of measures of association}\label{secmoa}

\begin{figure}[tbp]
    \centering
    \includegraphics[width=0.29\textwidth]{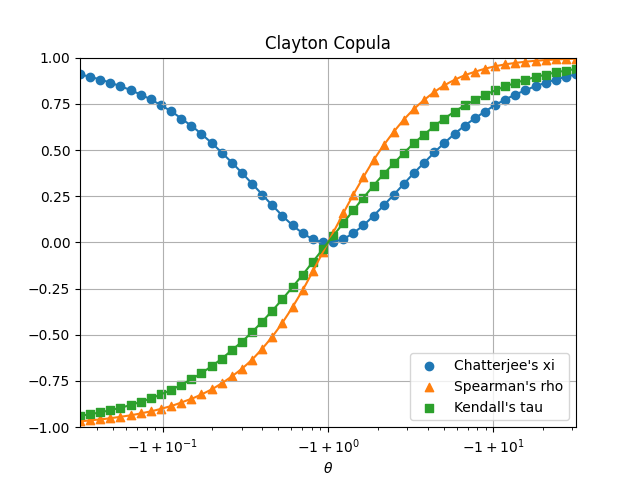}
    \includegraphics[width=0.29\textwidth]{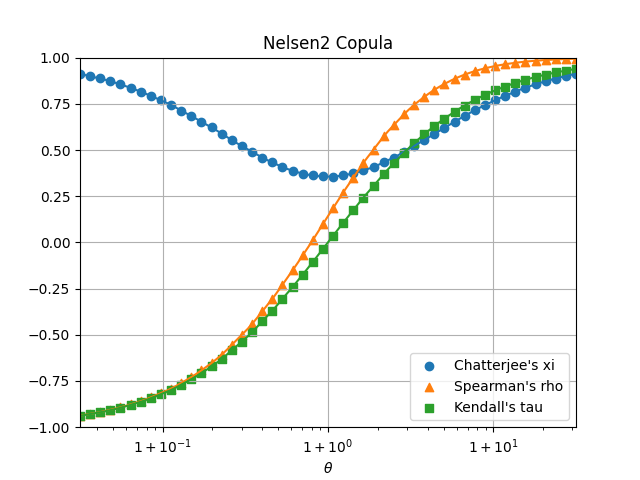}
    \includegraphics[width=0.29\textwidth]{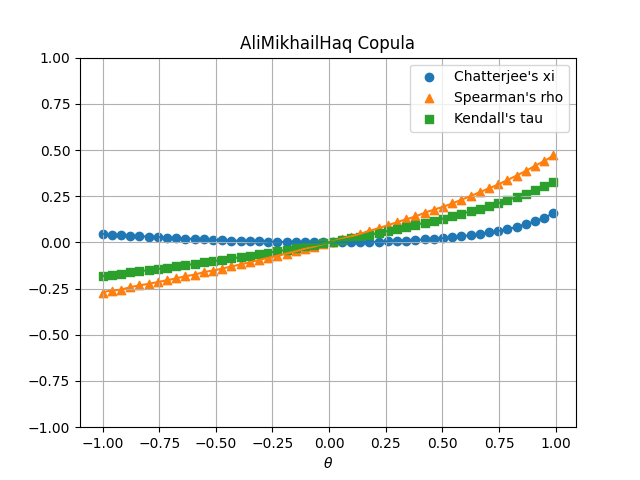} \allowdisplaybreaks\\[-0.6em]
    \includegraphics[width=0.29\textwidth]{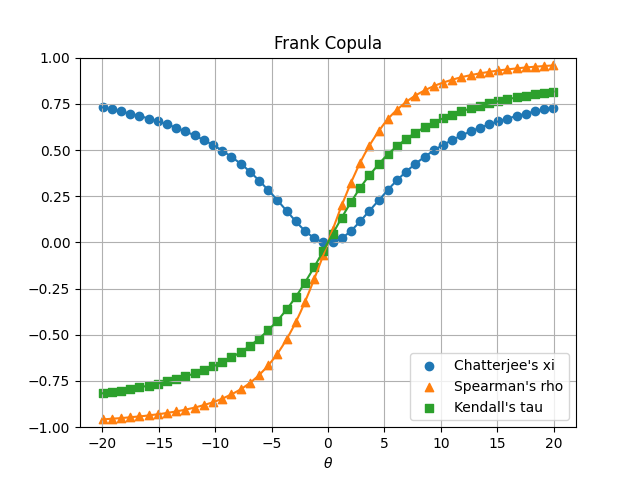}
    \includegraphics[width=0.29\textwidth]{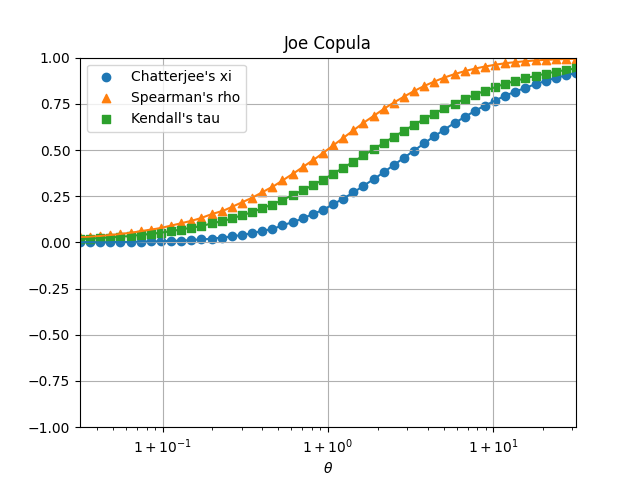} 
    \includegraphics[width=0.29\textwidth]{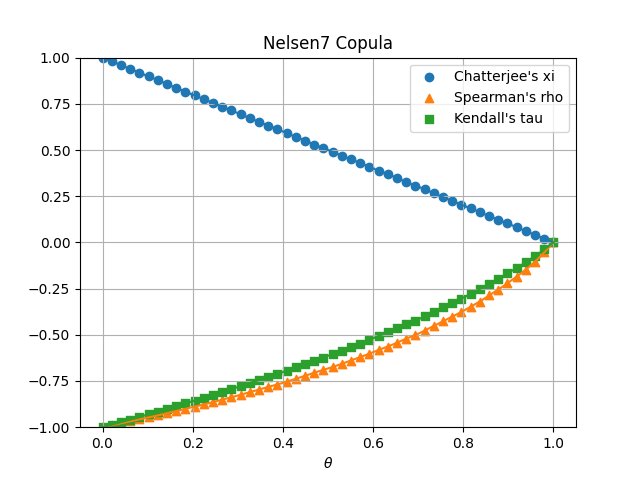}\allowdisplaybreaks\\[-0.6em]
    \includegraphics[width=0.29\textwidth]{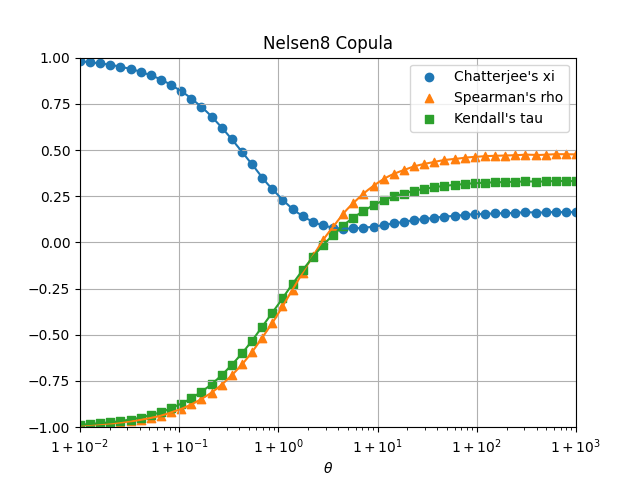} 
    \includegraphics[width=0.29\textwidth]{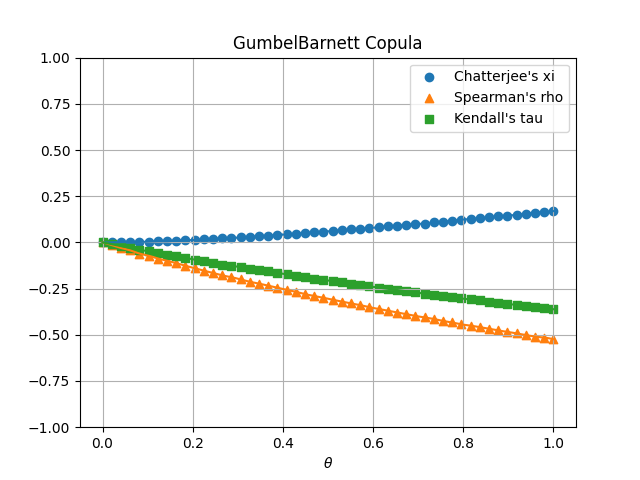} 
    \includegraphics[width=0.29\textwidth]{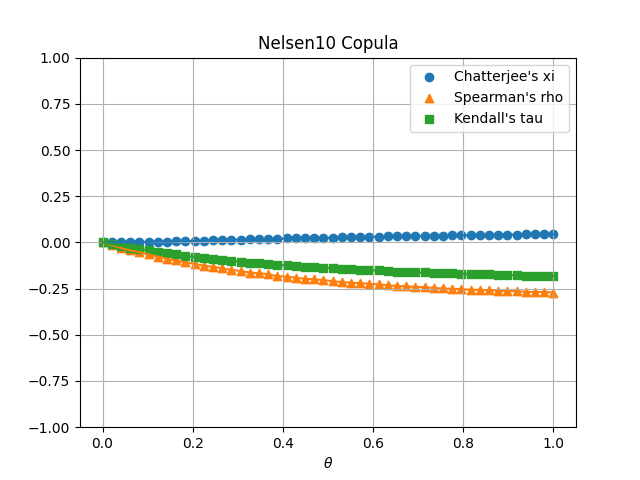}\allowdisplaybreaks\\[-0.6em]
    \includegraphics[width=0.29\textwidth]{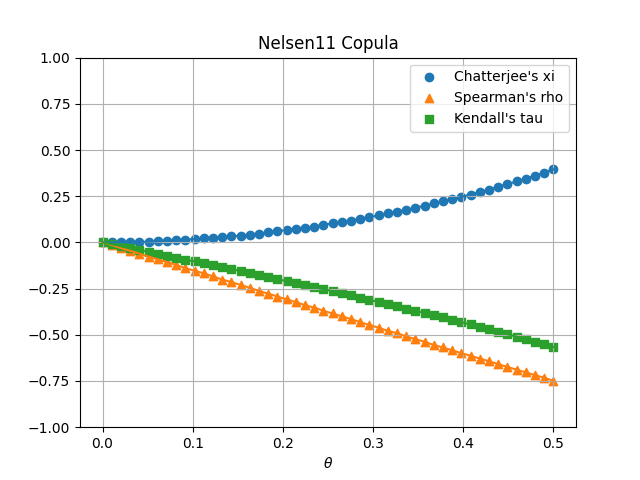} 
    \includegraphics[width=0.29\textwidth]{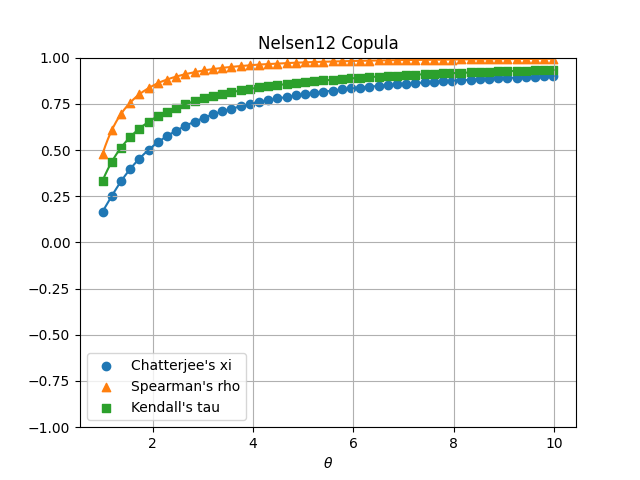}
    \includegraphics[width=0.29\textwidth]{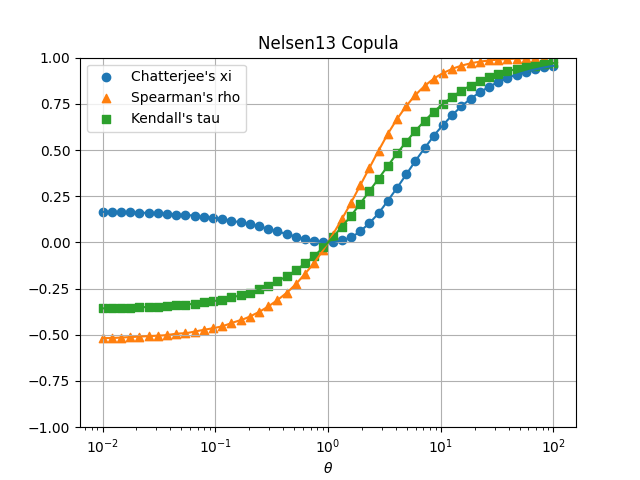}\allowdisplaybreaks\\[-0.6em]
    \includegraphics[width=0.29\textwidth]{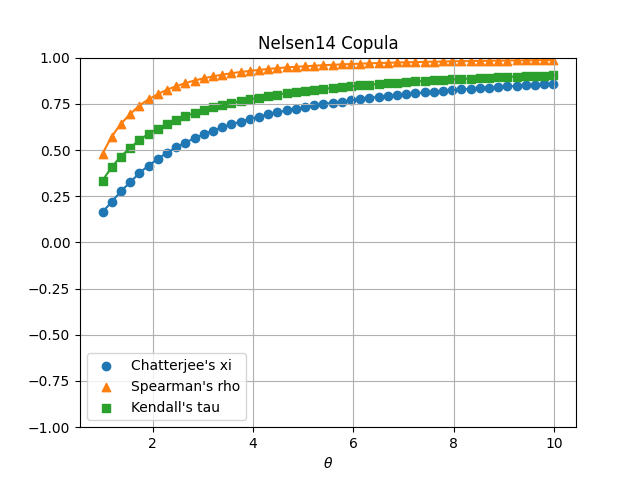}
    \includegraphics[width=0.29\textwidth]{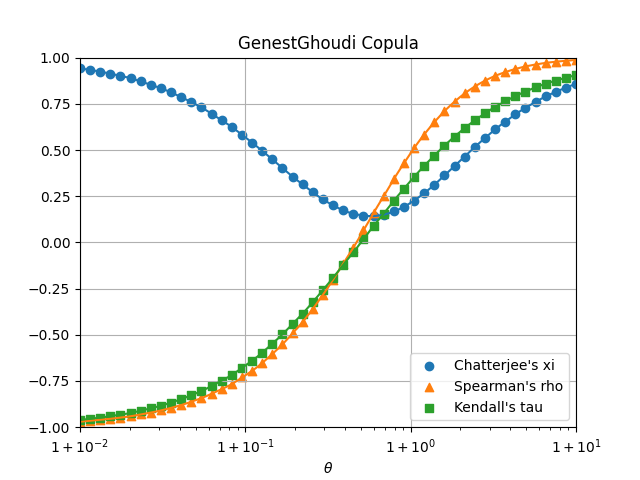} 
    \includegraphics[width=0.29\textwidth]{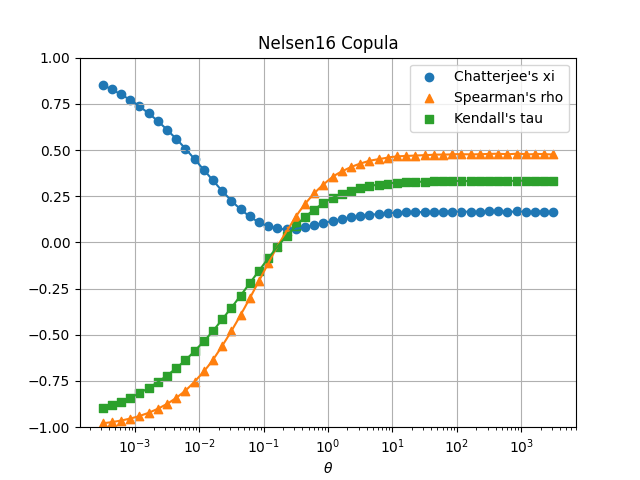}\allowdisplaybreaks\\[-0.6em]
    \includegraphics[width=0.29\textwidth]{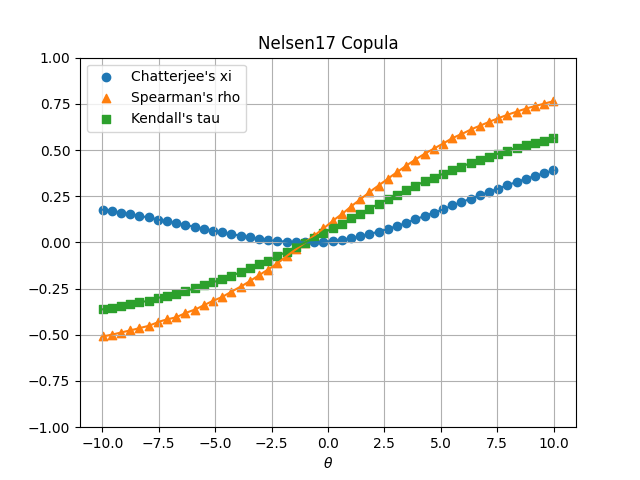}
    \includegraphics[width=0.29\textwidth]{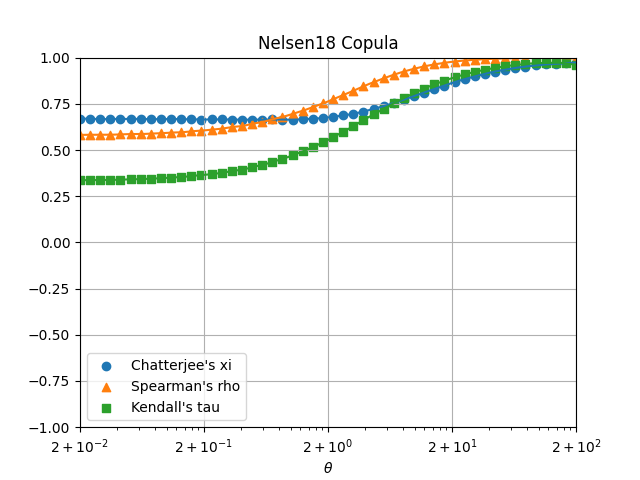} 
    \includegraphics[width=0.29\textwidth]{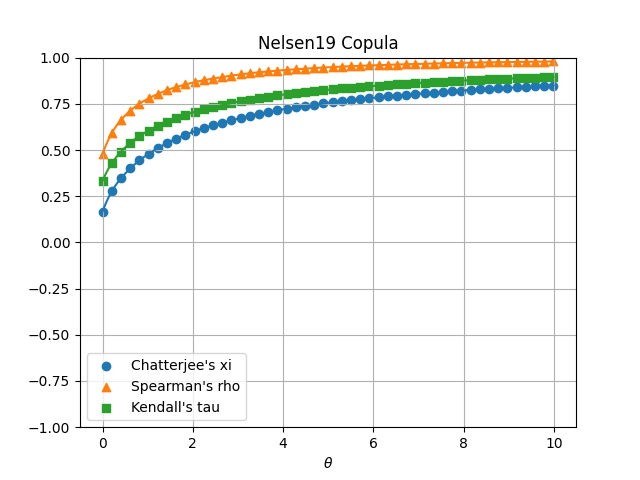}\allowdisplaybreaks\\[-0.6em]
    \includegraphics[width=0.29\textwidth]{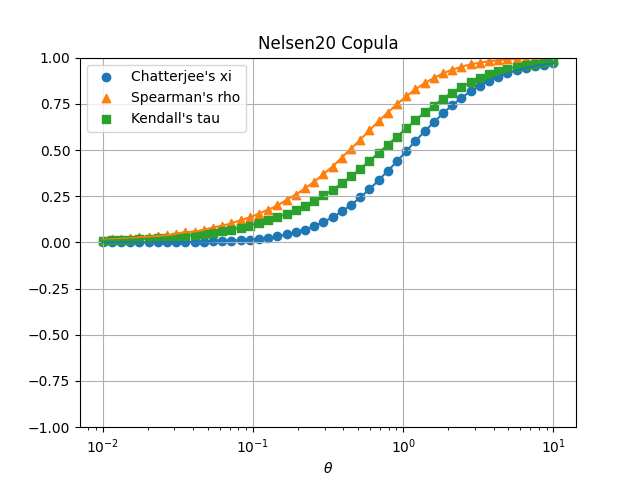}
    \includegraphics[width=0.29\textwidth]{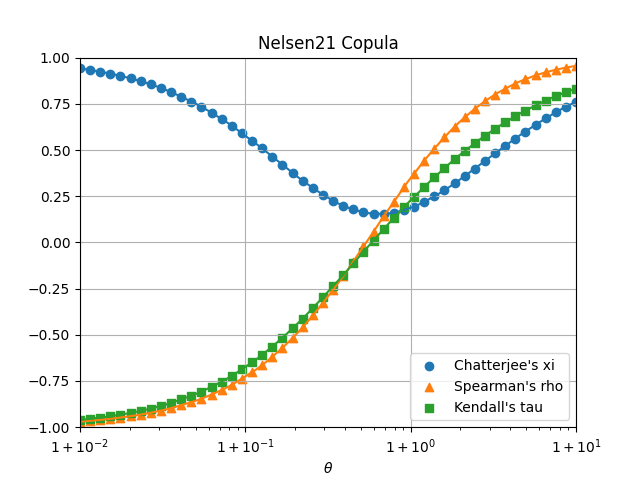} 
    \includegraphics[width=0.29\textwidth]{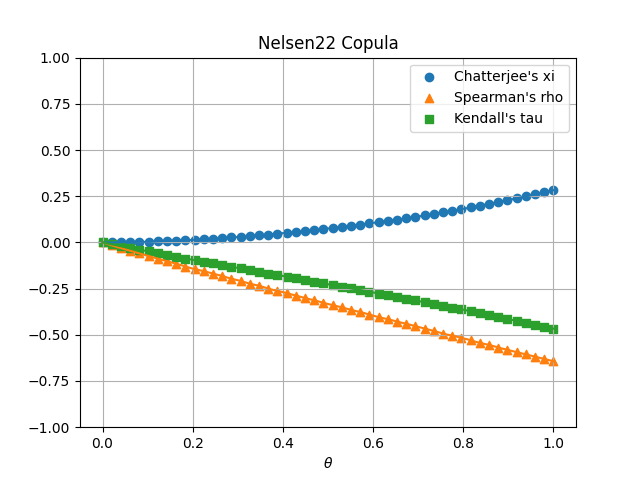}
    \caption{
        Chatterjee's xi, Spearman's rho and Kendall's tau in dependence of the parameter for the  Archimedean copula families in Table \ref{tab:copulas}.
        The graph for the Gumbel-Hougaard copula family is given in Figure \ref{fig:ev_chatterjee_rho}.
    }
    \label{fig:arch_chatterjee_rho_tau}
\end{figure}

\begin{figure}[tbp]
    \begin{centering}
        \includegraphics[width=0.29\textwidth]{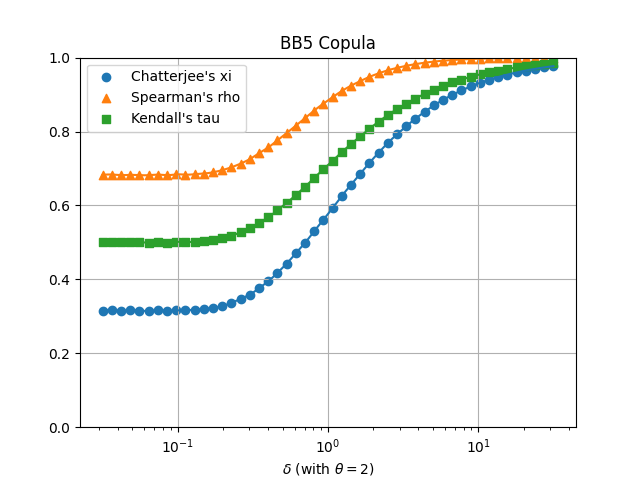}
        \includegraphics[width=0.29\textwidth]{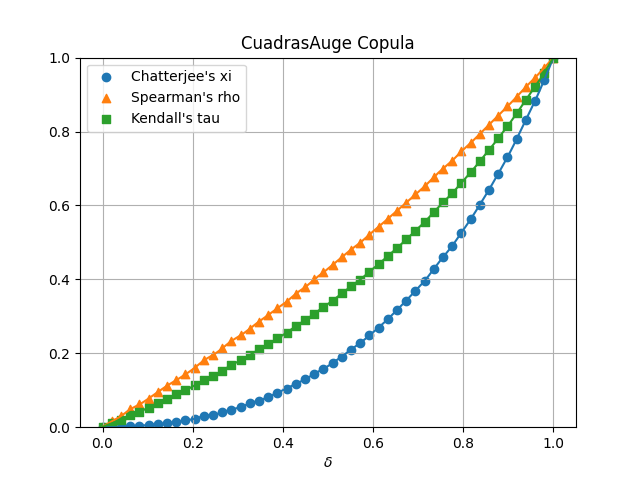}
        \includegraphics[width=0.29\textwidth]{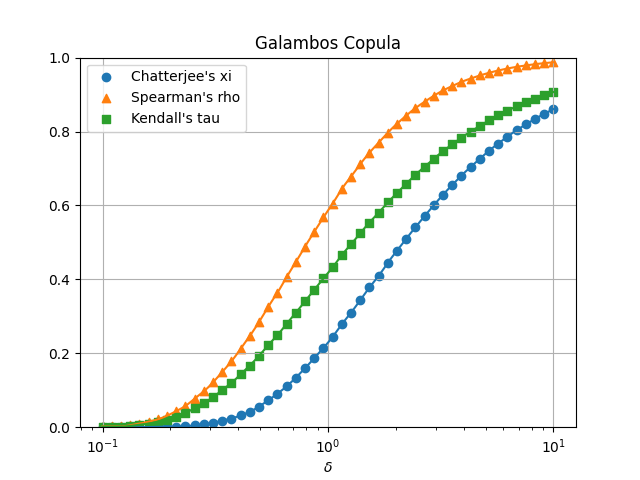} \\[-0.2em]
        \includegraphics[width=0.29\textwidth]{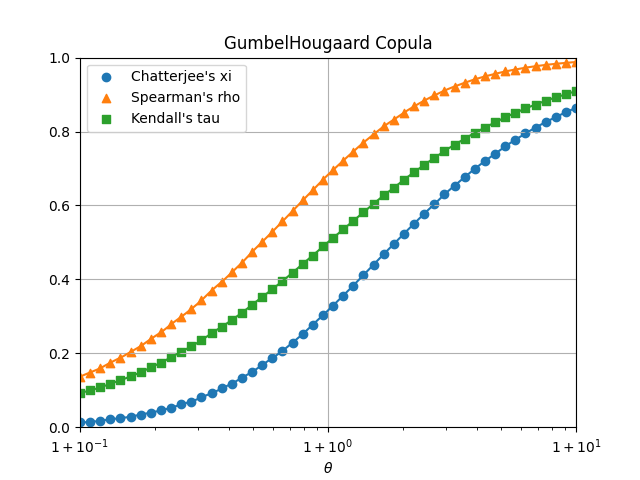}
        \includegraphics[width=0.29\textwidth]{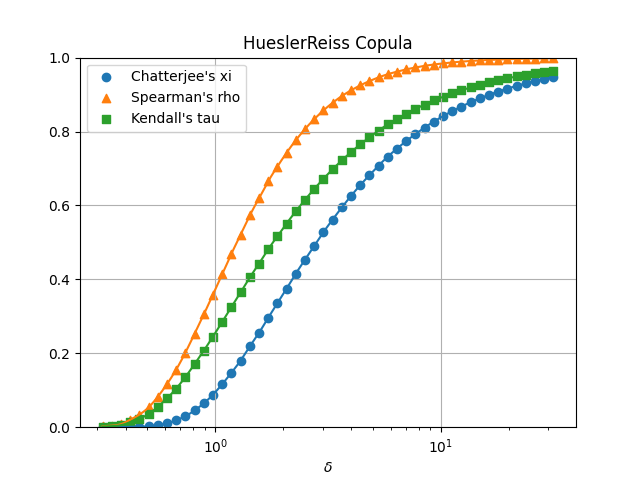}
        \includegraphics[width=0.29\textwidth]{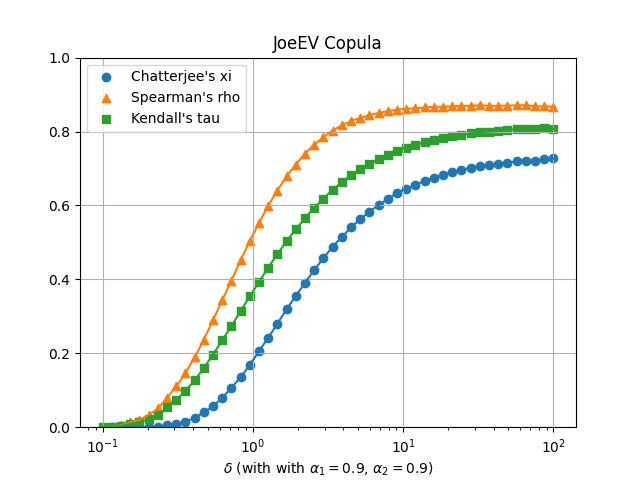}\\[-0.2em]
        \includegraphics[width=0.29\textwidth]{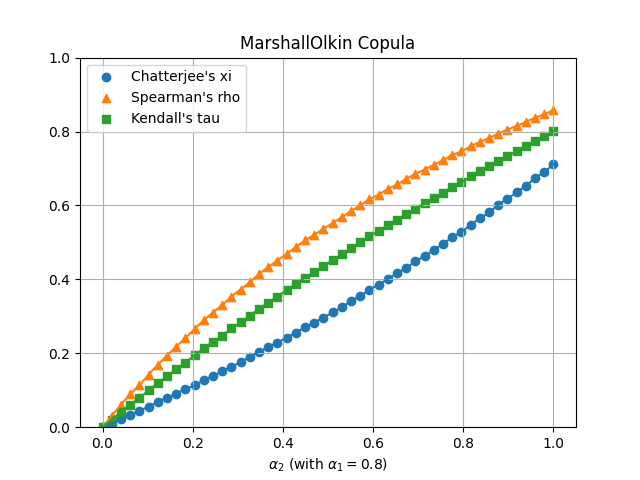}
        \includegraphics[width=0.29\textwidth]{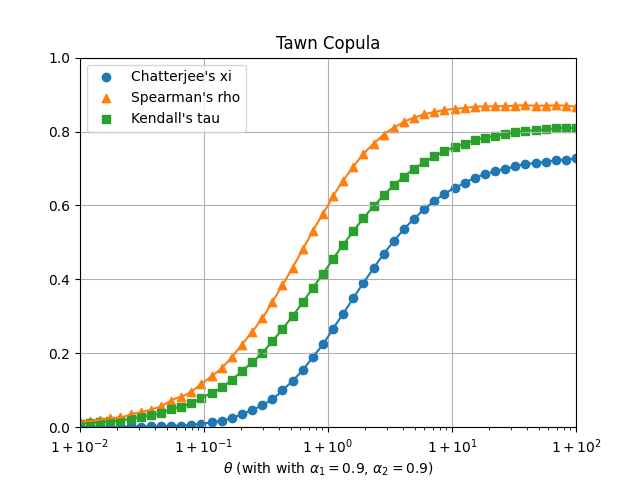} 
        \includegraphics[width=0.29\textwidth]{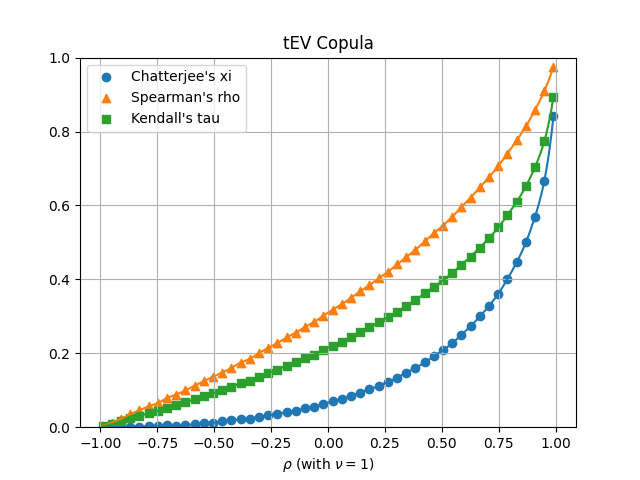} 
        \centering
        \caption{
            Chatterjee's xi, Spearman's rho and Kendall's tau for the extreme-value copula families in Table \ref{tab:copulas}.
            We consider special cases for multi-parameter families as stated in the $x$-axis labels.
        }
        \label{fig:ev_chatterjee_rho}
        ~\\[1em]

        \includegraphics[width=0.29\textwidth]{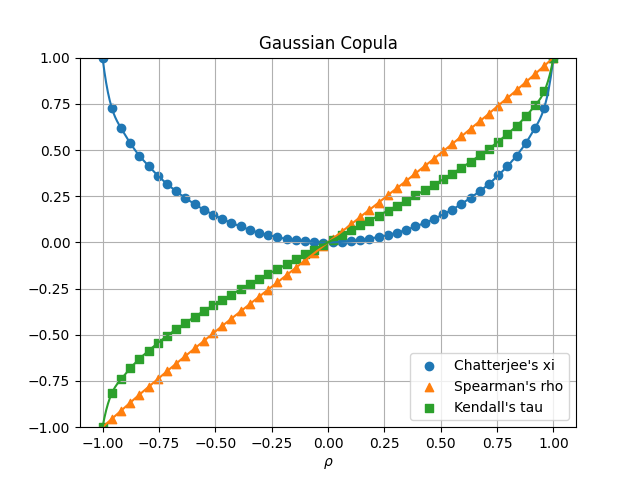}
        \includegraphics[width=0.29\textwidth]{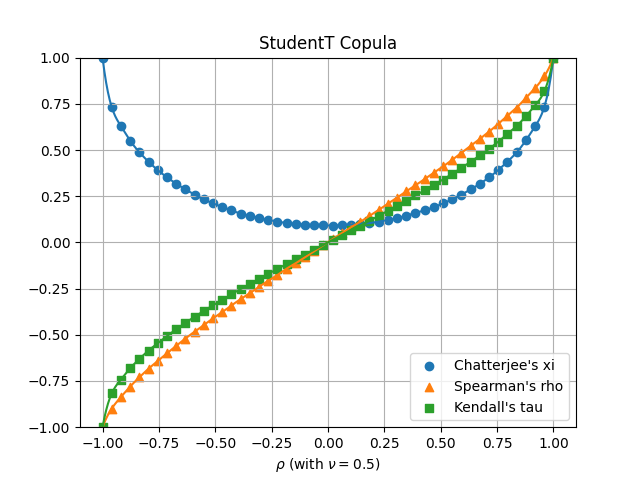}
        \includegraphics[width=0.29\textwidth]{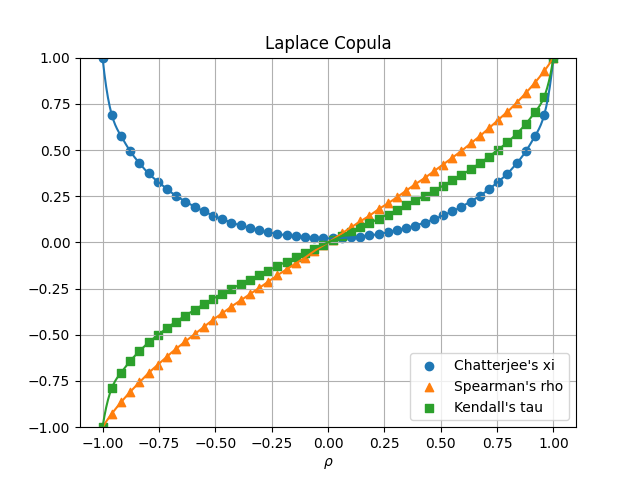}
        \centering
        \caption{
            Chatterjee's xi, Spearman's rho and Kendall's tau for the elliptical copula families in Table \ref{tab:copulas}.
            We consider $\nu=\frac12$ for the Student-t copula family.
        }
        \label{fig:elliptical_chatterjee_rho}
        ~\\[1em]
        \includegraphics[width=0.29\textwidth]{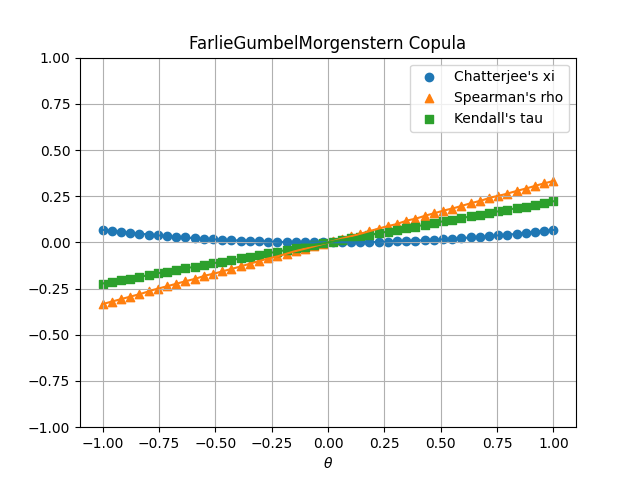}
        \includegraphics[width=0.29\textwidth]{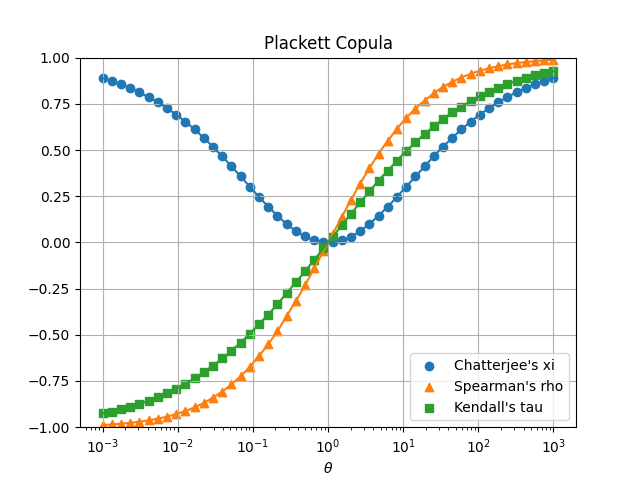} 
        \includegraphics[width=0.29\textwidth]{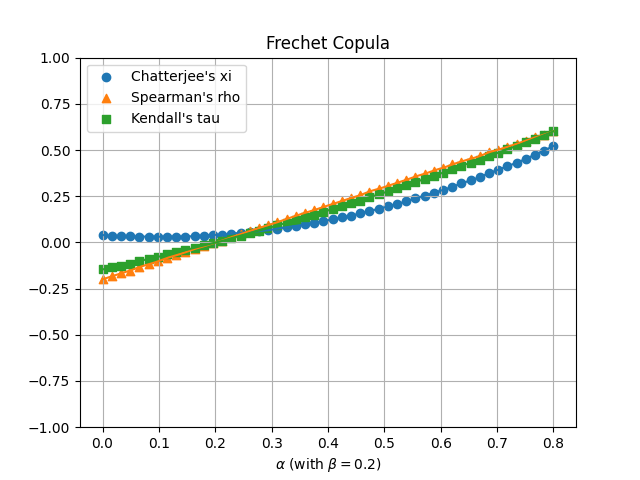} \\[-0.2em]
        \includegraphics[width=0.29\textwidth]{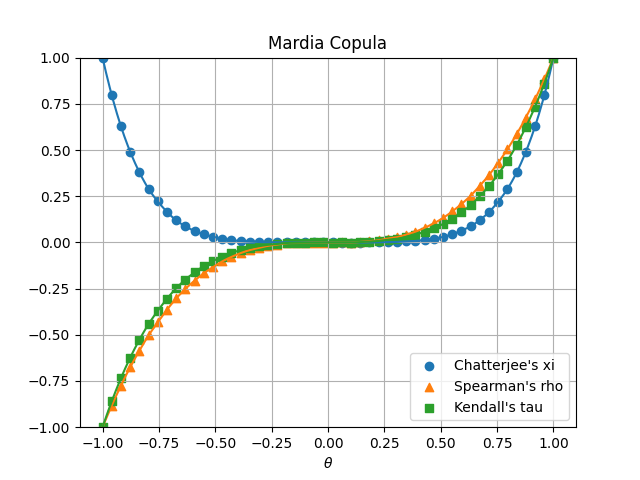}
        \includegraphics[width=0.29\textwidth]{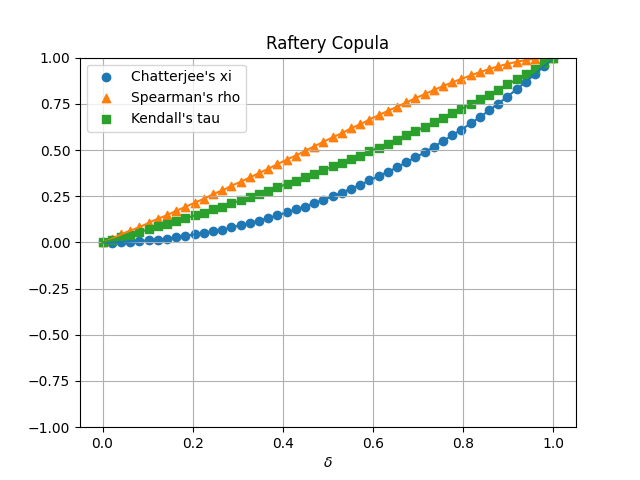}
        \centering
        \caption{
            Chatterjee's xi, Spearman's rho and Kendall's tau for the unclassified copula families  in Table \ref{tab:copulas}.
            We consider $\beta=0.2$ for the Fréchet copula family.
        }
        \label{fig:unclassified_chatterjee_rho_tau}
    \end{centering}
\end{figure}

\begin{table}[tbp]
    \begin{center}
        \scalebox{0.8}{
        \begin{tabular}{lllll}
            \toprule
            Type & Family Name & Chatterjee's xi & Spearman's rho & Kendall's tau \\
            \midrule
            Arch. & Clayton  & $\left(6 \int\limits_{0}^{1} {{}_{2}F_{1}\left(\begin{matrix} \frac{1}{\theta}, \frac{2\theta + 2}{\theta} \\ 1 + \frac{1}{\theta} \end{matrix}\middle| {1 - v^{- \theta}} \right)}\, \de v\right.$ & & $\frac{\theta}{\theta + 2}$ \\
            & & $\bigg.\quad-2\bigg)$ if $\theta > 0$ & & \\
            & AMH & $\left(\frac{3}{\theta} - \frac{\theta}{6} - \frac{2}{3} - \frac{2}{\theta^{2}}\right.$ &$\left(\frac{12(1+\theta)\int_1^{1-\theta} \frac{\ln(t)}{1-t}\de t}{\theta^2}\right.$ & $1 - \frac{2}{3\theta}- \frac{2(1-\theta)^2\ln(1-\theta)}{3\theta^2}$\\
            & & $\left.\quad - \frac{2 \left(\theta - 1\right)^{2} \log{\left(1 - \theta \right)}}{\theta^{3}}\right)$ &$\left.\quad - \frac{24(1-\theta)\ln(1-\theta)}{\theta^2} - \frac{3(\theta + 12)}{\theta}\right)$ & \\
            & Frank  & & $1 - \frac{12}{\theta}(D_1(\theta) - D_2(\theta))$ & $1 - \frac{4}\theta\of{1-D_1(\theta)}$ \\
            & Nelsen7  & $1-\theta$ & $\frac{9 \theta^{2} - 6 \theta - 6 \left(\theta - 1\right)^{2} \log{\left(1 - \theta \right)}}{\theta^{3}} - 3$ & $\frac{2\left(\theta^2 - \theta - \left(\theta - 1\right)^{2} \log{\left(1 - \theta \right)}\right)}{\theta^{2}}$ \\
            & Gumb.-Barn. & $\frac{3}{4 \theta} \mathrm{e}^{\frac{3}{2 \theta}} \mathrm{E}_1\left(\frac{3}{2 \theta}\right)+\frac{\theta}{3}-\frac{1}{2}$ & & \\
            \midrule
            EV & Cuadr.-Augé & $\frac{\delta^{2}}{2 - \delta}$ & $\frac{3\delta}{4-\delta}$ & $\frac{\delta}{2-\delta}$ \\
            & Gumb.-Houg.  & & $\frac{12}{\theta} \int_0^1 \frac{[t(1-t)]^{\frac{1}{\theta}-1}}{\left[1+t^{\frac{1}{\theta}}+(1-t)^{\frac{1}{\theta}}\right]^2} d t-3$ & $\frac{\theta-1}{\theta}$ \\
            & Marsh.-Olk. & $\frac{2 \alpha_{1}^{2} \alpha_{2}}{3 \alpha_{1} + \alpha_{2} - 2 \alpha_{1} \alpha_{2}}$ & $\frac{3 \alpha_1 \alpha_2}{2 \alpha_1-\alpha_1 \alpha_2+2 \alpha_2}$ & $\frac{\alpha_1 \alpha_2}{\alpha_1 - \alpha_1 \alpha_2+\alpha_2}$ \\
            \midrule
            Ellip. & Gaussian & $\frac{3}{\pi} \arcsin \left(\frac{1}{2}+\frac{\rho^2}{1+\rho}\right)-0.5$ & $\frac6{\pi}\arcsin(\rho/2)$ & $\frac2{\pi}\arcsin(\rho)$ \\
            & Student-t &  & $\frac{6}{\pi} \mathrm{E}_{\tilde{V_1}}\{\arcsin (\mathrm{r} \tilde{V_1})\}$ & $\frac2{\pi}\arcsin(\rho)$ \\
            & Laplace &  & $\frac{6}{\pi} \mathrm{E}_{\tilde{V_2}}\{\arcsin (\mathrm{r} \tilde{V_2})\}$ & $\frac2{\pi}\arcsin(\rho)$ \\
            \midrule
            Uncl. & Fréchet & $(\alpha - \beta)^2 + \alpha \beta$ & $\alpha-\beta$ & $\frac{(\alpha-\beta)(\alpha+\beta+2)}{3}$ \\
            & Mardia & $\frac{\theta^{4} \left(3 \theta^{2} + 1\right)}{4}$ & $\theta^3$ & $\frac{\theta^{3} \left(\theta^{2} + 2\right)}{3}$ \\
            & FGM & $\theta^2/15$ & $\theta/3$ & $2\theta/9$\\
            & Plackett &  & $\frac{\theta + 1}{\theta - 1} - 2\frac{2\theta}{(\theta-1)^2}\ln(\theta)$ & \\
            & Raftery & & $\delta\frac{4-3\delta}{(2-\delta)^2}$ & $\frac{2\delta}{3-\delta}$ \\
            \bottomrule
        \end{tabular}
        }
    \caption{
        Closed-form expressions if available for Chatterjee's xi, Spearman's rho and Kendall's tau for different copula families, where
        $D_k(x) := \frac{k}{x^k} \int_0^x \frac{t^k}{e^t-1}\de t$ and $\mathrm{E}_1(x):=\int_1^{\infty} \mathrm{e}^{-x s} / s \mathrm{~d} s$.
        Further, ${}_2F_1$ is the hypergeometric function defined in \eqref{frm:hypergeometric_function}, and $\tilde{V}_1$, $\tilde{V}_2$ are random variables with a mixing density given in \cite[Proposition 1]{heinen2020spearman}.
    }
    \label{tab:all_families_rho_and_tau}
\end{center}
\end{table}

It is well-known that Kendall's tau, Spearman's rho as well as the lower and upper tail dependence coefficients are increasing with respect to the lower orthant order, see Lemmas \ref{lemmoc} and \ref{lem:tail_consistent_with_lo}.
By Lemmas \ref{lemcharschur} and \ref{lemconS}, we know that Chatterjee's rank correlation is increasing with respect to the Schur order for conditional distributions/copula derivatives.
From Tables \ref{tab:arch_results} and \ref{tab:non_arch_results}, we see that many well-known bivariate copula families exhibit monotonicity properties with respect to the lower orthant order and the Schur order for copula derivatives.
Consequently, these ordering properties explain the monotonicity of Kendall's tau, Spearman's rho and Chatterjee's xi for many copula (sub-)families in Figures\footnote{Each plot has $50$ data points per measure of association, and each data point is estimated by sampling from the copula one million times.} \ref{fig:arch_chatterjee_rho_tau}--\ref{fig:unclassified_chatterjee_rho_tau}.
For example, we know from Table \ref{tab:arch_results} that the Clayton copulas \((C_{\theta}^{\text{Cl}})_{\theta\in [-1,\infty)}\) are increasing in their parameter with respect to \(\leq_{lo}\) on the entire parameter space \([-1,\infty)\) and increasing/decreasing with respect to \(\leq_{\partial S}\) whenever the parameter \(\theta\) is non-negative/non-positive.
Hence, Kendall's tau and Spearman's rho are both increasing in the Clayton copula parameter \(\theta\) while Chatterjee's xi is decreasing in \(\theta\) for \(\theta\leq 0\) and increasing in \(\theta\) for \(\theta\geq 0\,,\) see Figure \ref{fig:arch_chatterjee_rho_tau}. \\
We also see that Kendall's tau, Spearman's rho, and Chatterjee's xi are all continuous in the underlying copula family parameters, which is a consequence of continuity of the copulas in their parameter with respect to uniform convergence and weak conditional convergence, respectively, see, e.g., \cite{Kasper-2021}.
In particular, if the underlying copulas converge to the lower/upper Fr\'{e}chet copula, Kendall's tau and Spearman's rho converge to \(-1\)/\(+1\) while Chatterjee's rank correlation converges to \(+1\) in both cases.
In the case where the underlying copula is the independence copula, all three measures attain the value $0$.
For elliptical copulas, this can only be achieved by the Gaussian copula with parameter \(0\,,\) even though the plots in Figure \ref{fig:elliptical_chatterjee_rho} look very similar.
We further observe from the plots that Kendall's tau, Spearman's rho, and Chatterjee's xi are often ordered in the same way for fixed copula parameters.
In particular, Figure \ref{fig:ev_chatterjee_rho} suggests that \(\rho_S(C_\theta)\geq \tau(C_\theta)\geq \xi(C_\theta)\) for all extreme-value copulas and all parameters \(\theta\,.\)
The first inequality is generally correct for CI copulas, i.e.,
\begin{align*}
    C \quad \text{CI} \quad \Longrightarrow \quad \rho_S(C)\geq \tau(C)\geq 0\,,
\end{align*}
see \cite[Theorem 5.2.8]{Nelsen-2006}.
The plots and simulations also suggest that \(\tau(C)\geq \xi(C)\geq 0\) if the underlying copula is CIS.
However, we are not aware of a proof for this conjecture. \\
While closed-form expressions can easily be determined for the tail-dependence coefficients, see Tables \ref{tab:arch_results} and \ref{tab:non_arch_results}, closed-form formulas are generally not applicable for Chatterjee's rank correlation, Kendall's tau and Spearman's rho.
In Table \ref{tab:all_families_rho_and_tau}, we give some expressions that allow for a fast calculation of the respective measures.
Concerning Chatterjee's rank correlation, the expressions for the Clayton, Ali-Mikail-Haq, Nelsen7, Marshall-Olkin (see \cite[Example 4.2]{fuchs2021quantifying} for \(\alpha_1=1\)), and Cuadras-Augé families are up to our knowledge new, see Appendix \ref{subsec:computations_for_xi_rho_and_tau} for the calculations.

\section*{Conclusion}

In this paper, we have studied dependence properties of more than 35 well-known copula families with the focus on the Schur order for conditional distributions, which is a rearrangement-invariant dependence order that is consistent with Chatterjee's rank correlation.
In Section 3, we have provided a comprehensive overview of the copula families and their dependence properties.
Many of the considered copula families turn out to be Schur ordered either on the full parameter space or on a certain range of parameters, see Tables \ref{tab:arch_results} and \ref{tab:non_arch_results}.
Further, for some copula families, we have derived new closed-form expressions of Chatterjee's rank correlation.

\section*{Acknowledgements}
The first author gratefully acknowledges the support of the Austrian Science Fund (FWF) project
{P 36155-N} \emph{ReDim: Quantifying Dependence via Dimension Reduction}
and the support of the WISS 2025 project 'IDA-lab Salzburg' (20204-WISS/225/197-2019 and 20102-F1901166-KZP).
\appendix

\section{Appendix}
\label{sec:appendix}

In the sequel, we justify the properties in the Tables \ref{tab:arch_results}, \ref{tab:non_arch_results}, and \ref{tab:all_families_rho_and_tau} either by computation or by providing references.

\subsection{Computations for the Archimedean copula families in Table \ref{tab:arch_results}}
\label{subsec:archimedean}

The Gumbel-Hougaard copula family will be discussed in Subsection \ref{subsec:extreme_value}, as this family is not only Archimedean but also an extreme-value copula family.

\subsubsection{CI/CD and $\TP$}
\label{subsubsec:archimedean_ci_and_TP2}

For the CI/CD and $\TP$ properties, we repeatedly use a few classical observations.
First, recall that $\TP$ implies CI, see \eqref{implposdepcon}.
Further, since CI implies PLOD, the product copula $\Pi$ is at every point the smallest copula that is CI in the sense that for any other CI copula $C$ it holds $\Pi(u,v)\leq C(u,v)$ for all $(u,v)\in[0,1]^2$.
Likewise, it is the largest copula that is CD.
Consequently, it is clear that if the lower or upper tail dependence coefficient is strictly positive, then CD cannot hold.
$W$ is CD and $M$, noting that $M$ has no Lebesgue density, is CI but not $\TP$.
Lastly, $\frac{\Pi}{\Sigma - \Pi}$ is CI and $\TP$ as a special case of the Clayton copula with $\theta = 1$. \\
Recall from \eqref{lempdparch} that a bivariate Archimedean copula with inverse generator $\psi$ is CI/CD if and only if $-\psi'$ is log-convex/log-concave on the positive real line, and it is $\TP$ if and only if $\psi''$ is log-convex on the positive real line.
In a number of cases, the $\TP$ property follows from the generator of the copula family being completely monotone, see \cite[Theorem 2.14]{Mueller-2005}.
A list of copula families where this applies is given in \cite[Table 2]{hofert2008sampling} as well as in \cite[Exercise 4.24]{Nelsen-2006}.
These references however only verify $\TP$ for some parameter intervals.
If $\TP$ does not hold for all parameter choices, we show that their specified ranges in which $\TP$ holds cannot be extended.
In the following, we check one-by-one for which parameters the Archimedean copulas in Table \ref{tab:arch_overview} are CI/CD and $\TP$.

\begin{itemize}
    \item Clayton ($\theta \geq -1$):
    $C^{\text{Cl}}_0=\Pi$ is CD and for all $\theta < 0$, consider $0<y<-1/\theta$.
    Then, it is
    \[
        \psi'(y) ~=~ - \left(\theta y + 1\right)^{-1 - \frac1{\theta}},
        \quad (\log(-\psi'))'(y) ~=~ \frac{- \theta - 1}{\theta y + 1},
        \quad (\log(-\psi'))''(y) ~=~ \frac{\theta \left(\theta + 1\right)}{\left(\theta y + 1\right)^2}
    .\]
    As we assumed $\theta < 0$, $(\log(-\psi'))''$ is non-positive, so in this case the copula family is also CD.
    The $\TP$ property for $\theta \geq 0$ follows from \cite[Table 2]{hofert2008sampling}.
    \item Nelsen2 ($\theta \geq 1$):
    Let $\theta\geq 1$ be fixed and let $\varepsilon>0$ be sufficiently small.
    Then, it is
    \[ 
        C^{\text{N2}}_\theta\of{1 - \frac{1+\varepsilon}{2^{1/\theta}}, 1 - \frac{1+\varepsilon}{2^{1/\theta}}}
        = 0
    ,\]
    so that PLOD fails to hold for any choice of $\theta$.
    On the other hand, for $\theta>1$, the upper tail dependence coefficient is strictly positive, so that also CD fails to hold.
    \item Ali-Mikhail-Haq ($-1 \leq \theta \leq 1$):
    In the case $\theta\geq 0$, the $\TP$ property follows from \cite[Table 2]{hofert2008sampling}.
    It is immediate that PLOD does not hold as $C^{\text{AMH}}_1\nleq_{lo}\Pi$, so $C^{\text{AMH}}_1$ is not CD and for $\theta < 1$, it is
    \[
        \psi'(y) ~=~ \frac{\left(\theta - 1\right) e^{y}}{\left(\theta - e^{y}\right)^2},
        \quad (\log(-\psi'))'(y) ~=~ \frac{\theta + e^{y}}{\theta - e^{y}},
        \quad (\log(-\psi'))''(y) ~=~ \frac{2 \theta e^{y}}{\left(\theta - e^{y}\right)^2} 
    .\]
    $(\log(-\psi'))''$ is non-positive if and only if $\theta\leq 0$, so the Ali-Mikhail-Hak copula is CD if and only if $\theta \leq 0$.
    \item Frank ($\theta\in\R$):
    In the case $\theta\geq 0$, the $\TP$ property follows from \cite[Table 2]{hofert2008sampling}.
    For $\theta < 0$, note that
    \small
    \[
        \psi'(y) = \frac{1 - e^{\theta}}{\theta \left(1- e^{\theta} + e^{\theta + y} + 1\right)},
        ~~ (\log(-\psi'))'(y) = \frac{-e^{\theta + y}}{1 - e^{\theta} + e^{\theta + y}},
        ~~ (\log(-\psi'))''(y) = \frac{\left(e^{\theta} - 1\right) e^{\theta + y}}{\left(1 - e^{\theta} + e^{\theta + y} \right)^2}
    ,\]
    \normalsize
    and $(\log(-\psi'))''$ is non-positive if and only if $\theta < 0$.
    Since $C^{\text{Fr}}_{0}=\Pi$, this copula family is CD if and only if $\theta \leq 0$.
    \item Joe ($\theta \geq 1$):
    The generators of this copula family are completely monotone by \cite[Table 2]{hofert2008sampling}.
    The $\TP$ and CI properties hence follow for all $\theta$.
    \item Nelsen7 ($0 \leq \theta \leq 1$):
    Let $v\in[0,1]$ be given and observe that for $u\neq v$, it is
    \[
        \partial_2 C_{\delta}^{\text{N7}}(u,v)
        = \left(\theta v - \theta + 1\right) \1_{\setof{u \geq \frac{(\theta-1)(v-1)}{1+\theta(v-1)}}}
    ,\]
    which is trivially non-decreasing in $u$.
    Hence, by symmetry of the copula family, CD follows on the full parameter space.
    \item Nelsen8 ($1 \leq \theta$): 
    At $\theta=1$, it is $C^{\text{N8}}_1=W$, which is CD.
    For $u=v=t\in(0,1)$, it is
    \[
        C^{\text{N8}}_{\theta}(t,t) = \left(\frac{\theta^2 t^2-(1-t)^2}{\theta^2-(\theta-1)^2(1-t)^2}\right)_+
    .\]
    In the case of $\theta>1$, when $t>0$ is small enough, this expression become exactly zero.
    This shows that PLOD, and hence CI, fails not hold.
    Furthermore, it is
    \[
        \frac{C^{\text{N8}}_{\theta}(t,t)}{t^2}
        \geq \frac{\theta^2 - \frac1{t^2}(1-t)^2}{\theta^2-(\theta-1)^2(1-t)^2}
        >1
    \]
    when $\theta>2$ and $t$ is sufficiently close to $1$, showing that CD does not hold in this case.
    \item Gumbel-Barnett ($0 \leq \theta \leq 1$): $C_{0}=\Pi$ is $\TP$ and CD, and for all other $\theta$, it is
    \[
        \psi'(y) = - \frac{e^{y - \frac{e^{y} - 1}{\theta}}}{\theta},
        \quad (\log(-\psi'))'(y) = \frac{\theta - e^{y}}{\theta},
        \quad (\log(-\psi'))''(y) = - \frac{e^{y}}{\theta}
    .\]
    $(\log(-\psi'))''$ is non-positive, so this copula family is CD for all parameters.
    \item Nelsen10 ($0 \leq \theta \leq 1$): $C_{0}=\Pi$ is $\TP$ and CD, and for all other $\theta$, it is
    \small
    \[
        \psi'(y) = - \frac{2^{\frac1{\theta}} \left(e^{y} + 1\right)^{-1 - \frac1{\theta}} e^{y}}{\theta},
        \quad (\log(-\psi'))'(y) = \frac{\theta - e^{y}}{\theta \left(e^{y} + 1\right)},
        \quad (\log(-\psi'))''(y) = - \frac{\theta + 1}{4 \theta \cosh^2{\left(\frac{y}2 \right)}}
    .\]
    \normalsize
    $(\log(-\psi'))''$ is non-positive, so this copula is again CD for all parameters.
    \item Nelsen11 ($0 \leq \theta \leq 1/2$):
    $C_{0}=\Pi$ is $\TP$ and CD.
    For $\theta>0$ and $y\in(0,\log(2))$, it is
    \[
        \psi'(y) = - \frac{\left(2 - e^{y}\right)^{- \frac{\theta - 1}{\theta}} e^{y}}{\theta},
        \quad (\log(-\psi'))'(y) = \frac{- 2 \theta + e^{y}}{\theta \left(e^{y} - 2\right)},
        \quad (\log(-\psi'))''(y) = \frac{2 \left(\theta - 1\right) e^{y}}{\theta \left(e^{y} - 2\right)^2}
    .\]
    Since $\theta\leq 1/2$, $(\log(-\psi'))''$ is non-positive and hence the copula is CD.
    \item Nelsen12 ($1 \leq \theta$):
    The generators of this copula family are completely monotone by \cite[Table 2]{hofert2008sampling}.
    The $\TP$ and CI properties hence follow for all $\theta$.
    \item Nelsen13 ($0 < \theta$):
    We have
    \begin{align*}
        \psi'(y) ~=~ & - \frac{\left(y + 1\right)^{- \frac{\theta - 1}{\theta}} e^{1 - \left(y + 1\right)^{\frac1{\theta}}}}{\theta}, \allowdisplaybreaks\\
        (\log(-\psi'))'(y) ~=~ & \frac{- \theta - \left(y + 1\right)^{\frac1{\theta}} + 1}{\theta \left(y + 1\right)}, \allowdisplaybreaks\\
        (\log(-\psi'))''(y) ~=~ & \frac{\theta^2 + \theta \left(y + 1\right)^{\frac1{\theta}} - \theta - \left(y + 1\right)^{\frac1{\theta}}}{\theta^2 \left(y^2 + 2 y + 1\right)}
    .\end{align*}
    When $\theta < 1$, then $(\log(-\psi'))''$ is non-positive, so in this case the copula is CD in this case.
    Also, since $C^{\text{N13}}_{1} =\Pi $, it is CD at $\theta=1$.
    For $\theta\geq 1$ the $\TP$ and CI properties follow from the generators of this copula family being completely monotone by \cite[Table 2]{hofert2008sampling}.
    \item Nelsen14 ($1 \leq \theta$):
    The generators of this copula family are completely monotone by \cite[Table 2]{hofert2008sampling}.
    The $\TP$ and CI properties hence follow for all $\theta$.
    \item Genest-Ghoudi ($1 \leq \theta$):
    For $y\in(0,1)$, we have
    \begin{align*}
        \psi'(y) ~=~ & - y^{- \frac{\theta - 1}{\theta}} \left(1 - y^{\frac1{\theta}}\right)^{\theta - 1}, \\
        (\log(-\psi'))'(y) ~=~ & \frac{\theta - 1}{\theta y \left(y^{\frac1{\theta}} - 1\right)}, \\
        (\log(-\psi'))''(y) ~=~ & \frac{- \theta^2 y^{\frac1{\theta}} + \theta^2 - \theta + y^{\frac1{\theta}}}{\theta^2 y^2 \left(y^{\frac2{\theta}} - 2 y^{\frac1{\theta}} + 1\right)}
    .\end{align*}
    Substituting $z=y^{1/\theta}$, it is immediate that the denominator cannot become negative, and the numerator reads as $\theta^2(1-z) + z - \theta$, which for $\theta>1$ becomes negative as $y\rightarrow 1$, and positive as $y\rightarrow 0$.
    Consequently, this copula family is neither CI nor CD in this case.
    At $\theta=1$, it is $C^{\text{GG}}_1=W$, which is CD.
    \item Nelsen16 ($0 \leq \theta$):
    CD cannot hold for any choice of $\theta$ as the lower tail dependence coefficient is strictly positive.
    Further, for $C^{\text{N16}}_{0} = W$ and for $\theta > 0$, it is
    \begin{align}
        \psi'(y) ~=~ & \frac{1- \theta - y}{2 \sqrt{4 \theta + \left(\theta + y - 1\right)^{2}}} - \frac12 , \nonumber\allowdisplaybreaks\\
        (\log(-\psi'))'(y) ~=~ & \frac{4 \theta}{\left(4 \theta + \left(\theta + y - 1\right)^2\right) \left(\theta + y - \sqrt{4 \theta + \left(\theta + y - 1\right)^{2}} - 1\right)}, \label{frm:n16_ci_2} \\
        \psi''(y) ~=~ & \frac{2 \theta}{\left(4 \theta + \left(\theta + y - 1\right)^{2}\right)^{\frac{3}{2}}}, \nonumber\allowdisplaybreaks\\
        (\log(\psi''))'(y) ~=~ & \frac{3 \left(- \theta - y + 1\right)}{4 \theta + \left(\theta + y - 1\right)^{2}}, \nonumber \\
        (\log(\psi''))''(y) ~=~ & \frac{3 \left(\left(\theta + y - 1\right)^{2} - 4 \theta\right)}{\left(4 \theta + \left(\theta + y - 1\right)^{2}\right)^{2}} \label{frm:nelsen16_log3}
    .\end{align}
    \eqref{frm:nelsen16_log3} has roots for $y = 1- \theta \pm 2\sqrt{\theta}$, and the larger root will only be non-positive for $\theta \geq 3 + 2\sqrt{2}$, so that $\TP$ only holds under this condition.
    Regarding CI, let $f_\theta(y)$ denote the denominator of \eqref{frm:n16_ci_2}.
    Then, it is 
    \begin{align*}
        f'_\theta(y)
        = & \inbrackets{\sqrt{4 \theta + \left(\theta + y - 1\right)^{2}} - 2 \left(\theta + y - 1\right)} \left(\sqrt{4 \theta + \left(\theta + y - 1\right)^{2}} - (\theta + y - 1)\right)
    .\end{align*}
    Both factors of $f'_\theta(y)$ are non-increasing in $y$.
    Furthermore, it is $f'_0(0)>0$, $f'_\theta(0)=0$ if and only if $\theta=3$, and, e.g., $f'_5(0) < 0$.
    Consequently, \eqref{frm:n16_ci_2} is non-decreasing in $y$ if and only if $\theta\geq 3$, and thus CI holds in this range.
    \item Nelsen17 ($\theta \neq 0$): We have
    \begin{align}
        \psi'(y) ~=~ & \frac{\left(\frac{2^{\theta} e^{y}}{2^{\theta} e^{y} - 2^{\theta} + 1}\right)^{\frac{1}{\theta}} \left(1 - 2^{\theta}\right)}{\theta \left(2^{\theta} e^{y} - 2^{\theta} + 1\right)}, \nonumber\allowdisplaybreaks\\
        (\log(-\psi'))'(y) ~=~ & \frac{2^{- \theta} \left(2^{\theta} \left(1 - 2^{\theta}\right) - 4^{\theta} \theta e^{y}\right)}{\theta \left(2^{\theta} e^{y} - 2^{\theta} + 1\right)}, \nonumber\allowdisplaybreaks\\
        (\log(-\psi'))''(y) ~=~ & \frac{\left(  (2^{\theta} - 1 +  \theta \left(2^{\theta} - 1\right)\right) 2^{\theta}e^{y}}{\theta \left(2^{\theta} e^{y} - 2^{\theta} + 1\right)^{2}}, \label{frm:nelsen17log3}\\
        \psi''(y) ~=~ & \frac{\left(\frac{2^{\theta} e^{y}}{2^{\theta} e^{y} - 2^{\theta} + 1}\right)^{\frac{1}{\theta}} \left(2^{\theta} - 1\right) \left(2^{\theta} \theta e^{y} + 2^{\theta} - 1\right)}{\theta^{2} \left(2^{\theta} e^{y} - 2^{\theta} + 1\right)^{2}}, \nonumber\allowdisplaybreaks\\
        (\log(\psi''))'(y) ~=~ & \frac{2^{\theta} \theta^{2} \left(2^{\theta} e^{y} - 2^{\theta} + 1\right) e^{y} - 2^{\theta + 1} \theta (2^{\theta} \theta e^{y} + 2^{\theta} - 1) e^{y} + \left(1 - 2^{\theta}\right) \left(2^{\theta} \theta e^{y} + 2^{\theta} - 1\right)}{\theta \left(2^{\theta} e^{y} - 2^{\theta} + 1\right) \left(2^{\theta} \theta e^{y} + 2^{\theta} - 1\right)}, \nonumber\allowdisplaybreaks\\
        (\log(\psi''))''(y) ~=~ & 2^{\theta}e^{y}\left(\frac{2^{\theta} \theta^{2} \left(2^{\theta} e^{y} - 2^{\theta} + 1\right)^{2} + 2^{\theta} \left(2^{\theta} \theta e^{y} + 2^{\theta} - 1\right)^{2} + 2^{\theta + 1} \theta (2^{\theta} \theta e^{y} + 2^{\theta} - 1)^{2}}{\theta \left(2^{\theta} e^{y} - 2^{\theta} + 1\right)^{2} \left(2^{\theta} \theta e^{y} + 2^{\theta} - 1\right)^{2}} \right. \label{frm:nelsen17log4}, \\
        & \left.\frac{- \theta^{2} \left(2^{\theta} e^{y} - 2^{\theta} + 1\right)^{2} - 2 \theta \left(2^{\theta} \theta e^{y} + 2^{\theta} - 1\right)^{2} - \left(2^{\theta} \theta e^{y} + 2^{\theta} - 1\right)^{2}}{\theta \left(2^{\theta} e^{y} - 2^{\theta} + 1\right)^{2} \left(2^{\theta} \theta e^{y} + 2^{\theta} - 1\right)^{2}} \right)\nonumber.
    \end{align}
    At $\theta=-1$, it is $C^{N7}_{-1}=\Pi$, which is $\TP$, so consider $\theta\neq -1$.
    The numerator in \eqref{frm:nelsen17log3} is non-negative if and only if $\theta > 0$ or $\theta \leq -1$.
    Since the denominator flips signs if and only if $\theta < 0$, the whole expression is non-negative if and only if $\theta \geq -1$, and non-positive otherwise.
    Thus, this copula family is CI if and only if $\theta \geq -1$, and CD otherwise.
    In order to characterize the $\TP$ property, let's denote $a:=2^\theta e^y - 2^\theta + 1$ and $b:=2^\theta \theta e^y + 2^\theta - 1$.
    Then, the numerator in \eqref{frm:nelsen17log4} writes as 
    \[
        2^{\theta} a^{2} \theta^{2} + 2^{\theta} b^{2} + 2^{\theta + 1} b^{2} \theta - a^{2} \theta^{2} - 2 b^{2} \theta - b^{2}
        = \left(2^{\theta} - 1\right) \left(a^{2} \theta^{2} + (1+2 \theta)b^{2}\right)
    ,\]
    so for $\theta > 0$ numerator and denominator are clearly non-negative for any $y$, and thus the copula family is $\TP$ for $\theta > 0$.
    For $\theta \in (-1, 0)$, the denominator and $2^{\theta}-1$ are negative, so non-negativity of $(\log(\psi''))''$ is driven by $a^{2} \theta^{2} + (1+2 \theta)b^{2}$.
    This expression can be rewritten as 
    \[
        a^{2} \theta^{2} + (1+2 \theta)b^{2}
        = 2^{\theta + 1} (1+\theta)\theta\left(2^{\theta} \theta e^{y} + 2^{\theta} - 1\right) e^{y} + (\theta +1)^2 (2^\theta-1)^2
    ,\]
    and the right-hand side is indeed non-negative for $\theta \in (-1, 0)$, so the $\TP$-property follows also here.
    \item Nelsen18 ($2 \leq \theta$):
    The characterization for CI via \eqref{lempdparch} doesn't apply as the inverse generator is not differentiable at $y=e^{-\theta}$.
    Notice however that the copula is exactly zero on the set
    \[
        \setof{1+\frac{\theta}{\ln \left(e^{\theta /(u-1)}+e^{\theta /(v-1)}\right)} \leq 0}
        = \setof{v \leq 1 + \frac{\theta}{\ln\of{e^{-\theta} - e^{\theta/(u-1)}}}}
    .\]
    Note also that for $0<u<1$ it is
    \[
        1 + \frac{\theta}{\ln\of{e^{-\theta} - e^{\theta/(u-1)}}}
        > 1 + \frac{\theta}{\ln\of{e^{-\theta}}}
        = 0
    ,\]
    which shows that PLOD fails to hold.
    On the other hand, the upper tail dependence coefficient is strictly positive, so that also CD fails to hold. 
    \item Nelsen19 and Nelsen20 ($0 \leq \theta$):
    The generators of these copula families are completely monotone by \cite[Table 2]{hofert2008sampling}.
    The $\TP$ and CI properties hence follow for all $\theta$.
    \item Nelsen21 ($1 \leq \theta$): For $\varepsilon>0$, let
    \[
        u, v := 1 - \inbrackets{1-\varepsilon^\theta}^{\frac1{\theta}}
    .\]
    Then, for $\varepsilon$ sufficiently small, $0<u,v<1$ and $C^{\text{N21}}_\theta(u,v) = 0$, so that PLOD fails to hold for any choice of $\theta$.
    On the other hand, the upper tail dependence coefficient is strictly positive for $\theta>1$, so that also CD fails to hold in this case.
    \item Nelsen22 ($1 \leq \theta$): $C_{0}=\Pi$ is $\TP$ and CD.
    For all other $\theta$ and for $y\in(0,\pi/2)$, it is
    \begin{align*}
        \psi'(y) ~=~ & - \frac{\left(1 - \sin{\left(y \right)}\right)^{-1 + \frac1{\theta}} \cos{\left(y \right)}}{\theta}, \\
        (\log(-\psi'))'(y) ~=~ & \frac{\frac{\theta}{\cos{\left(y \right)}} - \tan{\left(y \right)} - \frac1{\cos{\left(y \right)}}}{\theta}, \\
        (\log(-\psi'))''(y) ~=~ & \frac{\theta \sin{\left(y \right)} - \sin{\left(y \right)} - 1}{\theta \cos^2{\left(y \right)}}
    .\end{align*}
    $(\log(-\psi'))''$ is non-positive as $\theta$ varies between $0$ and $1$, and $y$ can become at most $\pi/2$.
    Hence, this copula is CD for all parameters.
\end{itemize}
Note that all considered Archimedean copula families turn out to be CI if and only if they are $\TP$, see \cite[Remark 2.13]{Mueller-2005} for an Archimedean copula that is CI but not $\TP$.

\subsubsection{Lower orthant ordering and Schur ordering}
\label{subsubsec:archimedean_lo_order}
The lower orthant (or, equivalently, concordance) order properties of the Archimedean copulas can be found in \cite{Nelsen-2006}.
The families Ali-Mikhail-Haq, Frank, Joe, Nelsen7-8, Nelsen12-18, Nelsen20 and Nelsen21 are positively ordered by \cite[Exercise 4.18 (a)]{Nelsen-2006}, Clayton by \cite[Exercise 4.14]{Nelsen-2006}, Nelsen2 by \cite[Exercise 4.23]{Nelsen-2006}, and Nelsen19 by \cite[Exercise 4.15]{Nelsen-2006}.
Nelsen11 and Nelsen22 are negatively ordered by \cite[Exercise 4.18 (b)]{Nelsen-2006} and Gumbel-Barnett by \cite[Example 4.10]{Nelsen-2006}.
Nelson10 is unordered by \cite[Exercise 4.16]{Nelsen-2006}. \\
Recall from Lemma \ref{lem:3.16} that when a copula is CI, then lower orthant ordering and Schur order are equivalent.
Hence, the column on the Schur order in Table \ref{tab:arch_results} is a direct consequence of the columns on CI and the lower orthant order.
For those Archimedean copulas that are not CI/CD, we check the Schur ordering numerically by approximating the copulas with $40\times 40$ checkerboard copulas and rearranging the approximations with \cite[Algorithm1]{strothmann2022rearranged}.
The rearranged checkerboard copulas then need to be pointwise ordered in $\theta$ for the Schur order to hold, see Proposition \ref{proprearcop}.

\subsubsection{Tail dependencies}
\label{subsubsec:archimedean_tails}
Concerning the Archimedean copula families, the formulas for the lower and upper tail dependence coefficients $\lambda_L$ and $\lambda_U$ are given in \cite[Example 5.22]{Nelsen-2006}.

\subsection{Computations for the extreme-value copula families in Table \ref{tab:non_arch_results}}
\label{subsec:extreme_value}

For the extreme-value copulas in Table \ref{tab:all_families_rho_and_tau}, the associated Pickands dependence functions are given in Table \ref{tab:non_arch_overview}.
Note that the Marshall-Olkin extreme-value copula family yields the \textit{Cuadras-Augé} copula family in the special case $\alpha_1 = \alpha_2$.

\subsubsection{CI/CD and $\TP$}
\label{subsubsec:extreme_value_ci_and_TP2}

Extreme-value copulas are always CI, see \cite[Th\'{e}or\`{e}me 1]{Guillem-2000}.
The Marshall-Olkin copula is $\TP$ if and only if $\alpha_1 \wedge \alpha_2 = 0$, see \cite[Example 3.6]{fuchs2023total}.
In particular, the Cuadras-Augé copulas are $\TP$ only for $\delta=0$.
The Gumbel-Hougaard copula family is Archimedean with a completely monotone generator, see \cite[Table 2]{hofert2008sampling}.
The $\TP$ property hence follows for all $\theta$.
Our numerical checks for the log-supermodularity of the copula's density function on a $40\times40$ grid indicate that the Hüsler-Reiss copula family is also $\TP$, but not the Joe-EV, Marshall-Olkin, Tawn and t-EV copula families.

\subsubsection{Lower orthant ordering and Schur ordering}
\label{subsubsec:extreme_value_lo_ordering}

Due to Theorem \ref{thm:ordering_ev_copulas}, we only need to check whether Pickands dependence functions are on the interval \((0,1)\) pointwise decreasing/increasing in the parameter to obtain that the associated extreme-value copulas are increasing/decreasing with respect to the lower orthant order and Schur order for copula derivatives.

\begin{itemize}
    \item Galambos ($0<\delta$), Joe ($0\leq\alpha_1,\alpha_2\leq 1,~0<\delta$), BB5 ($1\leq\theta,~0<\delta$) and Tawn ($0\leq\alpha_1,\alpha_2\leq 1,~1\leq\theta$): Notice that Galambos Pickands function satisfies
    \[
        1 - \inbrackets{t^{-\delta} + (1-t)^{-\delta}}^{-1/\delta}
        = 1 - \frac{1}{\lVert (1/t, 1/(1-t))\rVert_{\delta}}
    ,\]
    where \(\lVert \cdot \rVert_p\) denotes the \(p\)-norm on \(\R\,.\)
    Since $p$-norms are decreasing in \(p\), the Galambos Pickands function is non-increasing in its parameter, and it follows that the copula family is positively ordered.
    In the same way, the Joe extreme-value copula family is positively ordered for fixed $\alpha_1$ and $\alpha_2$, and the BB5 copula family is for fixed $\theta$.
    Also, we can write Tawn's Pickands function as 
    \small
    \[
        (1-\alpha_1)(1-t) + (1-\alpha_2)t + \inbrackets{(\alpha_1(1-t))^{\theta} + (\alpha_2 t)^\theta}^{1/\theta}
        = (1-\alpha_1)(1-t) + (1-\alpha_2)t + \lVert (\alpha_1(1-t), \alpha_2 t)\rVert_{\theta}
    ,\]
    \normalsize
    and hence the same argument works here, showing that the Tawn copula family is positively ordered.
    \item Gumbel-Hougaard ($1\leq \theta$): This copula family is lower orthant ordered, see, e.g., \cite[Example 4.12]{Nelsen-2006}
    \item Hüsler-Reiss ($0\leq\delta$): The lower orthant ordering is mentioned in \cite[Chapter 5]{Joe-1997}.
    \item Marshall-Olkin ($0\leq\alpha_1,\alpha_2\leq1$) and Cuadras-Augé ($0\leq\delta\leq1$): The lower orthant ordering for Cuadras-Augé is shown in, e.g., \cite[Example 2.19]{Nelsen-2006}, which corresponds to the special case of $\alpha_1=\alpha_2$ for the Marshall-Olkin copula family.
    \item t-EV ($-1<\rho< 1,~0<\nu$): The t-EV Pickands function is symmetric in $t$ and decreases trivially for $t=1/2$ in $\rho$, so we may assume without loss of generality $t>1/2$.
    First, notice that 
    \[
        \frac{d}{d\rho} z_t
        = \frac{d}{d\rho} \sqrt{\frac{1+\nu}{1-\rho^2}}\inbrackets{\of{\frac{t}{1-t}}^{1/\nu}-\rho}
        = \frac{\rho \sqrt{\nu + 1} \left(- \rho + \left(\frac{t}{1 - t}\right)^{\frac{1}{\nu}}\right)}{\left(1 - \rho^{2}\right)^{1.5}} - \frac{\sqrt{\nu + 1}}{\sqrt{1-\rho^2}}
    .\]
    Secondly, the probability density function of the Student-t distribution with $\nu+1$ degrees of freedom is given by 
    \[
        T'_{\nu + 1}(x)
        = \frac{\Gamma\left(\frac{\nu}{2} + 1\right)}{\sqrt{(\nu + 1) \pi} \Gamma\left(\frac{\nu + 1}{2}\right)}\left(1+\frac{x^2}{\nu + 1}\right)^{-\frac{\nu}{2} - 1}
    .\]
    Together, we get for the Pickands function that
    \small
    \begin{align}
        & \frac{d}{d\rho} A_{\nu, \rho}(t) \nonumber\allowdisplaybreaks\\
        = ~ & \frac{d}{d\rho}(1-t)T_{\nu + 1}(z_{1-t}) + t T_{\nu + 1}(z_t) \nonumber\allowdisplaybreaks\\
        = ~ & (1-t)T'_{\nu + 1}(z_{1-t})\frac{d}{d\rho} z_{1-t} + tT'_{\nu + 1}(z_t)\frac{d}{d\rho} z_t \nonumber\allowdisplaybreaks\\
        = ~ & \frac{(1-t)\Gamma\left(\frac{\nu}{2} + 1\right)}{\sqrt{(\nu + 1) \pi} \Gamma\left(\frac{\nu + 1}{2}\right)}\left(1+\frac{\frac{1+\nu}{1-\rho^2}\left(\left(\frac{1-t}{t}\right)^{1/\nu}-\rho\right)^2}{\nu + 1}\right)^{-\frac{\nu}{2} - 1}\left(\frac{\rho \sqrt{\nu + 1} \left(\left(\frac{1-t}{t}\right)^{\frac{1}{\nu}}- \rho\right)}{\left(1 - \rho^{2}\right)^{1.5}} - \frac{\sqrt{\nu + 1}}{\sqrt{1-\rho^2}}\right) \nonumber\allowdisplaybreaks\\
        & + \frac{t\Gamma(\frac{\nu}{2} + 1)}{\sqrt{(\nu + 1) \pi} \Gamma(\frac{\nu + 1}{2})}\left(1+\frac{\frac{1+\nu}{1-\rho^2}\left(\left(\frac{t}{1-t}\right)^{1/\nu}-\rho\right)^2}{\nu + 1}\right)^{-\frac{\nu}{2} - 1}\left(\frac{\rho \sqrt{\nu + 1} \left(\left(\frac{t}{1 - t}\right)^{\frac{1}{\nu}}- \rho\right)}{\left(1 - \rho^{2}\right)^{1.5}} - \frac{\sqrt{\nu + 1}}{\sqrt{1-\rho^2}}\right) \nonumber\allowdisplaybreaks\\
        = ~ & \frac{(1-t)\Gamma(\frac{\nu}{2} + 1)}{\sqrt{\pi} \Gamma(\frac{\nu + 1}{2})\left(1-\rho^2\right)^{(\nu+5)/2}}\left(1 - \rho^2 + \left(\rho - \left(\frac{1-t}{t}\right)^{\frac{1}{\nu}}\right)^{2}\right)^{-\frac{\nu}{2} - 1}\left(\rho \left(\frac{1-t}{t}\right)^{\frac{1}{\nu}}- 1\right) \label{frm:first_line} \\
        & + \frac{t\Gamma(\frac{\nu}{2} + 1)}{\sqrt{\pi} \Gamma(\frac{\nu + 1}{2})\left(1-\rho^2\right)^{(\nu+5)/2}}\left(1 - \rho^2 + \left(\rho - \left(\frac{t}{1-t}\right)^{\frac{1}{\nu}}\right)^{2}\right)^{-\frac{\nu}{2} - 1}\left(\rho \left(\frac{t}{1 - t}\right)^{\frac{1}{\nu}}- 1\right)\,.\label{frm:second_line}
    \end{align}
    \normalsize
    \eqref{frm:first_line} is certainly non-positive as we assumed $t>1/2$.
    \eqref{frm:second_line} can only become positive when $\rho > \left(\frac{1-t}{t}\right)^{\frac{1}{\nu}}$.
    In this case, we get with $a:= (t/(1-t))^{1/\nu}\in(1,\infty)$ that
    \begin{align*}
        \eqref{frm:second_line}/\abs{\eqref{frm:first_line}}
        \quad = \quad & \frac{t}{1-t}\left(\frac{1 - \rho^2 + \left(\rho - \left(\frac{1-t}{t}\right)^{\frac{1}{\nu}}\right)^{2}}{1 - \rho^2 + \left(\rho - \left(\frac{t}{1-t}\right)^{\frac{1}{\nu}}\right)^{2}}\right)^{\frac{\nu}{2} + 1}\frac{\rho \left(\frac{t}{1 - t}\right)^{\frac{1}{\nu}}- 1}{1-\rho \left(\frac{1-t}{t}\right)^{\frac{1}{\nu}}} \allowdisplaybreaks\\
        = \quad & a^{\nu}\left(\frac{1 - \rho^2 + \left(\rho - a^{-1}\right)^{2}}{1 - \rho^2 + \left(\rho - a\right)^{2}}\right)^{\frac{\nu}{2} + 1}\frac{\rho a- 1}{1-\rho a^{-1}} \allowdisplaybreaks\\
        = \quad & \left(a^2\frac{1 - \rho^2 + \left(\rho - a^{-1}\right)^{2}}{1 - \rho^2 + \left(\rho - a\right)^{2}}\right)^{\frac{\nu}{2} + 1}\frac{\rho a- 1}{a(a-\rho)} \allowdisplaybreaks\\
        = \quad & \frac{\rho a- 1}{a(a-\rho)}
        \quad = \quad 1 - \frac{1+a^2}{\rho a}
        \quad < \quad 1
    .\end{align*}
    Consequently, $\frac{d}{d\rho} A_{\nu, \rho}(t) \leq 0$, and it follows that the t-EV extreme-value copula is positively ordered.
\end{itemize}

\subsubsection{Tail dependencies}
\label{subsubsec:extreme_value_tail_dependencies}

The formulas for the upper tail dependence coefficients are taken from \cite[Table 3.2]{eschenburg2013properties}. \\
Lower tail-dependence coefficients are generally trivially given by $\lambda_L = \1_{\setof{A(1/2)=1/2}}$ with the intuition that lower tail dependence exists for extreme-value copulas when there is complete dependence, compare \cite[Section 6.4]{jaworski2010copula}.
For the Marshall-Olkin extreme-value copula, $A(1/2)=1/2$ holds if and only if $\alpha_1=\alpha_2=1$.
In particular, the Cuadras-Augé copula has lower tail dependence if and only if $\delta=1$.
For the Hüsler-Reiss copula family, this requires $\Phi(z_{1/2}) = 1/2 \Leftrightarrow z_{1/2} = 0$, which holds only in the limiting case $\delta = \infty$.
Furthermore, for the t-EV extreme-value copula, it must also be $z_{1/2} = 0$, but again this only holds in the limit, here as $\rho\rightarrow 1$.
For the Tawn extreme-value copula, note that
\[
    A(1/2)
    = 0.5\inbrackets{2 - \alpha_1 - \alpha_2 + \inbrackets{\alpha_1^\theta + \alpha_2^\theta}^{1/\theta}}
    \geq 0.5\inbrackets{2 - \alpha_2}
,\]
and the inequality is strict when $\alpha_2>0$.
The last expression however is louwer bounded by $1/2$ with eqaulity only for $\alpha_2 = 1$, so that $A(1/2)>1/2$ for all parameter choices.
Likewise, for Galambos Pickands function, it is $A(1/2) = 2^{(1 - \delta)/\delta} > 1/2$ for any choice of $\delta$.
The Joe Pickands function is always at least as large as a corresponding Galambos Pickands function, so also in this case $A(1/2) = 1/2$ cannot hold.
For the BB5 Pickands function, it is $A(1/2) = \frac12(2 - 2^{-1/\delta}) > 1/2$, so also here the lower tail dependence coefficient is zero.
Lastly, the tail dependence parameters for the Gumbel-Hougaard copula family are evaluated in, e.g., \cite[Example 4.12]{Nelsen-2006}

\subsection{Computations for the elliptical copula families in Table \ref{tab:non_arch_results}}
\label{subsec:elliptical_copula_families}

\subsubsection{CI/CD and $\TP$}
\label{subsubsec:elliptical_ci_and_TP2}

The Gauss copula is $\TP$ if and only if it is CI if and only if $\rho\geq 0$, see \cite[Theorem 2 and 3]{Rueschendorf-1981b} or \cite[Theorem 3.6]{Mueller-Scarsini-2001}.
The Student-t copula family is not CI, see \cite[Proposition 4.3]{rossell2021dependence}, and in particular not $\TP$.
Lastly, the Laplace distribution is not $\TP$, see \cite[Theorem 4.9 and below]{rossell2021dependence}.
Further, for \(\rho\leq 0\,,\) the Laplace family is not CI, see \cite[Remark 4.1 and Theorem 4.2]{rossell2021dependence}.

\subsubsection{Lower orthant ordering and Schur ordering}
\label{subsubsec:elliptical_lo_ordering}

If the radial variable admits a Lebesgue density, then the copulas associated with a family of elliptical distributions are uniquely determined and by Proposition \ref{prop:ordering_elliptical_copulas}\,(i) \(\leq_{lo}\)-increasing in \(\rho\,.\)
The Gaussian copula family is increasing in \(\rho\) with respect to the Schur order if \(\rho\geq 0\) as a consequence of the CI property.
When all other variables are fixed for an elliptical distribution, positive lower orthant ordering always holds with respect to the correlation parameter $\rho$.
This is a consequence of \cite[Theorem 5.1]{gupta1971inequalities}.
Hence, elliptical copulas are \(\leq_{lo}\)-increasing in the parameter $\rho$ by Proposition \ref{prop:ordering_elliptical_copulas}(i).
Schur order results follow for the Gauss copula when $\rho\geq 0$ as in this case CI holds.

\subsubsection{Tail dependencies}
\label{subsubsec:elliptical_tail_dependencies}
When $\rho\in\setof{-1, 1}$, the upper and lower tail dependence coefficients of the Gauss copula are trivially $1$.
In all other cases, they vanish, see for example \cite[Corollary 1]{furman2016tail}.
For Student-t distributions, it is 
\[
    \lambda_L
    = \lambda_U
    = 2-2 t_{\nu+1}\of{\sqrt{\nu+1} \frac{\sqrt{1-\rho}}{\sqrt{1+\rho}}}
,\]
see \cite[Section 5.3]{embrechts2001modelling}.

\subsection{Computations for the unclassified copula families in Table \ref{tab:non_arch_results}}
\label{subsec:unclassified_copula_families}

Here, we discuss the Fréchet, Mardia, Farlie-Gumbel-Morgenstern, Plackett and Raftery copula families, which are important examples for copula families that don't fit into the elliptical, Archmimedean, or extreme-value case.

\subsubsection{CI/CD and $\TP$}\label{subsubsec:unclassified_ci_and_TP2}

The CI/CD and $\TP$ columns for the unclassified copula families in Table \ref{tab:non_arch_results} are justified by the following references and computations:
\begin{itemize}
    \item Fréchet ($0\leq \alpha,\beta\leq 1$, $\alpha+\beta\geq 1$): CI holds if and only if $\beta = 0$, and it is $\TP$ if and only if $\alpha=1$ and $\beta=0$, see \cite[Example 3.5]{fuchs2023total}.
    Similarly, CD holds if and only if $\alpha = 0$.
    \item Mardia ($-1\leq \theta\leq 1$): CI and $\TP$ hold if and only if $\theta =1$ as the Mardia copula family is a special case of the Fréchet copula family with the correspondence
    \begin{align}\label{frm:mardia_frechet_correspondence}
        \alpha = \frac{\theta^2(\theta+1)}{2}, \quad \beta=\frac{\theta^2(\theta-1)}{2}
    .\end{align}
    Likewise, CD holds if and only if $\theta = -1$.
    \item Farlie-Gumbel-Morgenstern ($-1\leq \theta\leq 1$):
    The CI property holds if and only if $\theta\geq 0$, as 
    \[
        \frac{\partial}{\partial v} C(u, v)
        = u + \theta u (1-u)(1-2v)
    \]
    is clearly non-increasing in $v$ for any choice of $u$ if and only if $\theta\geq 0$.
    Likewise, this copula family is CD if and only if $\theta\leq 0$.
    \cite[Theorem 3]{amblard2002symmetry} applied for the concave function $\psi(x) = x(1-x)$ yields the $\TP$ property for $\theta > 0$.
    For $\theta = 0$ we obtain the product copula $\Pi$, which is also $\TP$.
    \item Plackett ($0 < \theta$):
    The Plackett copula is CI if and only if $\theta \geq 1$, and it is CD if and only if $\theta\leq 1$, compare \cite[Example 5.16]{Nelsen-2006}.
    Regarding $\TP$, note that if $c^{\text{Pl}}_\theta$ denotes the density of $C^{\text{Pl}}_\theta$, it then is
    \begin{align*}
        \log c^{\text{Pl}}_\theta(u,v)
        \quad = \quad & \phantom{}\log{\left(\theta \right)} + \log{\left(2 (1-\theta) u v + (\theta-1) (u + v) + 1 \right)} \\
        & - \frac32 \log{\left(- 4 \theta u v \left(\theta - 1\right) + \left(\left(\theta - 1\right) \left(u + v\right) + 1\right)^{2} \right)}
    \end{align*}
    on $(0,1)^2$.
    The $\TP$-property states that $\log c$ is 2-increasing on $(0,1)^2$.
    Since $\log c$ is twice differentiable and by the Topkis Characterization Theorem (compare \cite{milgrom1990rationalizability}), this holds if and only if $\frac{\partial^2\log c}{\partial u \partial v} \geq 0$ on $(0,1)^2$.
    The latter evaluates to
    \small
    \begin{equation}
        \label{frm:TP2_for_plackett_explicit}
        \frac{\partial^2}{\partial u \partial v}\log c(u,v)
        = \frac{f(u,v, \theta)}{\left((\theta-1) (u + v - 2uv) + 1\right)^{2} \left(- 4 \theta u v \left(\theta - 1\right) + \left(\left(\theta - 1\right) \left(u + v\right) + 1\right)^{2}\right)^{2}} 
    ,\end{equation}
    \normalsize
    with 
    \small
    \begin{align*}
        f(u,v, \theta) ~:=~ & (\theta-1)(2 v - 1) \left(- 4 \theta u v \left(\theta - 1\right) + \left(\left(\theta - 1\right) \left(u + v\right) + 1\right)^{2}\right) \\
        & + 2 (\theta-1)(2 u - 1) \left(- 4 \theta u v \left(\theta - 1\right) + \left(\left(\theta - 1\right) \left(u + v\right) + 1\right)^{2}\right) \\
        & + 6 \left(- \theta^{2} u + (\theta-1)^{2} v + \theta + u - 1\right) \left((\theta-1) (u + v - 2uv) + 1\right) \\
        & \left((\theta-1)^{2} u - \theta^{2} v + \theta + v - 1\right) \left((\theta-1) (u + v - 2uv) + 1\right) \\
        & + \left[
            3 u^2 (\theta - 1)^2 + 2 u (\theta - 1) ((v - 2) \theta + v + 1) - v (\theta - 1) (2 v (\theta - 1) - \theta + 3) - \theta - 1 \right] \\
        &(-2) ( \theta - 1)\left((\theta-1) (u + v - 2uv) + 1\right)
        \left(- 4 \theta u v \left(\theta - 1\right) + \left(\left(\theta - 1\right) \left(u + v\right) + 1\right)^{2}\right) \\
        & - (\theta-1)(2 u - 1)\left(- 4 \theta u v \left(\theta - 1\right) + \left(\left(\theta - 1\right) \left(u + v\right) + 1\right)^{2}\right)  \\
        & - 3 \left(- \theta^{2} u + (\theta-1)^{2} v + \theta + u - 1\right) \left((\theta-1) (u + v - 2uv) + 1\right)
    .\end{align*}
    \normalsize
    The denominator in \eqref{frm:TP2_for_plackett_explicit} is always positive, so we can focus on the positiveness of $f$.
    At $(u, v) = (0, 1)$, the expression for $f$ simplifies to $f(u,v, \theta) = -4 + 12/\theta - 8/\theta^2$, which is negative for $\theta > 2$.
    By continuity, if we let $v<1$ sufficiently close to $1$ and $u>0$ sufficiently close to $0$, the expression will also be negative for $\theta > 2$, and such a pair qualifies for contradicting the $\TP$ property.
    Hence, the $\TP$ property fails to hold for $\theta \in (2, \infty)$.
    For $\theta\in[1, 2]$, our numerical checks indicate that the copula family is $\TP$.
    \item Raftery ($0\leq \delta\leq 1$):
    A density does not exist for any choice of $\delta$, so the $\TP$ property fails to hold.
    CI certainly holds for $C_{0}^{\text{Ra}}(u,v)= \Pi$ and $C_{1}^{\text{Ra}}(u,v)=M$.
    For all other $\delta$, note that due to symmetry of the Raftery copulas, it is sufficient to check CIS.
    For that, let $(U,V)\sim C_{\delta}^{\text{Ra}}$, and observe that for $v>0$ and $u\neq v$, it is
    \begin{align}
        \partial_2 C_{\delta}^{\text{Ra}}(u,v)
        \quad = \quad & \frac{\left(u v\right)^{\frac{1}{1 - \delta}} \1_{\setof{u\leq v}}}{\max\left(u, v\right)^{\frac{2}{1 - \delta}}} + \1_{\setof{u \geq v}} + \frac{u^{\frac{1}{1 - \delta}} v^{\frac{\delta}{1 - \delta}} \left(1 - \max\left(u, v\right)^{\frac{- \delta - 1}{1 - \delta}}\right)}{\delta + 1} \nonumber\allowdisplaybreaks\\
        = \quad & \begin{cases}
            \frac1{1+\delta}u^{\frac{1}{1 - \delta}}\inbrackets{\delta v^{-\frac{1}{1 - \delta}} + v^{\frac{\delta}{1 - \delta}}} & \text{ if } u < v \\
            1 + \frac{u^{\frac{1}{1-\delta}} - u^{-\frac{\delta}{1-\delta}}}{\delta + 1}v^{\frac{\delta}{1-\delta}}& \text{ if } u > v
        \end{cases} \label{frm:raf}
    .\end{align}
    The first expression in \eqref{frm:raf} is non-increasing in $v$, because
    \[
        \frac{\partial}{\partial v} \inbrackets{\delta v^{-\frac{1}{1 - \delta}} + v^{\frac{\delta}{1 - \delta}}}
        = \frac{\delta}{(1-\delta)v} \inbrackets{v^{\frac{\delta}{1 - \delta}}- v^{-\frac{1}{1 - \delta} }}
    \]
    is negative for $v\in(0,1)$ and $\delta\in(0,1)$.
    The second expression is also non-increasing in $v$, because $u^{\frac{1}{1-\delta}} \leq u^{-\frac{\delta}{1-\delta}}$.
    Also note that both, the first expression of \eqref{frm:raf} as $u\nearrow v$ and the second expression of \eqref{frm:raf} as $u\searrow v$, converge to $\frac{\delta}{1+\delta} + \frac{u^{\frac{1+\delta}{1-\delta}}}{1+\delta}$, which shows that $P(U\leq u|V=v)$ is non-increasing in $v$ on the interval $(0, 1]$.
    Since $u\in[0, 1]$ was arbitrary, the CI property follows.
\end{itemize}

\subsubsection{Lower orthant ordering and Schur ordering}\label{subsubsec:unclassified_lo_ordering}

The Fréchet copula family is trivially lower orthant increasing when fixing either $\alpha$ or $\beta$.
The Farlie-Gumbel-Morgenstern, Plackett and Raftery copula families are all lower orthant increasing, see \cite[Exercise 3.22]{Nelsen-2006}, \cite[Exercise 3.37]{Nelsen-2006} and \cite[Chapter 5.1]{Joe-1997}, respectively.
Lastly, the Mardia copula family is not lower orthant ordered, see \cite[Example 2.8]{pfeifer2004modeling}.

\subsubsection{Tail dependencies}\label{subsubsec:unclassified_tail_dependencies}

The tail dependence coefficients for the Fréchet copula family are found in, e.g., \cite[Exercise 2.4]{Nelsen-2006}.
From that, one obtains the tail dependence coefficients for the Mardia copula family via \eqref{frm:mardia_frechet_correspondence}.
The tail dependence coefficients for the Plackett and Raftery copula families are found in, e.g., \cite[Exercise 5.21]{Nelsen-2006}.
Lastly, the tail dependence coefficients for the Farlie-Gumbel-Morgenstern copula family directly evaluate to 
\[
    \lambda_L
    = \lim_{t\rightarrow 0} \frac{C(t,t)}{t}
    = \lim_{t\rightarrow 0} t + \theta t (1-t)^2
    = 0
,\]
and likewise $\lambda_U = 0$ for all $\theta$.

\subsection{Computations for Chatterjee's xi, Spearman's rho and Kendall's tau in Table \ref{tab:all_families_rho_and_tau}}
\label{subsec:computations_for_xi_rho_and_tau}

In the sequel, we justify all entries in Table \ref{tab:all_families_rho_and_tau}, either by reference or by computation.
For the computations, we leverage the integral formulas given in  \eqref{frm_rho_integral}, \eqref{frm_tau_integral}, and \eqref{frm_xi_integral} above.

\subsubsection{Archimedean copulas}
\label{subsubsec:archimedean_spearman_and_kendall}

In this subsection, we cover dependence measures for a number of Archimedean copula families, namely the Clayton, Ali-Mikhail-Haq, Gumbel-Hougaard, Frank, Nelsen7 and Gumbel-Barnett copula families.

\begin{itemize}
    \item Clayton ($0<\theta$): 
    The formula for Spearman's rho can be found in \cite[Example 5.4]{Nelsen-2006}.
    Regarding Chatterjee's xi, notice that the first partial derivative is 
    \[
        \partial_1 C^{\text{Cl}}_\theta(u, v)
        = v^{\theta + 1} \left(- u^{\theta} v^{\theta} + u^{\theta} + v^{\theta}\right)^{- \frac{\theta + 1}{\theta}}
    ,\]
    which satisfies
    \small
    \[
        \int(\partial_1 C^{\text{Cl}}_\theta(u, v))^2\de u
        = u v^{2} \left(u^{\theta} \left(v^{- \theta}-1\right) + 1\right)^{\frac{2}{\theta}} \left(v^{\theta} - u^{\theta} \left(v^{\theta} - 1\right) \right)^{- \frac{2}{\theta}} {{}_{2}F_{1}\left(\begin{matrix} 2 + \frac{2}{\theta}, \frac{1}{\theta} \\ 1 + \frac{1}{\theta} \end{matrix}\middle| {u^{\theta} v^{- \theta} (v^{\theta} - 1)} \right)}
    .\]
    \normalsize
    From this, one obtains
    \small
    \begin{align*}
        \xi\of{C^{\text{Cl}}_{\theta}}
        = 6 \int_0^1\int_0^1 v^{2 \theta + 2} \left(- u^{\theta} v^{\theta} + u^{\theta} + v^{\theta}\right)^{-2 - \frac{2}{\theta}} \de u \de v - 2
        = 6 \int\limits_{0}^{1} {{}_{2}F_{1}\left(\begin{matrix} \frac{1}{\theta}, 2 + \frac{2}{\theta} \\ 1 + \frac{1}{\theta} \end{matrix}\middle| {1 - v^{- \theta}} \right)}\, \de v - 2
    ,\end{align*}
    \normalsize
    where ${}_{2}F_{1}$ is the hypergeometric function given by
    \begin{align}\label{frm:hypergeometric_function}
        { }_2 F_1(a, b ; c, z)
        :=\frac{\Gamma(c)}{\Gamma(a) \Gamma(c-a)} \int_0^1 v^{a-1}(1-v)^{c-a-1}(1-v z)^{-b} d v
    ,\end{align}
    where $\Gamma$ is the gamma function.
    \item Ali-Mikhail-Haq ($-1\leq \theta < 1$):
    The formulas for Spearman's rho and Kendall's tau are given in \cite[Exercise 5.10]{Nelsen-2006}.
    For Chatterjee's xi, observe that
    \[
        \partial_1 C^{\text{AMH}}_{\theta}(u,v) = \frac{v \left(\theta u \left(v - 1\right) - \theta \left(u - 1\right) \left(v - 1\right) + 1\right)}{\left(\theta \left(u - 1\right) \left(v - 1\right) - 1\right)^{2}}
    ,\]
    from which we get for $\theta\neq 0$
    \begin{align*}
        \xi\of{C^{\text{AMH}}_{\theta}}
        \quad = \quad & 6 \int_0^1\int_0^1 \inbrackets{\frac{v \left(\theta u \left(v - 1\right) - \theta \left(u - 1\right) \left(v - 1\right) + 1\right)}{\left(\theta \left(u - 1\right) \left(v - 1\right) - 1\right)^{2}}}^2 \de u \de v - 2 \\
        \quad = \quad & 6 \int_0^1 \frac{v^{2} \left(\frac{\theta^{2} v^{2}}{3} - \frac{2 \theta^{2} v}{3} + \frac{\theta^{2}}{3} + \theta v - \theta + 1\right)}{\theta v - \theta + 1}
        \de v - 2\\
        \quad = \quad & 6\left(\frac{\theta^{3} \left(8 - \theta\right) + 18 \theta^{2} - 12 \theta - 12 \left(\theta - 1\right)^{2} \log{\left(1 - \theta \right)}}{36 \theta^{3}}
        \right) - 2 \\
        = \quad & - \frac{\theta}{6} - \frac{2}{3} + \frac{3}{\theta} - \frac{2}{\theta^{2}} - \frac{2 \left(\theta - 1\right)^{2} \log{\left(1 - \theta \right)}}{\theta^{3}}
    .\end{align*}
    As $\theta\rightarrow 0$, the computed formula converges to $\xi\of{C^{\text{AMH}}_{0}} = 0$.
    \item Frank ($\theta\in\R$): The formulas for Spearman's rho and Kendall's tau are given in \cite[Exercise 5.9]{Nelsen-2006}.
    \item Nelsen7 ($0\leq\theta\leq 1$): The cases $\theta=0$ and $\theta=1$ are immediate.
    For $\theta\in(0, 1)$, we have
    \begin{align*}
        \rho_S\of{C^{\text{N7}}_{\theta}}
        \quad = \quad & 12 \int_0^1\int_0^1 \left(\theta u v - \left(\theta - 1\right) \left(u + v - 1\right)\right)_+ \de u \de v - 3 \\
        = \quad & 12 \int_0^1 \frac{v^{2}}{2 \left(\theta v - \theta + 1\right)} \de v - 3 \\
        = \quad & 12 \frac{3 \theta^{2} - 2 \theta - 2 \left(\theta - 1\right)^{2} \log{\left(1 - \theta \right)}}{4 \theta^{3}} - 3 \\
        = \quad & -3 + \frac{9}{\theta} - \frac{6}{\theta^{2}} - \frac{6 \left(\theta - 1\right)^{2} \log{\left(1 - \theta \right)}}{\theta^{3}}
    \end{align*}
    and
    \[
        \tau\of{C^{\text{N7}}_{\theta}}
        = 1 + 4\int_0^1 \frac{\left(t \theta - \theta + 1\right) \log{\left(t \theta - \theta + 1 \right)}}{\theta} \de t
        = 2 - \frac{2}{\theta} - \frac{2 \left(\theta - 1\right)^{2} \log{\left(1 - \theta \right)}}{\theta^{2}}
    .\]
    In the limiting cases, $\rho_S\of{C^{\text{N7}}_{0}} = \tau\of{C^{\text{N7}}_{0}} = -1$ and $\rho_S\of{C^{\text{N7}}_{1}} = \tau\of{C^{\text{N7}}_{1}} = 0$.
    The graph of Chatterjee's rank correlation for the Nelsen7 family obtained by simulations shows perfect negative linear relationship dependence between the rank correlation and the copula parameter $\theta$.
    Indeed, the theoretical rank correlation evaluates to
    \begin{align*}
        \xi\of{C^{\text{N7}}_{\theta}}
        \quad = \quad & 6 \int_0^1\int_0^1 (\theta v + 1 - \theta)^2(\theta u v + (1-\theta)(u+v-1))_+ \de u \de v - 2 \allowdisplaybreaks\\
        = \quad & 6 \int_0^1 (\theta v + 1 - \theta)^2\inbrackets{1 - \frac{(1-\theta)(1-v)}{1-\theta + v}} \de v - 2 \allowdisplaybreaks\\
        = \quad & 6 \int_0^1 \theta v^2 + v - \theta v \de v - 2 \\
        = \quad & 1 - \theta
    .\end{align*}
    \item Gumbel-Barnett ($0\leq \theta\leq 1$): The formula for Chatterjee's xi is given in \cite[Example 1]{dette2013copula}.
\end{itemize}

\subsubsection{Extreme-value copula families}
\label{subsubsec:ev_spearman_and_kendall}

The formulas for Spearman's rho and Kendall's tau for the Gumbell-Hougaard copula family are found in \cite[Example 4.2]{hurlimann2004properties} and \cite[Example 5.4]{Nelsen-2006}.
For the Marshall-Olkin copula family, the formulas for Spearman's rho and Kendall's tau can be found, e.g., in \cite[Example 5.7 a), Example 5.9 c)]{Nelsen-2006}.
Chatterjee's rank correlation $\xi$ with $\alpha_1 \neq 1/2$ evaluates to
\begin{align*}
    \xi\of{C^{\text{MO}}_{\delta}}
    \quad = \quad & 6 \int_0^1\int_0^1 \frac{\left(u v^{1 - \alpha_{2}} \1_{\setof{u v^{1 - \alpha_{2}} \leq u^{1 - \alpha_{1}} v}} - u^{1 - \alpha_{1}} v \left(\alpha_{1} - 1\right) \1_{\setof{u v^{1 - \alpha_{2}} \geq u^{1 - \alpha_{1}} v}}\right)^{2}}{u^{2}} \de u \de v - 2 \allowdisplaybreaks\\
    = \quad & 6\int_0^1\int_0^{v^{\frac{\alpha_2}{\alpha_1}}} \frac{\left(u v^{1 - \alpha_{2}}\right)^{2}}{u^{2}} \de u \de v
    + \int_0^1\int_{v^{\frac{\alpha_2}{\alpha_1}}}^1 \frac{\left(u^{1 - \alpha_{1}} v \left(\alpha_{1} - 1\right)\right)^{2}}{u^{2}} \de u \de v - 2 \allowdisplaybreaks\\
    = \quad & 6 \int\limits_{0}^{1}  v^{- 2 \alpha_{2} + 2 + \frac{\alpha_{2}}{\alpha_{1}}} \de v
    + 6 \int\limits_{0}^{1}\frac{v^2\left(\alpha_{1} - 1\right)^{2} \left(v^{\frac{\alpha_{2}}{\alpha_{1}}-2\alpha_{2}}-1 \right)}{2 \alpha_{1} - 1}
    \de v - 2 \\
    = \quad & 6 \int_0^1 \frac{\alpha_{1}^{2} v^{2+\frac{\alpha_{2}}{\alpha_{1}}- 2 \alpha_{2}} - v^2(\alpha_1-1)^2}{2 \alpha_{1} - 1}
    \de v - 2 \\
    = \quad & \frac{2 \alpha_{1}^{2} \alpha_{2}}{3 \alpha_{1} + \alpha_{2} - 2 \alpha_{1} \alpha_{2}}
\end{align*}
and for $\alpha_1=1/2$ to
\begin{align*}
    \xi\of{C^{\text{MO}}_{\delta}}
    \quad = \quad & 6 \int_0^1\int_0^1 \frac{\left(\frac12 \sqrt{u} v \1_{\setof{\sqrt{u} v \leq u v^{1 - \alpha_{2}}}} + u v^{1 - \alpha_{2}} \1_{\setof{\sqrt{u} v \geq u v^{1 - \alpha_{2}}}}\right)^{2}}{u^{2}} \de u \de v - 2 \\
    = \quad & 6\int_0^1\int_0^{v^{\frac{\alpha_2}{2}}} v^{2 - 2 \alpha_{2}} \de u \de v
    + \frac32\int_0^1\int_{v^{\frac{\alpha_2}{2}}}^1 \frac{v^{2}}{u} \de u \de v - 2 \allowdisplaybreaks\\
    = \quad & 6 \int\limits_{0}^{1}  v^2 \de v
    - \frac32 \int\limits_{0}^{1} v^{2} \log{\left(v^{2\alpha_{2}} \right)}
    \de v - 2 \\
    = \quad & \frac{\alpha_2}3
.\end{align*}
Hence,
\[
    \xi\of{C^{\text{MO}}_{\alpha_1, \alpha_2}} = \frac{2 \alpha_{1}^{2} \alpha_{2}}{3 \alpha_{1} + \alpha_{2} - 2 \alpha_{1} \alpha_{2}}
\] 
for all $\alpha_1,\alpha_2\in[0, 1]$.
This generalizes the formula $\xi\of{C^{\text{MO}}_{1, \alpha_2}} = \frac{2 \alpha_{2}}{3 - \alpha_{2}}$ given in \cite[Example 4]{fuchs2021quantifying}.
Furthermore, letting $\alpha_1=\alpha_2$, we also obtain the closed-form formula
\[
    \xi\of{C^{\text{CA}}_{\delta}} = \frac{\delta^2}{2-\delta}
\]
for the Cuadras-Augé copula family as a special case.

\subsubsection{Elliptical copula families}
\label{subsubsec:ellipticals_spearman_rho_and_kendalls_tau}

A general formula for Kendall's tau of elliptical copulas is given in \cite[Theorem 5.4]{embrechts2001modelling}.
The formula for Spearman's rho for the Gaussian copula family can be found in \cite[Theorem 5.36]{Embrechts-2015} or \cite{heinen2020spearman} and for the Student-t and Laplace copula families in \cite[Proposition 1]{heinen2020spearman}.
Note that \cite[Proposition 4 and Remark 2]{heinen2020spearman} also give longer, more explicit formulas for Spearman's rho of the Student-t and the Laplace copula families.
The formula for Chatterjee's xi in the Gassian case is given e.g. in \cite[Example 4]{fuchs2021quantifying}.

\subsubsection{Unclassified copula families}
\label{subsubsec:unclassified_spearman_rho_and_kendalls_tau}

The in Table \ref{tab:all_families_rho_and_tau} stated formulas for Spearman's rho and Kendall's tau of unclassified copula families are found in \cite[Example 5.2 and 5.7 a)]{Nelsen-2006} for the Farlie-Gumbel-Morgenstern, in \cite[Example 5.3 and 5.6]{Nelsen-2006} for the Fréchet (and the Mardia), in \cite[Exercise 5.8]{Nelsen-2006} for the Plackett, and in \cite[Excercise 5.11]{Nelsen-2006} for the Raftery copula family. \\
The in Table \ref{tab:all_families_rho_and_tau} stated formulas for Chatterjee's xi of unclassified copula families are found in \cite[Example 4]{fuchs2021quantifying} for the Farlie-Gumbel-Morgenstern and the Fréchet coupla family, and from the latter one directly obtains the formula for the Mardia copula family.

\end{document}